\documentclass[11pt]{article}
\usepackage[a4paper]{geometry}

\usepackage{amsmath}
\usepackage{amssymb}
\usepackage{amsthm}
\usepackage{mathrsfs}
\usepackage{bbm}
\usepackage{empheq}
\usepackage[loose]{subfigure}
\usepackage{epsfig}
\usepackage{graphicx}
\usepackage{psfrag}
\usepackage[usenames,dvipsnames]{pstricks}
\usepackage{pst-plot}
\usepackage[colorlinks]{hyperref}
\usepackage{tabls}
\usepackage{paralist}
\usepackage{hyperref}

\DeclareSymbolFontAlphabet{\Bbb}{AMSb}

\newlength{\fixboxwidth}
\newcommand{\argmin}{\mathop{\mathrm{arg\,min}}}
\newcommand{\COMMENT}[1]{}
\newcommand{\E}{\mathbb{E}}

\newcommand{\Dmaps}{\mathfrak{D}}
\newcommand{\Dmap}{\mathbb{D}}
\newcommand{\Drv}{D}
\newcommand{\Dspace}{\mathcal{D}}
\newcommand{\Ddata}{d}
\DeclareMathOperator{\ext}{ext}

\newcommand{\eins}{\mathbbm{1}}

\newcommand{\one}{\mathbbm{1}}
\renewcommand{\P}{\mathbb{P}}

\newcommand{\R}{\mathbb{R}}

\newcommand{\quark}{\setbox0\hbox{$x$}\hbox to\wd0{\hss$\cdot$\hss}}
\newcommand{\smid}{\,\middle|\,}

\DeclareMathOperator{\supp}{supp}

\newtheorem{thm}{Theorem}[section]
\newtheorem{prop}[thm]{Proposition}
\newtheorem{lem}[thm]{Lemma}

\theoremstyle{definition}
\newtheorem{defn}[thm]{Definition}
\newtheorem{rmk}[thm]{Remark}
\newtheorem{eg}[thm]{Example}
\newtheorem{pb}{Problem}
\newtheorem{assumption}{Assumption}

\def \d         { \delta }
\hypersetup{
    linkcolor=blue,
}

\title{Brittleness of Bayesian Inference\\ Under Finite Information in a Continuous  World\footnotetext{\noindent	2010 Mathematics Subject Classification:
	62A01, 
	62E20, 
	62F12, 
	62F15, 
	62G20, 
	62G35. 
 \\Keywords:  Bayesian inference, misspecification, robustness, uncertainty quantification, optimal uncertainty quantification.\\
Houman Owhadi: Corresponding author, California Institute of Technology, owhadi@caltech.edu\\Clint Scovel: California Institute of Technology,  clintscovel@gmail.com\\ Tim Sullivan: Mathematics Institute, University of Warwick, Coventry, CV4 7AL, UK.  Tim.Sullivan@warwick.ac.uk
}
}

\author{Houman Owhadi, Clint Scovel, Tim Sullivan
}

\date{\today}

\makeatletter
\@addtoreset{equation}{section}
\@addtoreset{figure}{section}
\@addtoreset{table}{section}
\makeatother

\renewcommand{\thefigure}{\arabic{section}.\arabic{figure}}

\makeatletter
\renewcommand{\p@subfigure}{\thefigure}
\makeatother

\newcounter{mycount}

\begin{document}

\maketitle

\begin{abstract}
We derive, in the classical framework of Bayesian sensitivity analysis,
optimal lower and upper bounds on posterior values obtained from Bayesian models that exactly capture an arbitrarily large number of finite-dimensional marginals of the data-generating distribution and/or that are as close as desired to the data-generating distribution in the Prokhorov or total variation metrics; these bounds show that such models may still make the largest possible prediction error after conditioning on an arbitrarily large number of sample data measured at finite precision.
These results are obtained through the development of a reduction calculus for optimization problems over measures on  spaces of measures.
We use this calculus to investigate the mechanisms that generate brittleness/robustness
and, in particular, we
 observe that learning and robustness are antagonistic properties.
 It is now well understood that the numerical resolution of PDEs requires the satisfaction of specific stability
conditions. Is there a missing stability condition for using Bayesian inference
in a continuous world under finite information?
\end{abstract}

\tableofcontents

\section{Introduction}
\label{sec:Intro}
With the advent of high-performance computing, Bayesian methods are increasingly popular tools for the quantification of uncertainty throughout science and industry. Since these methods impact the making of sometimes critical decisions in increasingly complicated contexts, the sensitivity of their posterior conclusions with respect to the underlying models and prior beliefs is becoming a pressing question.

While it is known that Bayesian methods are robust and consistent when the number of possible outcomes is finite,  the exploration of Bayesian inference
in a continuous world has revealed both positive \cite{Bernstein:1964, CastilloNickl:2013, Doob:1949, KleijnVaart:2012, LeCam:1953, Stuart:2010, vonMises:1964}
and negative results \cite{Belot:2013, Belot:2013b, DiaconisFreedman:1986, Freedman:1999, Freedman:1963, Johnstone:2010, Leahu:2011}. One contribution of this paper is the development of a calculus for the elucidation of the mechanisms generating robustness or brittleness in Bayesian inference.
In particular, this paper
\begin{enumerate}
	\item shows that the process of Bayesian conditioning on data at fine
enough resolution is sensitive (as defined in \cite{TibshiraniWasserman:1988}, modulo a small technicality) with respect to the underlying distributions, under the total variation and Prokhorov metrics;
and
	\item raises the question of a missing stability condition for using Bayesian inference in a continuous world under finite information, somewhat akin to the CFL condition for the stability of a discrete numerical scheme used to approximate a continuous PDE.
\end{enumerate}
Point (1) is the source of negative results similar to those caused by tail properties in statistics \cite{BahadurSavage:1956, Donoho:1988}, and can be seen as an extreme occurrence of the dilation phenomenon from robust Bayesian inference \cite{WassermanSeidenfeld:1994}.

Let us now illustrate the main question explored in this paper with a simple example of Bayesian reasoning in action:

\begin{pb}
	\label{pb:1}
	There is a bag containing 102  coins, one of which always lands on heads, while the other 101 are perfectly fair.  One coin is picked uniformly at random from the bag, flipped 10 times, and 10 heads are obtained.  What is the probability that this coin is the unfair coin?
\end{pb}

The correct probability is given by applying Bayes' theorem:
\begin{equation}
	\label{eq:Bayes}
	\P[A|B]=\P[B|A]\frac{\P[A]}{\P[B]}=\frac{1}{1+101 \times 2^{-10}} \approx 0.91,
\end{equation}
where $A$ is the event ``the coin is the unfair coin'' and $B$ is the event ``10 heads are observed''.
If the number of coins is not known exactly and the supposedly fair coins are not exactly fair, then Bayes' theorem can be used to  produce a robust Bayesian inference in the following sense:  if
  the fair coins are slightly unbalanced and the probability of getting a tail is  $0.51$,
and an estimate of 100 coins is used and an estimate $\frac{1}{2}$ of the fairness of the fair coins is used,
 then the resulting estimate $\frac{1}{1+99 \times 2^{-10}}$ is still a good approximation of the correct answer.

Does this robustness hold when the underlying probability space is continuous or an approximation thereof?  For example, what if the random outcomes are decimal numbers --- perhaps given to finite precision --- rather than heads or tails?

\subsection{The General Question}

To investigate these questions in a general context let us now consider the situation in which the space  $\mathcal{X}$ where observations/samples take their values is no longer $\{\text{Head},\text{Tail}\}$ but an arbitrary Polish space (with the real line $\R$ as a prototypical example). Write $\mathcal{M}(\mathcal{X})$ for
the set of probability measures on $\mathcal{X}$ and let $\Phi \colon \mathcal{M}(\mathcal{X})\to \R$ be a function\footnote{All spaces will be topological spaces, the term ``function'' will mean Borel measurable function and ``measure'' will mean Borel measure.} defining a \emph{quantity of interest}.  When $\mathcal{X}$ is the real line $\R$,  a prototypical example is $\Phi(\mu):=\mu[X\geq a]$, the probability that the random variable $X$ distributed according to $\mu$ exceeds the threshold value $a$;  another typical example is $\Phi(\mu):=\E_{\mu}[X]$, the mean of $X$.

\begin{pb}
	\label{pb:2}
Let the \emph{data-generating distribution}  $\mu^\dagger \in \mathcal{M}(\mathcal{X})$ be an unknown or partially known probability measure on $\mathcal{X}$.  The objective is to estimate $\Phi(\mu^\dagger)$ from the observation of $n$ i.i.d.~samples from $\mu^\dagger$, which we denote by $d=(d_1,\ldots,d_n)\in \mathcal{X}^n$.
\end{pb}
For practical reasons (and to avoid problems associated with conditioning with respect to events of measure zero) we will assume that the data is observed up to resolution/precision $\delta>0$, i.e.\ what
we actually observe in Problem \ref{pb:2}
is the event $d\in B^n_\delta$, where $B^n_\delta:=\prod_{i=1}^n B_\delta(x_i)$, $(x_1,\ldots,x_n)$ is a fixed point of $\mathcal{X}^n$, and $B_\delta(x)$ is the open ball of radius $\delta$ and center $x$ (defined with respect to a consistent metric on the Polish space $\mathcal{X}$).

Now observe that the Bayesian answer to Problem  \ref{pb:2} is to assume that $\mu^\dagger$ is the realization of some random measure $\mu$ on $\mathcal{M}(\mathcal{X})$. This is done by choosing a \emph{model class} $\mathcal{A}\subseteq\mathcal{M}(\mathcal{X})$
 and a probability measure $\pi \in \mathcal{M}(\mathcal{A})$ which we call \emph{the prior}. This prior
determines the randomness with which a representative $\mu \in \mathcal{A}$ is selected, and for each such
$\mu \in \mathcal{A}$, the generation of $n$ i.i.d.~samples $d \in \mathcal{X}^n$ by randomly sampling from
$\mu^{n}$ naturally determines a product measure on $\mathcal{A} \times \mathcal{X}^n$.
In analogy to Problem \ref{pb:1},
$\mathcal{A}$ plays the role of the bag of coins (measures) and  each measure $\mu \in \mathcal{A}$ plays the role of a coin.

Now the  prior estimate of the quantity of interest is $\E_{\mu\sim \pi} [\Phi(\mu)]$ and the posterior estimate is defined as the conditional expectation \begin{equation}\label{eq:posterior}
\E_{\mu \sim \pi, d\sim \mu^{n}} [\Phi(\mu)| d\in B^n_\delta]
\end{equation}
 with respect to this product measure.

One response to the concern that the choice of prior $\pi$ is somewhat arbitrary is to explore classes of priors.
Indeed:
\begin{quotation}
	\noindent ``Most statisticians would acknowledge that an analysis is not complete unless the sensitivity of the conclusions to the assumptions is investigated.  Yet, in practice, such sensitivity analyses are rarely used.  This is because sensitivity analyses involve difficult computations that must often be tailored to the specific problem.  This is especially true in Bayesian inference where the computations are already quite difficult.'' \cite{WassermanEtAl:1993}
\end{quotation}
In this paper we will
investigate this approach, known as \emph{robust Bayesian inference} \cite{Berger:1984, Berger:1994, Box:1953, Wasserman:1990}  or \emph{Bayesian sensitivity analysis},
and examine the robustness of Bayesian inference by computing optimal bounds on prior and posterior values in terms of given sets of priors.  To do so, we need some definitions.

\begin{defn}
	\label{defn:robustness}
	For a  model class $\mathcal{A}\subseteq \mathcal{M}(\mathcal{X})$, a quantity of interest $\Phi \colon \mathcal{A} \to \R$,  and a set of priors $\Pi \subseteq \mathcal{M}(\mathcal{A})$, let
	\begin{align*}
		\mathcal{L}(\Pi)&:=\inf_{\pi \in \Pi}\E_{\mu\sim \pi}\bigl[\Phi(\mu)\bigr]\\
  		\mathcal{U}(\Pi)&:=\sup_{\pi \in \Pi}\E_{\mu\sim \pi}\bigl[\Phi(\mu)\bigr]
	\end{align*}
	denote the optimal lower and upper bounds on the prior values of the quantity of interest.  For $B \subseteq \mathcal{X}^{n}$ a non-empty open subset of the data space, let $\Pi_B\subseteq \Pi$ be the subset of priors $\pi$ such that the probability that $d\in B$ is nonzero, i.e.~$\mathbb{P}_{\mu \sim \pi, d\sim \mu^{n}}[d \in B]>0$, and let
	\begin{align*}
		\mathcal{L}(\Pi|B)&:=\inf_{\pi\in \Pi_B }\E_{\mu \sim \pi, d\sim \mu^{n}} \bigl[\Phi(\mu)\big| d\in B\bigr]\\
		\mathcal{U}(\Pi|B)&:=\sup_{\pi\in \Pi_B }\E_{\mu \sim \pi, d\sim \mu^{n}} \bigl[\Phi(\mu)\big| d\in B\bigr]
	\end{align*}
	denote the optimal lower and upper bounds on the posterior values of the quantity of interest, given that $d\in B$.
\end{defn}

\subsection{Example of Brittleness Under Finite Information}
\label{subsecex1}

As
illustrated in Problem \ref{pb:1}, it is already known from classical Bayesian sensitivity analysis that posterior values are robust
if the random outcomes live in a finite space (i.e.\ $\mathcal{X}$
is finite) or if the class of priors  $\Pi$ is finite-dimensional (i.e.\ if
what one does not know can be represented by a finite number of known parameters). One purpose of this paper is to investigate what the very same classical Bayesian sensitivity analysis framework would conclude in the presence of finite information (i.e.~if for instance $\Pi$ is finite codimensional).
To understand this question let us  consider
the following example:

\begin{eg}
Our purpose is to estimate the mean $\Phi(\mu^\dagger):=\E_{\mu^\dagger}[X]$ of some random variable $X$ with respect to some unknown distribution $\mu^\dagger$ on the interval $[0,1]$ based on the observation of $n$ i.i.d.~samples $d:=(d_1,\ldots,d_n)$,
given to finite resolution $\delta$ (i.e.\ we
observe $d\in B^n_\delta$, where $B^n_\delta$ is the product of $n$ open balls of radius $\delta$), from the unknown distribution $\mu^\dagger$.
\end{eg}

The Bayesian answer to that problem is to assume that $\mu^\dagger$ is the realization of some random measure distributed according to some prior $\pi$ (i.e.~$\mu\sim \pi$)  and then compute the posterior value of the mean by conditioning on the data, i.e.\ compute
\eqref{eq:posterior} with
$\Phi(\mu):=\E_{\mu}[X]$.  Observe that to specify the prior $\pi$ we need to specify the distribution of all the moments\footnote{In fact, this is a necessary but not a sufficient condition to determine $\pi$, since there are cases in which the moment problem is indeterminate.  See \cite{Akhiezer:1965} for a full discussion of such issues.}
of $\mu$ (i.e.~the distribution of the infinite-dimensional vector $(\E_{\mu}[X], \E_{\mu}[X^2], \E_{\mu}[X^3],\ldots)$).

It is known, from classical robust Bayesian inference, that the posterior value \eqref{eq:posterior} is robust with respect  to finite dimensional perturbations of the particular choice of the prior $\pi$. However,  rather than specifying a finite-dimensional class of priors $\Pi$ (i.e.\ assuming
infinite information), it appears epistemologically more reasonable to specify a finite-\emph{codimensional}
$\Pi$ (i.e.\ assume
finite information) and a natural way to do so is to specify the distribution $\mathbb{Q}$  of only a large, but finite, number of moments of $\mu$  (i.e.~to specify the distribution of $(\E_{\mu}[X], \E_{\mu}[X^2], \ldots, \E_{\mu}[X^k])$, where $k \in \mathbb{N}$
can be arbitrarily large). This defines a class of priors $\Pi$ on $\mathcal{M}([0,1])$ such that if $\pi\in \Pi$ and $\mu\sim \pi$ then
\[
(\E_{\mu}[X], \E_{\mu}[X^2], \ldots, \E_{\mu}[X^k]) \sim \mathbb{Q} .
\]
More precisely, writing $\Psi$ as the function mapping each measure $\mu$ on $[0,1]$ to its first $k$ moments $\Psi(\mu):=(\E_{\mu}[X], \E_{\mu}[X^2], \ldots, \E_{\mu}[X^k])$ and choosing a measure $\mathbb{Q}$ on $\Psi\big(\mathcal{M}([0,1])\big) \subset \R^k$, $\Pi$ is simply defined as the pullback of the measure $\mathbb{Q}$ under $\Psi$, i.e.\ writing $\mathcal{A}:=\mathcal{M}([0,1])$,
\[
\Pi:=\Psi^{-1} \mathbb{Q}=\big\{\pi \in \mathcal{M}(\mathcal{A}) \mid \Psi \pi=\mathbb{Q}\big\}.
\]
One consequence of one of the main results of this paper, Theorem \ref{thm:shiva}, is that no matter how large $k$ is, no matter how large the number of samples $n$ is, for any $\mathbb{Q}$ that has a density with respect to the uniform distribution on the first $k$ moments, if you observe the data at a fine enough resolution, then the minimum and maximum of the posterior value of the mean over the class of priors $\Pi$ are $0$ and $1$, i.e.\ the
following proposition holds.

\begin{prop}\label{eq:prop1}
For all $k \in \mathbb{N}$,
if $\mathbb{Q}$ is absolutely continuous with respect to the uniform distribution on $\Psi\big(\mathcal{M}([0,1])\big)$, then
\[
 \lim_{\delta \downarrow 0}\mathcal{L}(\Pi|B^n_\delta)=0 \text{ and } \lim_{\delta \downarrow 0}\mathcal{U}(\Pi|B^n_\delta)=1
\]
and the convergence holds uniformly in $n$.
\end{prop}

This example of brittleness is derived from
Theorem \ref{thm:shiva} (see Example \ref{eg:shiva3}), the proof of which sheds light on
the mechanism leading to brittleness in a general context and shows that
the pathology illustrated by Proposition \eqref{eq:prop1} is general and inherent to using Bayesian inference in continuous spaces (or their discretizations)  under finite information. Furthermore, although this simple example concerns the posterior mean,
 the quantity of interest in Theorem \ref{thm:shiva}  is arbitrary and the brittleness results apply to the whole posterior distribution.

\subsection{Example of Brittleness Under Infinitesimal Model Perturbations}
\label{subsecex2}

Theorem \ref{thm:shiva} (and its corollary, Theorem \ref{thm:shiva0cor}), which leads to brittleness under finite information as illustrated in the previous example, also leads to brittleness under infinitesimal model perturbations in the total variation and Prokhorov metrics. We will now illustrate one mechanism causing brittleness with a simple example.

In this example we are interested in estimating $\Phi(\mu^\dagger)=\E_{\mu^\dagger}[X]$ where $\mu^\dagger$ is an unknown distribution on the unit interval ($\mathcal{X}=[0,1]$)  based on the observation of a single data point $d_1=0.5$ up to resolution $\delta$ (i.e.\ we
observe $d_1\in B_\delta(x_1)$ with $x_1=0.5$).

Consider the following two Bayesian models (measures) $\mu^a(\theta)$ and $\mu^b(\theta)$ on the unit interval $[0,1]$, parametrized by $\theta \in (0,1)$, and with densities $f^{a}$ and $f^{b}$ given by
\[
	f^a(x,\theta)= (1-\theta) \bigl( 1+ \tfrac{1}{\theta} \bigr) (1-x)^\frac{1}{\theta} + \theta \bigl( 1 + \tfrac{1}{1-\theta} \bigr) x^\frac{1}{1-\theta},
\]
\[
	f^b(x,\theta)= \begin{cases}
	f^a(x,\theta)
	 \frac{1}{Z}\big(\one_{\{x \not\in (x_1-\frac{\delta_c}{2},x_1+\frac{\delta_c}{2})\}}+ 10^{-9} \one_{\{x \in (x_1-\frac{\delta_c}{2},x_1+\frac{\delta_c}{2})\}}\big),
	&\text{ if }\theta< 0.999,\\
	f^a(x,\theta), & \text{ if }\theta\geq 0.999,
	\end{cases}
\]
where $Z$ is a normalization constant (close to one) chosen so that $\int_{[0,1]} f^b(x,\theta)\,dx=1$.  See Figure \ref{fig:faandfb} for an illustration of these densities.

\begin{figure}[tp]
	\begin{center}
		\subfigure[$f^a(x,\theta)$]{
			\includegraphics[height=2.5cm]{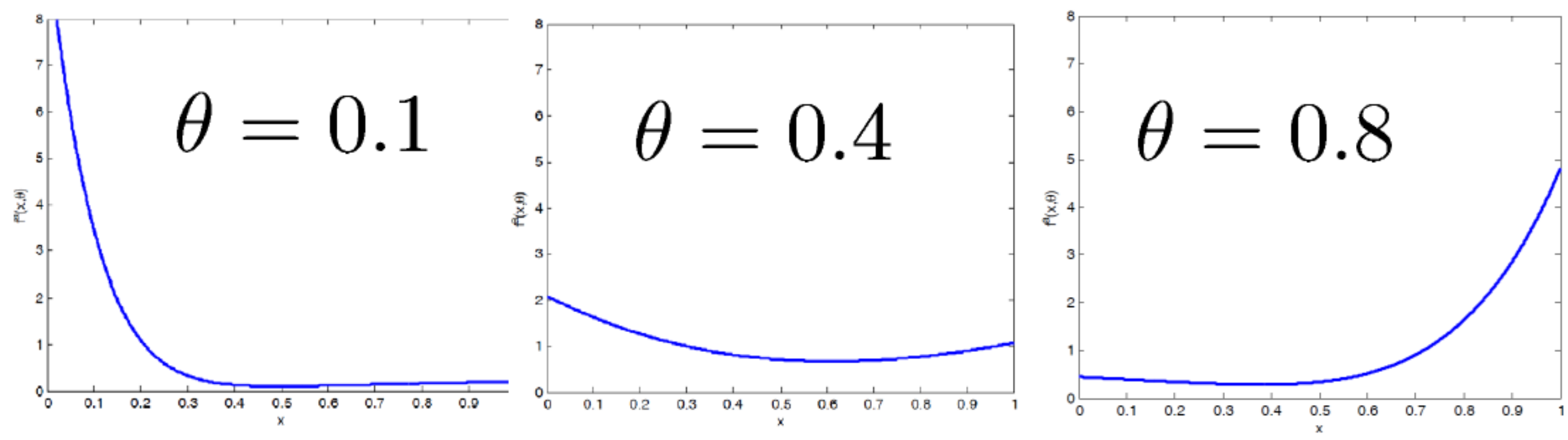}
		}
		\subfigure[$f^b(x,\theta)$]{
			\includegraphics[height=2.4cm]{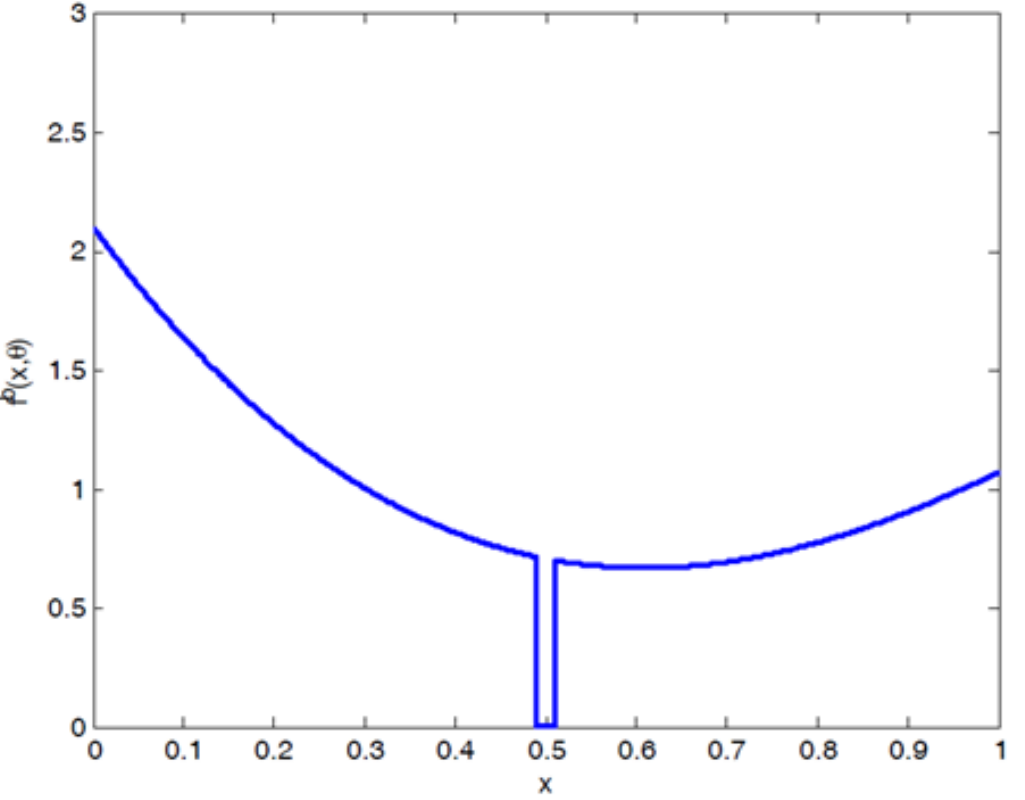}
		}
	\end{center}
	\caption{Illustration of the density $f^a(x,\theta)$ of model $a$ and $f^b(x,\theta)$ of model $b$.}
	\label{fig:faandfb}
\end{figure}

Observe that the density of model \emph{b} is that of model \emph{a} besides the small gap of width $\delta_c > 0$ created around the data point for model \emph{b} (if $\theta< 0.999$, see Figure \ref{fig:faandfb});
since the data point is fixed at $x_{1} = \tfrac{1}{2}$, the total variation distance $d_{\textup{TV}}\big(\mu^a(\theta),\mu^b(\theta)\big)$ between the two models is, uniformly over $\theta \in (0,1)$, a constant times $\delta_c$.
Assuming that the prior distribution on $\theta$ is the uniform distribution on $(0,1)$, observe that the prior value of the quantity of interest $\E_{\mu}[X]$ under both models (\emph{a} and \emph{b}) is approximately $\frac{1}{2}$.
Now, when $\theta$ is close to one (zero) then the density of model \emph{a} puts most of its mass towards one (zero). Observe also that the density of model \emph{b} behaves in a similar way, with the important exception that the probability of observing the data under model \emph{b} is infinitesimally small for $\theta<0.999$.  Therefore,  for $\delta <\delta_c$, the posterior value of  the quantity of interest $\E_{\mu}[X]$ under model \emph{a} is $\frac{1}{2}$ whereas it is close to one under model \emph{b}.
Observe also that a  perturbed model \emph{c}  analogous to  \emph{b} would lead to a posterior value close to zero.

This simple example  of brittleness under infinitesimal model perturbations is derived from the proof of
Theorem \ref{thm_localshivacor}, which shows that Bayesian posterior values are generally brittle under infinitesimal perturbations of Bayesian models in TV and in Prokhorov metrics.

$\mu^b(\theta)$ is  also a simple example of what worst priors can look like after a classical Bayesian sensitivity analysis over a class of priors specified via constraints on the TV or Prokhorov distance or the distribution of a finite number of moments.

Can we dismiss these worst priors because they depend on the data? The problem with this argument is that in the context of Bayesian sensitivity analysis worst priors always depend on (or are pre-adapted to) the data. Therefore the same argument would lead to  a dismissal of Bayesian sensitivity analysis and therefore the robust Bayesian framework. Can we dismiss these worst priors because they depend \emph{too much} on the data? The problem with this argument is that it is not a transparent task to define \emph{too much}
without introducing  the following element of circular reasoning: \emph{the degree of pre-adaptation determines the degree of brittleness, the framework is dismissed is when the degree of pre-adaptation is ``too much'', therefore the method cannot be brittle}.

Can we dismiss these worst priors because they can ``look nasty'' and make the probability of observing the data very small? The problem with this  argument is that these worst priors are not ``isolated pathologies'' but directions of instability and their number increase with the number of data points.
We will illustrate this point with another simple example by placing a uniform constraint on the probability of observing the data in the model class.
We already know that if the data is equally likely under all measures in the model class then posterior values are robust but learning is not possible (prior and posterior values are equal). The following example will show that although variations in the probability of the data in the model class make learning possible, they also lead to brittleness.

\subsection{Example of Learning vs Robustness}
\label{subsec:learningvsrobustness}

In this  example we are interested in estimating $\Phi(\mu^\dagger)=\mu^\dagger[a,1]$ for some $a\in (0,1)$, where $\mu^\dagger$ is an unknown distribution on the unit interval ($\mathcal{X}=[0,1]$)  based on the observation of $n$  data point $d_1, \ldots, d_n$ up to resolution $\delta$ (i.e.\ we
observe $d_i\in B_\delta(x_i)$ with $x_i\in [0,1]$ for $i=1,\ldots,n$).

Our purpose is to examine the sensitivity of the Bayesian answer to this problem with respect to the choice of a particular prior.
Consider the  model class
\begin{equation}\label{eqa}
\mathcal{A}:=\mathcal{M}([0,1]),
\end{equation}
and the
 class of priors
\[
		\Pi := \left\{ \pi \in \mathcal{M}(\mathcal{A}) \smid \E_{\mu \sim \pi} \bigl[ \E_{\mu}[X] \bigr] = m \right\}.
\]
Observe that $\Pi$ corresponds to the assumption that
$\mu^\dagger$ is the  realization of a random measure on $[0,1]$  whose mean
is on average $m$.

As in the previous example, the finite codimensional
class of priors $\Pi$ leads to brittleness in the sense that the least upper bound on prior values is
\begin{equation}\label{eqmua1ewer}
\mathcal{U}(\Pi)=\frac{m}{a} ,
\end{equation}
whereas, for $\delta \ll 1/n$, the least upper bound on posterior values (using Definition \ref{defn:robustness}) is the deterministic supremum of the quantity of interest (over $\mathcal{A}$), i.e.
\begin{equation}\label{eqmua1ewereses}
\mathcal{U}(\Pi|B^n_\delta)=1.
\end{equation}
Furthermore, worst priors are obtained by selecting priors for which the probability of observing the data $\mu^n[B^n_\delta]$ is
arbitrarily  close to zero except when $\Phi(\mu)$ is close to its deterministic supremum.
The bound on prior values \eqref{eqmua1ewer} is obtained from theorems \ref{thm_primred} and \ref{thm:reducpriormarg} in Examples \ref{eg:shiva} and \ref{eg:shiva2}. The bound on posterior values \eqref{eqmua1ewereses} is obtained from theorems \ref{thm:alternredded} and \ref{thm:shiva} in Examples \ref{eg:shivabis} and \ref{eg:shiva3}.

Can this brittleness be avoided by adding a uniform constraint on the probability of observing the data in the model class? To investigate this question let us introduce $\alpha \geq 1$ and a probability measure $\mu_{0}$ on $[0,1]$ with strictly positive Lebesgue density (with a prototypical example being that $\mu_0$ is itself uniform measure on $[0,1]$),
consider the (new) model class
\begin{equation}\label{eqadconstdata}
\mathcal{A}(\alpha):= \left\{\mu \in \mathcal{M}[0,1] \smid  \frac{1}{\alpha}\mu_0^n[B^n_\delta] \leq \mu^n[B^n_\delta] \leq \alpha\mu_0^n[B^n_\delta] \right\} ,
\end{equation}
and the (new) class of priors
\begin{equation}\label{eqpiddkje}
		\Pi(\alpha) := \left\{ \pi \in \mathcal{M}(\mathcal{A}(\alpha)) \smid \E_{\mu \sim \pi} \bigl[ \E_{\mu}[X] \bigr] = m \right\} .
\end{equation}

Note that,
for the model class $\mathcal{A}(\alpha)$, the probability of observing the data is uniformly bounded
below by $\frac{1}{\alpha}\mu_0^n[B^n_\delta]$ and
above by $\alpha\mu_0^n[B^n_\delta]$. Therefore, for $\alpha=1$,  the probability of observing the data is uniform in the model class, prior values are equal to posterior values, and the method is robust but learning is impossible.
If $\alpha$ slightly deviates from $1$, then the calculus developed in this paper
allows us to compute the least upper bound on posterior values and obtain that
\begin{equation}\label{eqhieuhdee}
\lim_{\delta \rightarrow 0}\mathcal{U}\big(\Pi(\alpha)|B^n_\delta\big)=\frac{1}{1+\frac{1}{\alpha^2} \frac{a-m}{m}}=\frac{m}{\frac{a}{\alpha^2}+m (1-\frac{1}{\alpha^2})} .
\end{equation}
We refer to Example \ref{eg:learnvsstab} for the derivation of  \eqref{eqhieuhdee} from
Theorem \ref{thm:alternredded}.

Note that the right hand side of \eqref{eqhieuhdee} is equal to $m/a$ for $\alpha=1$ (when the probability of the data is constant on the model class)
and \emph{quickly} converges towards $1$ as $\alpha$ increases. As a numerical application observe that for $a=\frac{3}{4}$ and $m=\frac{a}{2}=\frac{3}{8}$, we have $ \lim_{\delta \rightarrow 0} \mathcal{U}\big(\Pi(\alpha)\big)=\frac{1}{2} $ and
\[
 \lim_{\delta \rightarrow 0}\mathcal{U}\big(\Pi(\alpha)|B^n_\delta\big)=\frac{1}{1+\frac{1}{\alpha^2}}
\]
Therefore, for $\alpha=2$, we have (irrespective of the number of data points)
\[
 \lim_{\delta \rightarrow 0}\mathcal{U}\big(\Pi(2)|B^n_\delta\big)=0.8 ,
\]
and for $\alpha=10$, we have (irrespective of the number of data points)
\[
\lim_{\delta \rightarrow 0}\mathcal{U}\big(\Pi(10)|B^n_\delta\big)\approx 0.99  .
\]

Moreover, if $\alpha$ is derived by assuming the probability of each data point to be known up to some tolerance $\gamma$, i.e.\ if
the model class $\mathcal{A}(\alpha)$ is replaced by
\begin{equation}\label{eqadconstdatagt}
\mathcal{A}_\gamma:=\left\{\mu \in \mathcal{M}[0,1] \smid  \frac{1}{\gamma}\mu_0[B_\delta(x_i)] \leq \mu[B_\delta(x_i)] \leq \gamma \mu_0[B_\delta(x_i)]\text{ for }i=1,\ldots,n \right\}
\end{equation}
for some $\gamma>1$, then it can be shown that
\[
\lim_{\delta \rightarrow 0} \mathcal{U}(\Pi|B^n_\delta)=\frac{1}{1+\frac{1}{\gamma^{2n}}},
\]
 which exponentially converges towards $1$ as the number $n$ of data points goes to infinity.

In conclusion,  the effects of a uniform constraint on the probability of the
data under finite information in the model class show
that learning ability comes at the price of loss in stability in the following sense: when $\alpha=1$, the data is equiprobable under all measures in the model class, posterior values are equal to prior values, the method is robust but learning is not possible. As $\alpha$ deviates from one, the learning ability increases as  robustness decreases, and when $\alpha$ is large, learning is possible but the method is brittle.

\subsection{Missing Stability Condition for Using Bayesian Inference Under Finite Information}

The previous examples have shown that Bayesian inference can be unstable under finite information, therefore, at the very least, the question of the existence and of the nature of a stability condition for using Bayesian inference remains to be answered.
Indeed it is well known that numerical solutions of PDEs can become unstable if specific stability conditions such as the CFL stability condition are not satisfied. Although numerical schemes that do not satisfy the CFL condition may look grossly inadequate,
the existence of such perverse examples does not imply the dismissal of the necessity of a stability condition. Similarly, although one may, as in Subsection \ref{subsecex2}, exhibit grossly perverse worst priors, the existence of such priors does not invalidate the question of the missing stability condition for using Bayesian inference under finite information.

The example provided in Subsection \ref{subsec:learningvsrobustness} suggests that,
in the framework of Bayesian sensitivity analysis,
(i) such a stability condition would depend on how well the probability of the data is  known or  constrained in the model class, and
(ii) learning and robustness are antagonistic/conflicting requirements --- there is no free lunch and increased learning potential is paid for by decreased stability of posterior values.

Could this stability condition be derived from closeness in Kullback--Leibler
divergence? The problem with this approach is that closeness in Kullback--Leibler
divergence cannot be tested with discrete data and it requires the non-singularity of the data generating distribution with respect to the model, which could be a strong assumption for the certification the safety of a critical system. Indeed, when performing Bayesian analysis on function spaces, as is now increasingly popular, for
studying PDE solutions, results like the Feldman–-H{\'a}jek theorem
\cite{Feldman:1958, Hajek:1958} tell us that \emph{most} pairs of measures are mutually singular, and hence at Kullback--Leibler
\emph{distance} infinity from one another. Another problem with using Kullback--Leibler
divergence  is that a local sensitivity analysis (in the sense of Fr{\'e}chet
derivatives) of posterior values suggests infinite sensitivity as the number of data point goes to infinity  \cite{GustafsonWasserman:1995} (and this result is valid for the broader class of divergences
that includes the Hellinger distance).

A close inspection of some of the cases where Bayesian inference has been successful shows the existence of a non-Bayesian feedback loop on the evaluation of its performance \cite{Mayo:2012, MayoSpanos:2004, Senn:2011}.
Therefore one natural question is whether the missing stability condition could be derived by exiting the strict framework of Bayesian analysis/inference. According to Efron \cite{Efron:2013}, without genuine prior information
\begin{quote}
``Bayesian calculations
cannot be uncritically accepted and should be checked
by other methods, which usually means frequentistically.''
\end{quote}

\subsection{Calculus for Measures over Measures}
The results of this paper are derived from a calculus allowing us to solve/reduce optimization problems with variables corresponding to measures over measures over arbitrary Polish spaces. The following assertion of Theorem \ref{thm:reducpriormarg} is an example of this calculus.
	\begin{equation}
		\label{eq:2qrewriteiubis}
		\sup_{\pi \in \Psi^{-1}\mathfrak{Q}} \E_{\mu \sim \pi}\big[\Phi(\mu)\big] = \sup_{\mathbb{Q}\in \mathfrak{Q}}\left[\E_{q\sim \mathbb{Q}}
\Big[\sup_{\mu\in \Psi^{-1}(q)}\Phi(\mu) \Big]\right].
	\end{equation}
In \eqref{eq:2qrewriteiubis}, $\Psi$ is a measurable function mapping $\mathcal{A}$ (a Suslin subset of the set $\mathcal{M}(\mathcal{X})$ of probability measures on a Polish space $\mathcal{X}$) into
a separable  metrizable space $\mathcal{Q}$, $\mathfrak{Q}$ is a subset of $\mathcal{M}(\mathcal{Q})$, and $\Phi$ is a measurable quantity of interest defined on $\mathcal{M}(\mathcal{X})$.
Therefore, \eqref{eq:2qrewriteiubis} states that the optimization problem (in its left hand side) over $\Psi^{-1}\mathfrak{Q}$ (a subset of the set of measures of $\mathcal{A}$, i.e.\ a
subset of the set of measures of the set of measures over $\mathcal{X}$) is equal to the nesting of an optimization problem over $\Psi^{-1}(q)$ (a subset of $\mathcal{A}$, i.e.\ a
subset of the set of measures over $\mathcal{X}$) and an optimization problem over $\mathfrak{Q}$ (a subset of the set of measures over $\mathcal{Q}$).

We will now illustrate this calculus by showing how \eqref{eqmua1ewer} can be derived through a simple application of \eqref{eq:2qrewriteiubis}.
First we need to give a short reminder on  optimization  over measures via the following problem.
\begin{pb}\label{pb:3}
A child is given one pound of playdoh and the seesaw illustrated by Figure \ref{fig:playdoh}.(a). How much mass can she put above the threshold $a$ while keeping the seesaw balanced at $m$?
\end{pb}
The mathematical formulation of the question articulated in Problem \ref{pb:3} is as follows. What is the least upper bound on $\P[X\geq a]$ if $\P$ is an unknown (imperfectly known) probability measure
 on $[0,1]$ having mean $m$? The answer to this question is
\begin{equation}\label{eq:pbopinf}
\sup_{\mu \in \mathcal{A}}\mu[a,1]
\end{equation}
where $\mathcal{A}$ is the set of probability measures
on $[0,1]$ having mean $m$.
Although \eqref{eq:pbopinf} is an infinite dimensional optimization problem over measures, it is easy to see that to achieve the maximum, any mass put above $a$ should be placed exactly at $a$ to create minimum leverage towards the right hand side of the seesaw and any mass put below $a$ should be placed at $0$ to create maximum leverage towards the left hand side of the seesaw (as illustrated in Figure \ref{fig:playdoh}.(a)). This simple argument allows to reduce  \eqref{eq:pbopinf} to a simple one dimensional problem whose solution is $\frac{m}{a}$ and corresponds to Markov's inequality. This simple example of reduction calculus has a generalization to spaces of functions and measures  \cite{OSSMO:2011} and is based on a form of linear programming in spaces of measures. In particular, the calculus developed  in \cite{OSSMO:2011} uses results of Winkler \cite{Winkler:1988} --- which follow from an extension of Choquet theory
(see e.g.~\cite{Phelps:2001}) by von Weizs\"{a}cker and Winkler \cite[Corollary 3]{WeizsackerWinkler:1979} to sets of probability measures with generalized moment constraints --- and a result of Kendall \cite{Kendall:1962} characterizing cones, which are lattice cones in their own order.

	\begin{figure}[tp]
\begin{center}
	\subfigure[Problem \ref{pb:3}]{
			\includegraphics[width=0.4\textwidth]{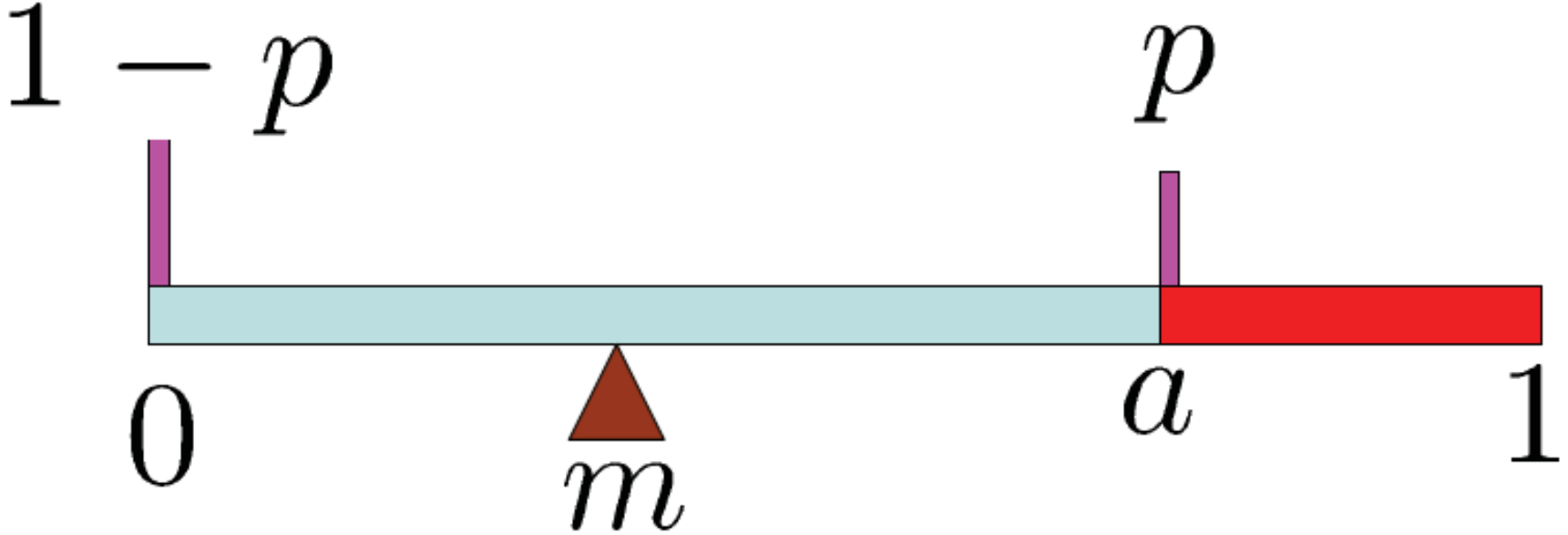}
		}
		\subfigure[Problems \ref{pb:4} and \ref{pb:5}]{
			\includegraphics[width=0.4\textwidth]{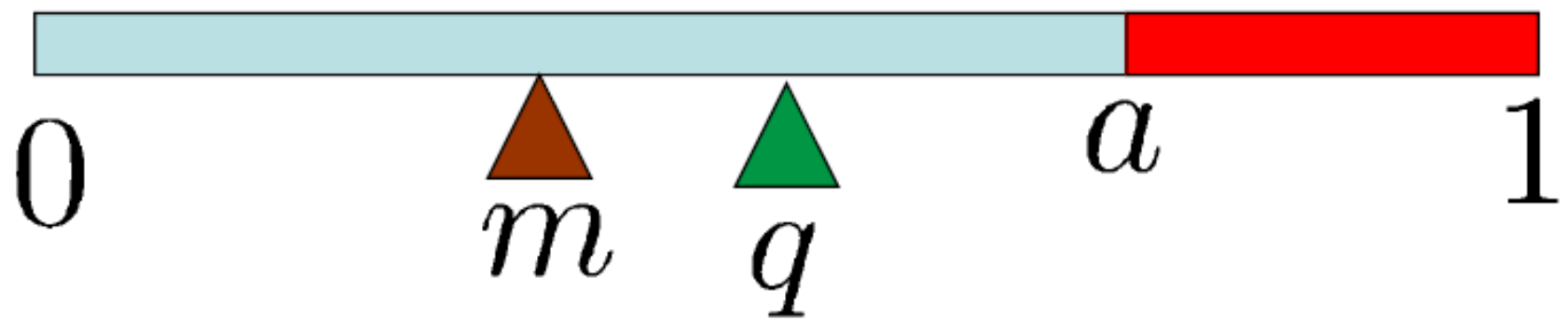}
		}
		\caption{Reduction of optimization problems over measures and over measures over measures.}
		\label{fig:playdoh}
\end{center}
	\end{figure}

We will now consider the next level of
complexity,
illustrated by the following two equivalent problems.

\begin{pb}\label{pb:4}
$10,000$ children are, each, given one pound of playdoh and a seesaw. On average, how much mass can they put above the threshold $a$ while, on average, keeping the seesaws balanced at $m$?
\end{pb}

\begin{pb}\label{pb:5}
A child is given one pound of playdoh and a seesaw. What can you say about how much mass she is putting above the threshold $a$ if all you have is the belief that she is keeping the seesaw balanced at $m$?
\end{pb}

The mathematical formulation of problems \ref{pb:4} and \ref{pb:5} is as follows (for Problem \ref{pb:4}, replace $10,000$ by $N$ and consider the asymptotic limit $N\rightarrow \infty$). What is the least upper bound on $\E_{\mu \sim \pi }\big[\mu[X\geq a]\big]$ if $\pi$ is an unknown (imperfectly known) probability measure
 on $\mathcal{M}\big([0,1]\big)$ (the set of probability distributions on $[0,1]$) such that
$\E_{\mu \sim \pi }\big[\E_{\mu}[X]\big]=m$?

 The answer to this question is
\begin{equation}\label{eq:pbopinfnextlev}
\sup_{\pi \in \Pi}\E_{\mu \sim \pi }\big[\mu[X\geq a]\big] ,
\end{equation}
where $\Pi$ is the set of measures of probability $\pi$ on the set of measures of probability on $[0,1]$ such that $\E_{\mu \sim \pi }\big[\E_{\mu}[X]\big]=m$.

Although \eqref{eq:pbopinfnextlev} is an optimization over measures over measures, the  calculus  of \eqref{eq:2qrewriteiubis} introduced in Theorem \ref{thm:reducpriormarg} allows us to reduce it to the nesting of two optimization problems over measures as follows.

\begin{equation}\label{eq:sj2jgee}
\sup_{\pi \in \Pi}\E_{\mu \sim \pi }\big[\mu[X\geq a]\big]= \sup_{\mathbb{Q}\in \mathcal{M}([0,1])\,:\, \E_\mathbb{Q}[q]=m} \mathbb{E}_{q \sim \mathbb{Q}}
 \left[ \sup_{\mu \in \mathcal{M}([0,1])\,:\, \E_\mu[X]=q } \mu[X\geq a] \right] .
\end{equation}

Observe that \eqref{eq:sj2jgee} is obtained from \eqref{eq:2qrewriteiubis} by taking $\mathcal{X}=[0,1]$, $\Psi(\mu)=\E_\mu[X]$, $\mathcal{Q}=[0,1]$ and $\mathfrak{Q}$ as the set of measures of probability on $\mathcal{Q}$ having mean $m$.
In particular, note that in \eqref{eq:sj2jgee}, the inner optimization problem involves taking a supremum over all measures $\mu$ on $[0,1]$ having mean $q$ and the outer optimization problem involves
 taking a supremum over the probability distribution of $q$, i.e.\ the
 set of distributions on $[0,1]$ having mean $m$.
Solving the inner optimization problem as described below Problem \ref{pb:3} leads to:

\[
\sup_{\pi \in \Pi}\E_{\mu \sim \pi }\big[\mu[X\geq a]\big]= \sup_{\mathbb{Q}\in \mathcal{M}([0,1])\,:\, \E_\mathbb{Q}[q]=m} \mathbb{E}_{q \sim \mathbb{Q}}
 \left[ \min(\frac{q}{a},1)\right]  ,
\]
and solving the outer optimization step gives the following solution.
\[
\sup_{\pi \in \Pi}\E_{\mu \sim \pi }\big[\mu[X\geq a]\big]= \frac{m}{a} .
\]

\subsection{Structure of the Paper and Main Results}

This paper is structured as follows:

Section~\ref{sec:Plan} incorporates Bayesian priors into the Optimal Uncertainty Quantification (OUQ) framework \cite{OSSMO:2011}.  In the OUQ framework, Uncertainty Quantification (UQ) is formulated as an optimization problem (over an infinite-dimensional set of functions and measures) corresponding to extremizing (i.e.~finding worst and best case scenarios) probabilities of failure or other quantities of interest, subject to the constraints imposed by the scenarios compatible with the assumptions and information.
In this generalization, priors are probability measures on spaces of measures, and computing optimal bounds on prior values (given a set of priors) requires solving problems in which the optimization variables are measures on spaces of measures (the results of this paper can be extended to measures over spaces of measures and functions but, for the sake of simplicity and clarity, we will limit the presentation to measures over measures).

Section~\ref{Sec:OpBoundsPriorEstimates} shows how such optimization problems can, under general conditions, be reduced to the nesting of two optimization problems over measures, where then we can apply the reduction theorems of \cite{OSSMO:2011}.

Section \ref{Sec:OpBoundsPosteriorEstimates} provides similar reduction theorems for the computation of optimal bounds on posterior values given a set of priors and the observation of the data.
These reduction theorems lead to the  brittleness results of Theorems~\ref{thm:shiva}, \ref{thm_localshivacor},
 and \ref{thm_localshiva}.

Section~\ref{sec:misspecori} reviews questions of Bayesian consistency, inconsistency, model misspecification, and robustness through a motivating analysis and interprets the results of this paper in relation to those questions.

Section~\ref{Sec:local} presents the  brittleness under local misspecification results of
Theorems~\ref{thm_localshivacor} and \ref{thm_localshiva}.
That is, given a model, Theorem~\ref{thm_localshivacor} provides optimal bounds on posterior values for priors that are at arbitrarily small distance (in the Prokhorov or total variation metrics) from a given model. Theorems~\ref{thm_localshivacor} and \ref{thm_localshiva} show that these optimal bounds on posterior values are the essential supremum and infimum of the quantity of interest irrespective of the size of data and of the size of the metric neighborhood around the model.
Finally, Sections \ref{Sec:proofs} and \ref{sec:appendix} contain the proofs.

\section{General Set-Up}
\label{sec:Plan}

\subsection{Notation and Conventions}
Throughout, for  a topological space $\mathcal{Y}$,  $\mathcal{B}(\mathcal{Y})$ will denote the Borel $\sigma$-algebra of subsets of $\mathcal{Y}$ and $\mathcal{M}(\mathcal{Y})$ will denote the space of Borel probability measures generally endowed with the weak topology and the corresponding Borel $\sigma$-algebra unless specified otherwise. For an alternative $\sigma$-algebra $\Sigma_{\mathcal{Y}}$ of subsets of $\mathcal{Y}$ the set of probability measures on the  $\sigma$-algebra $\Sigma_{\mathcal{Y}}$  will be denoted $\mathcal{M}(\Sigma_{\mathcal{Y}})$.    For a mapping between topological spaces, the term ``measurable''
will mean Borel measurable unless specified otherwise.
Moreover, suprema over the empty set will have the value $-\infty$ and infima over the empty set the value $+\infty$.

\subsection{The General Problem and the Optimal Uncertainty Quantification (OUQ) Framework}
Let $\mathcal{X}$ be Polish and $\Phi$ be a measurable function mapping $\mathcal{M}(\mathcal{X})$, the set of measures of probability on $\mathcal{X}$, onto the real line $\R$, known as the \emph{quantity of interest}.
Let $\mu^\dagger$ be an unknown or imperfectly known probability measure
on $\mathcal{X}$. The general problem guiding our presentation will be that of estimating $\Phi(\mu^\dagger)$.

Let $\mathcal{A}$ be an arbitrary subset of $ \mathcal{M}(\mathcal{X})$.  If $\mathcal{A}$ represents all that is known about $\mu^\dagger$ (in the sense that $\mu^\dagger \in \mathcal{A}$ and that any $ \mu\in \mathcal{A}$ could, a priori, be $\mu^\dagger$ given the available information) then
 \cite{OSSMO:2011} shows that the quantities
\begin{align}
	\label{eq:defma1}
	\mathcal{U}(\mathcal{A}) &:= \sup_{\mu \in \mathcal{A}} \Phi(\mu)\\
	\label{eq:defma2}
	\mathcal{L}(\mathcal{A}) &:= \inf_{ \mu\in \mathcal{A}} \Phi( \mu)
\end{align}
determine the inequality
\begin{equation}
	\label{ineq_OUQ}
	\mathcal{L}(\mathcal{A}) \leq \Phi(\mu^\dagger) \leq \mathcal{U}(\mathcal{A}),
\end{equation}
to be optimal given the available information $\mu^\dagger \in \mathcal{A}$ as follows:  It is simple to see that the inequality \eqref{ineq_OUQ}  follows from the assumption that $\mu^\dagger \in \mathcal{A}$.  Moreover, for any $\varepsilon >0$  there exists a $\mu \in \mathcal{A}$ such that
\[
	\mathcal{U}(\mathcal{A})-\varepsilon < \Phi(\mu) \leq \mathcal{U}(\mathcal{A}) .
\]
Consequently since all that we know about $\mu^\dagger$ is that $\mu^\dagger \in \mathcal{A}$, it follows that the upper bound $\Phi(\mu^\dagger) \leq \mathcal{U}(\mathcal{A})$ is the best obtainable given that information, and the lower bound is optimal in the same sense.

 Although the OUQ optimization problems \eqref{eq:defma1} and \eqref{eq:defma2} are extremely large, we have shown in \cite{OSSMO:2011}, for the more general situation where $\mathcal{A}$ is a set of functions $f$ and measures $\mu$ and $\Phi$ a function of $(f,\mu)$, that an important subclass enjoys significant and practical finite-dimensional reduction properties.  First, by \cite[Cor.~4.4]{OSSMO:2011}, although the optimization variables $(f, \mu)$ lie in a product space of functions and probability measures, for OUQ problems governed by linear inequality constraints on generalized moments, the search can be reduced to one over probability measures that are products of finite convex combinations of Dirac masses with explicit upper bounds on the number of Dirac masses.
Furthermore, in the special case that all constraints are generalized moments of functions of $f$, the dependency on the coordinate positions of the Dirac masses is eliminated by observing that the search over admissible functions reduces to a search over functions on an $m$-fold product of finite discrete spaces, and the search over $m$-fold products of finite convex combinations of Dirac masses reduces to a search over the products of probability measures on this $m$-fold product of finite discrete spaces \cite[Thm.~4.7]{OSSMO:2011}.  Finally, by \cite[Thm.~4.9]{OSSMO:2011}, using the lattice structure of the space of functions, the search over these functions can be reduced to a search over a finite set.

For the sake of clarity we will now restrict the presentations of our results to the (simpler) situation where the quantity of interest $\Phi$ is  (solely) a function of an unknown measure $\mu$.  As in \cite{OSSMO:2011}, the results of this paper can be generalized to situations where $\Phi$ is a function of $(f,\mu)$.

\begin{eg}
	\label{eg:1}
	A classic example, when $\mathcal{X}=\R$ is $\Phi(\mu):=\mu[X\geq a]$ where $a$ is a safety margin. In the certification context one is interested in showing that $\mu^\dagger[X\geq a]\leq \varepsilon$, where $\varepsilon$ is a safety certification threshold (i.e.\ the maximum acceptable $\mu^\dagger$-probability of the system exceeding the safety margin $a$).  If $\mathcal{U}(\mathcal{A}) \leq \varepsilon$, then the system associated with $\mu^\dagger$ is safe even in the worst case scenario (given the information represented by $\mathcal{A}$).  If $\mathcal{L}(\mathcal{A}) > \varepsilon$, then the system associated with $\mu^\dagger$ is unsafe even in the best case scenario (given the information represented by $\mathcal{A}$).  If $\mathcal{L}(\mathcal{A}) \leq \varepsilon < \mathcal{U}(\mathcal{A})$, then the safety of the system cannot be decided (although we could declare the system to be unsafe due to lack of information).
\end{eg}

\subsection{Bayesian Priors on the Admissible Set}

In the OUQ setting, an assumption of the form $\mu^\dagger \in \mathcal{A}$
was used to derive the optimal inequality \eqref{ineq_OUQ}.
This paper will consider the situation in which one has priors on the admissible set $\mathcal{A}$ and also information in the form of sample data.  One of our goals is to analyse the robustness (or brittleness) of Bayesian inference by obtaining optimal bounds on posterior values given local misspecifications.
In that context $\mathcal{A}$ can be viewed as a model class, and $\mu^\dagger$, as the realization of a probability measure
 (the prior) on $\mathcal{A}$.
In order to define priors on the space of admissible scenarios, $\mathcal{A}$ needs to be given the structure of a measurable space;  i.e.\ a suitable $\sigma$-algebra $\Sigma_{\mathcal{A}}$ on $\mathcal{A}$ must be provided.
From now on,  we will assume $\mathcal{A}$ to be a Borel subset of the
Polish space $\mathcal{M}(\mathcal{X})$, endowed with the Borel $\sigma$-algebra for $\mathcal{A}$.
We will also refer to a probability measure $\pi \in \mathcal{M}(\Sigma_{\mathcal{A}})$ as a \emph{prior}.

\begin{rmk}
	\label{rmk:Polish}
	The desire to have the Borel measurable structure of a Polish space might seem to be a spurious level of abstraction, but there are many good reasons for it.  The first is that, by Suslin's Theorem \cite[Thm.~14.2]{Kechris:1995},    all Borel subsets of a Polish space are Suslin, where a \emph{Suslin space} is a continuous Hausdorff image of a Polish space.  Indeed, Suslin sets are important in measurable selection theorems (see e.g.~\cite{CastaingValadier:1977}) such as those that we use in the proof of Lemma~\ref{lem:Umeasurable};  furthermore, in addition to Ulam's theorem \cite[Thm.~4.3.8]{Ash:1972} that all probability measures on a Polish space are regular (approximable from within by compact sets), Schwartz' theorem \cite{Schwartz:1974} implies that that all probability measures on a Suslin space are regular, and, therefore, \cite[Thm.~11.1]{Topsoe:1970} implies that the extreme points in the space of probability measures on a Suslin space are the Dirac measures.  Consequently, when  $\mathcal{M}(\mathcal{X})$ is Polish, any Borel subset $\mathcal{A} \subseteq  \mathcal{M}(\mathcal{X})$ is Suslin and so the extreme points of probability measures on $\mathcal{A}$ are the Dirac measures, and some powerful measurable selection theorems are available.  Moreover, when the base space is metrizable, then the space of probability measures is Polish in the weak topology if and only if the base space is Polish.

	Furthermore, since separability is equivalent to second countability for metric spaces, we have that the Borel structure of a product is  the product of Borel structures of Polish spaces.  In addition, by \cite[Thm.~10.2.2]{Dudley:2002}, regular conditional probabilities exist for observables with values in a Polish space.    Also, Polish spaces are the spaces of Descriptive Set Theory, see e.g.~Kechris \cite{Kechris:1995}.  Polish spaces
	appear to be the appropriate spaces to play topological games such as the Banach--Mazur game \cite{Oxtoby:1971}, the Sierpi{\'n}ski game, the Ulam game, the Banach game, and the Choquet game.  Moreover, a theorem of Choquet \cite[Thm.~8.18]{Kechris:1995} shows that a separable metric space is completely metrizable (and hence Polish) if and only if the second player has a winning strategy in the strong Choquet game.   For a review of topological games, see Telg\'{a}rsky's review \cite{Telgarsky:1987}, and for topological games in hyperspace see that of Zsilinszky \cite{Zsilinszky:1998}.
\end{rmk}

\subsection{Data Spaces and Maps}
\label{subsec:dataspacesanmaps}

  In practice,
  the probability measure $\mu^\dagger$ is not observed directly;  instead
  the sample data arrives in the form of (realizations of) observation random variables, the distribution of which is related to $\mu^\dagger$.  To simplify the current presentation,
  we will assume that this relation is determined by a function of $\mu^\dagger$ --- such as the case where the data  $X_1, \ldots, X_n$ are determined by $n$ independent realizations $X_i$ of the random variable $X$ determined by the possibly unknown distribution $\mu^\dagger$.  Throughout this paper we will use the following notation:  $\Dspace$ will denote the observable space (i.e.\ the space in which the sample data take values);  $\Dspace$ will be assumed to be a metrizable Suslin space and $\Drv$ will denote a $\Dspace$-valued random variable producing the observed sample data.  To represent the dependence of the observation random variable
$\Drv$ on the unknown state $\mu^\dagger \in \mathcal{A}$ we introduce a measurable function
\[
	\Dmap \colon \mathcal{A} \to \mathcal{M}(\Dspace),
\]
where $\mathcal{M}(\Dspace)$ is given the Borel structure corresponding to the weak topology, to define this relation.  The idea is that $\Dmap( \mu)$ is the probability distribution of the observed sample data $\Drv(\mu)$ if $\mu^\dagger=\mu$, and for this reason it may be called the \emph{data map} or --- even more loosely --- the \emph{observation operator}.  Often, for simplicity, we will write $\Drv$ instead of $\Drv( \mu)$.  Note that when the data comes in the form of $n$ i.i.d.\ realizations of $\mu^\dagger$ we have $\Dspace=\mathcal{X}^n$ and $\Dmap(\mu)=\mu^n$ (where $\mu^n$ is the $n$-fold tensorization of $\mu$).

We proceed with a natural generalization of the Campbell measure and Palm
distribution associated with a random measure as described in \cite{Kallenberg:1975} (see also \cite[Ch.~13]{Daley} for a more current treatment).  To that end, observe that since $\Dspace$ is metrizable, it follows from \cite[Thm.~15.13]{AliprantisBorder:2006}, that, for any $B \in \mathcal{B}(\Dspace)$, the evaluation $\nu \mapsto \nu(B)$, $\nu \in \mathcal{M}(\Dspace)$,  is measurable.  Consequently, the measurability of $\Dmap$ implies that the mapping
\[
	\widehat{\Dmap} \colon \mathcal{A} \times \mathcal{B}(\Dspace) \to \R
\]
defined by
\[
	\widehat{\Dmap}( \mu, B) := \Dmap(\mu)[B], \quad \text{for } \mu \in  \mathcal{A}, B \in
 \mathcal{B}(\Dspace)
\]
is a transition function in the sense that, for fixed $\mu \in \mathcal{A}$,  $\widehat{\Dmap}( \mu, \quark )$ is a probability measure, and, for fixed $B \in \mathcal{B}(\Dspace)$, $\widehat{\Dmap}\bigl( \quark ,B\bigr)$ is Borel measurable.
Therefore, by \cite[Thm.~10.7.2]{Bogachev2}, any $\pi \in \mathcal{M}(\mathcal{A})$, defines a probability measure
\[
	\pi \odot \Dmap \in \mathcal{M}\bigl(\mathcal{B}(\mathcal{A}) \times \mathcal{B}(\mathcal{D})\bigr)
\]
through
\begin{equation}
	\label{eq:palm}
	\pi\odot \Dmap \big[ A \times B \big] := \E_{\mu \sim \pi} \big[ \one_{A}( \mu) \Dmap( \mu)[B] \big],\quad \text{for } A \in \mathcal{B}(\mathcal{A}), B \in \mathcal{B}(\Dspace) ,
\end{equation}
where $\one_{A}$ is the indicator function of the set $A$:
\[
	\one_{A}(\mu) :=
	\begin{cases}
		1, & \text{if $\mu \in A$,} \\
		0, & \text{if $\mu \notin A$.}
	\end{cases}
\]
It is easy to see that $\pi$ is the $\mathcal{A}$-marginal of $\pi\odot \Dmap$.  Moreover, when $\mathcal{X}$ is Polish, \cite[Thm.~15.15]{AliprantisBorder:2006} implies that $\mathcal{M}(\mathcal{X})$ is Polish, and  it follows that $\mathcal{A} \subseteq \mathcal{M}(\mathcal{X})$
is second countable.  Consequently, since $\Dspace$ is Suslin and hence second countable, it follows from \cite[Prop.~4.1.7]{Dudley:2002} that
\[
	\mathcal{B}\bigl(\mathcal{A} \times \Dspace\bigr)=\mathcal{B}(\mathcal{A}) \times\mathcal{B}(\Dspace)
\]
and hence $\pi \odot \Dmap$ is a probability measure on $\mathcal{A} \times \Dspace$.  That is,
\[
	\pi\odot \Dmap \in \mathcal{M}(\mathcal{A} \times \Dspace).
\]

Let us refer to an element of $\mathcal{M}(\mathcal{A})$ as a \emph{prior} on $\mathcal{A}$.  With a prior $\pi$ on $\mathcal{A}$, the quantity of interest $\Phi(\mu)$ becomes a random variable and we will be interested in estimating its distribution conditioned on the observation $\Drv\in B$, where $B \in \mathcal{B}(\Dspace)$.

\begin{eg}
	In the context of Example~\ref{eg:1}, we are interested in estimating the probability (under the prior $\pi$) that the system is unsafe, conditioned on the observations  $\Drv\in B$, i.e.\ the conditional expectation
	\[
		(\pi\odot \Dmap)\Bigl[ \mu[X\geq a]> \epsilon \Big| \Drv\in B\Bigr] .
	\]
	If $\Drv$ corresponds to observing independent realizations of $X$, then the observation space $\Dspace$ is $\mathcal{X}^n$ and the measure $\Dmap(\mu)$ is $\mu^n$.

	If $\Drv$ is the random variable that results from observing $n$ independent realizations of $(X +\xi)$ ($X$ is observed with additive Gaussian noise $\xi \sim \mathcal{N}(0, \sigma^{2})$), then the measure $\Dmap( \mu)$ is the one associated with the random variable $\Drv = \big( X^1 +\xi^1, \dots, X^n+ \xi^n) \big)$ where the $X^i$ are independent and distributed according to $\mu$ and the $\xi^i$ are independent Gaussian random variables of mean zero and variance $\sigma^2$.
\end{eg}

\subsection{Bayes' Theorem and Conditional Expectation}
\label{subsec:condexp}

Henceforth $\mathcal{A}$ will be a Suslin space, and suppose now that we have $\pi\odot \Dmap \in \mathcal{M}(\mathcal{A} \times \Dspace)$ constructed in the above way.  Let $\pi\cdot\Dmap$ denote the corresponding  Bayes' sampling distribution defined by the $\Dspace$-marginal of $\pi\odot \Dmap$,  and note that, by \eqref{eq:palm}, we have
\begin{equation}
    \label{eq:cdotexp}
    \pi\cdot\Dmap[B]:=\E_{\mu\sim \pi}\big[\Dmap(\mu)[B]\big], \quad \text{for } B \in
\mathcal{B}(\Dspace) .
\end{equation}

Since both $\Dspace$ and $\mathcal{A}$ are Suslin it follows that the product $\mathcal{A} \times \Dspace$ is Suslin. Consequently, \cite[Cor.~10.4.6]{Bogachev2} asserts that regular conditional probabilities exist for any sub-$\sigma$-algebra of $\mathcal{B}\bigl(\mathcal{A} \times \Dspace\bigr)$. In particular, the product theorem of \cite[Thm.~10.4.11]{Bogachev2} asserts that product regular conditional probabilities
\[
	\bigl(\pi\odot \Dmap\bigr)|_{d} \in \mathcal{M}(\mathcal{A}), \quad \text{for } d \in \Dspace
\]
exist and that they are $\pi\cdot \Dmap$-a.e.\ unique.

When we consider $\pi \in \mathcal{M}(\mathcal{A})$ a prior, then this result can be interpreted as the posteriors of Bayes' theorem.  However, because such regular conditional probabilities are only uniquely defined
$\pi\cdot \Dmap$-a.e., when a data sample $d \in \Dspace$ arrives such that $\pi\cdot \Dmap[\{d\}]=0$, a  posterior $\bigl(\pi\odot \Dmap\bigr)|_{d}$ that could be \emph{any} of the $\pi\cdot \Dmap$-a.e.-equal regular conditional probabilities evaluated at $d$ appears to have dubious utility.  Indeed, the fact that the regular conditional probabilities are only uniquely defined $\pi\cdot \Dmap$-a.e.\ suggests that integrals of posteriors over subsets $B \in \mathcal{B}(\Dspace)$ such that $\pi\cdot \Dmap[B]>0$ are the more natural objects.  Moreover, the restriction that $B$ be an open set is natural for practical reasons, since conditioning on $\Drv$ lying in an open subset $B$  rather than on its  exact value is what one has to do when the sample data can only be observed after rounding error.  Furthermore, we will show in Section~\ref{Sec:OpBoundsPosteriorEstimates} that
if the data $\Ddata$ have been sampled from a probability measure $\pi^{\dagger}\cdot \Dmap$ for some $\pi^{\dagger}\in \mathcal{M}(\mathcal{A})$ (commonly called a ``true prior'' in Bayesian statistics) then with $\pi^{\dagger}\cdot \Dmap$ probability one (on the realization of $d$), the $\pi^{\dagger}\cdot \Dmap$-measure of any open set containing $d$ is strictly positive.  In other words, $\pi^{\dagger}\cdot \Dmap$-almost surely, $\pi^{\dagger}$ (the ``true prior'') belongs to the random subset of $\mathcal{M}(\mathcal{A})$ defined as the priors $\pi \in \mathcal{M}(\mathcal{A})$ such that $\pi\cdot \Dmap[B]>0$ for any open set $B$ containing the data $d$ (this subset is randomized through the realization of the data $d$).

Finally, throughout, we will find it useful to assume that
\begin{assumption}
	\label{ass_semibounded}
	$\Phi$ is semibounded
\end{assumption}
in that it is either bounded above or bounded below.  Semiboundedness is sufficient to ensure that the integral of $\Phi$ with respect to any probability  measure exists, possibly with the value $\infty$ or $-\infty$, and such integrands are sufficient for the reduction theorems of Winkler \cite{Winkler:1988} that we use.

\begin{rmk}
	Note that the assumption that $\Phi$ is semibounded is mostly for convenience since integrands which are not semibounded, like that defining the first moment, can be considered by restricting the space of measures to those measures that have well-defined
	first moments.
\end{rmk}

\subsection{Incompletely Specified Priors}
\label{sec_incompletepriors}

In  practical situations, (1) the choice of a particular prior on $\mathcal{A}$ involves a degree of arbitrariness that may be incompatible with the certification of rare/critical events, and
(2) the definition of such a prior is a non-trivial
task if $\mathcal{A}$ is infinite dimensional.
For these reasons it is necessary to consider situations in which the prior $\pi$ is imperfectly known or specified.  More precisely, the (lack of) information (or specification) on $\pi$  can be represented via the introduction of a space $\Pi$  where the  subset
$\Pi \subseteq \mathcal{M}(\mathcal{A})$ consists of the set of admissible priors $\pi$.

One of our goals in allowing incompletely specified priors is to assess the robustness of posterior Bayesian estimates with respect to the particular choice of priors. More precisely we will compute optimal bounds on $\E_{\pi}[\Phi]$ when $\pi \in \Pi$ and show how these bounds are affected by the introduction of sample data by computing optimal bounds on  $\E_{\pi \odot \Dmap}[\Phi|B]$, for $B \in \mathcal{B}(\Dspace)$.

\section{Optimal Bounds on the Prior Value}
\label{Sec:OpBoundsPriorEstimates}

Recall that for a subset $\mathcal{A}$ and a measurable quantity of interest $\Phi \colon \mathcal{A} \to \R$, that under the assumption $\mu^\dagger \in \mathcal{A}$, we have the optimal upper $\mathcal{U}(\mathcal{A})$ and lower $\mathcal{L}(\mathcal{A})$ bounds on the \emph{value} $\Phi(\mu^\dagger)$ of the quantity of interest,
defined in \eqref{eq:defma1} and \eqref{eq:defma2} by
\begin{align*}
	\mathcal{U}(\mathcal{A}) &:= \sup_{ \mu\in \mathcal{A}} \Phi( \mu)\\
	\mathcal{L}(\mathcal{A}) &:= \inf_{ \mu\in \mathcal{A}} \Phi( \mu)\, .
\end{align*}

When we put a prior $\pi$ on $\mathcal{A}$, we have to define the \emph{value $\bar{\Phi}(\pi)$ of the prior} $\pi$
corresponding to an extended quantity $\bar{\Phi} \colon \mathcal{M}(\mathcal{A}) \to \R$ of interest
corresponding to $\Phi$.  Disregarding integrability concerns, for a given $\Phi$, let us call the induced function
\begin{equation}
	\label{def_interest_can}
	\bar{\Phi}(\pi):=\E_{\pi}[\Phi], \quad \pi \in \mathcal{M}(\mathcal{A}),
\end{equation}
the canonical one associated with $\Phi$ and abuse notation by denoting the function $\bar{\Phi}$ as $\Phi$.
For such a canonical quantity of interest, we call the value $\E_{\pi}[\Phi]$ the
\emph{prior value}, and note that the values
\begin{align}
	\label{eq:UPi}
    \mathcal{U}(\Pi) &:= \sup_{\pi \in \Pi} \E_{\pi} \big[\Phi\big]\\
	\label{eq:LPi}
    \mathcal{L}(\Pi) &:= \inf_{\pi \in \Pi} \E_{\pi} \big[\Phi \big]
\end{align}
form a natural generalization of the values $\mathcal{U}(\mathcal{A})$ and  $\mathcal{L}(\mathcal{A})$.  Moreover, in the same way that $\mathcal{U}(\mathcal{A})$ and $\mathcal{L}(\mathcal{A})$ are optimal upper and lower bounds on $\Phi(\mu^\dagger)$ given the information that $(\mu^\dagger) \in \mathcal{A}$,  $\mathcal{U}(\Pi)$ and $\mathcal{L}(\Pi)$ are optimal upper and lower bounds on $\E_{\pi}\big[\Phi\big]$ given the information that $\pi \in \Pi$.  Of course, for these expressions
to be well defined, integrability concerns should be addressed.  Indeed, Assumption~\ref{ass_semibounded}
implies that $\E_{\pi} \big[\Phi\big]$ is well defined for any bounded measure $\pi$, possibly with the value $\infty$ or
$-\infty$, and therefore the quantities in \eqref{eq:UPi} and \eqref{eq:LPi} are well defined.

\begin{rmk}
	The restriction that the the extended quantity of interest corresponding to $\Phi$ be canonical is really no restriction, but is assumed only to simplify the presentation and notation.  Indeed, there are many important extended quantities of interest that are not affine as functions of the measure $\pi$.  However, all the ones that we have thought of can be handled by small modifications of the present framework, and their inclusion here would simply complicate the presentation and notation.  Moreover, note that many affine non-canonical extended quantities of interest become canonical through simple transformations.  For example, when $\Phi_{1}(\mu) := \mu[X \geq a]$ is a quantity of interest, and the extended quantity of interest is the probability that the system is unsafe, i.e.~$ \pi(\{\mu \mid \mu[X\geq a ]>\varepsilon\})$ where $\{\mu \mid \mu[X\geq a ]>\varepsilon\}$ is the set of unsafe $\mu$, then this extended quantity of interest is not canonical with respect to $\Phi_{1}$. However,  by transformation to $\Phi_{2}:=\eins_{\{r \mid r >\varepsilon \}}\circ \Phi_{1}$, the extended quantity of interest becomes canonical and $\mathcal{U}(\Pi)$ and $\mathcal{L}(\Pi)$, defined in terms of $\Phi_{2}$, are optimal upper and lower bounds on the probability that the system is unsafe given the set of priors $\Pi$.
\end{rmk}

\subsection{General Information Bounds on Prior Values}

Let $\delta \colon \mathcal{A} \to \mathcal{M}(\mathcal{A})$ be the mapping of points to unit Dirac measures, where $\delta_{\mu}$ denotes the Dirac mass at $\mu$, and, for $\Pi \subseteq \mathcal{M}(\mathcal{A})$, define
\begin{equation}
	\label{def-apiorigin}
	\mathcal{A}_{\Pi} := \delta^{-1}\Pi = \{  \mu \in \mathcal{A} \mid \delta_{\mu} \in \Pi \} .
\end{equation}
That is, $\mathcal{A}_{\Pi}$ consists of those scenarios $ \mu$ that are not only admissible in the sense that they lie in $\mathcal{A}$, but are also admissible as a prior in the sense that $\delta_{\mu}$ is an element of $\Pi$.

With the convention that $\mathcal{U}(\varnothing) := -\infty$ and $\mathcal{L}(\varnothing) := +\infty$, the following theorem shows the relationships among $\mathcal{U}(\mathcal{A})$ and $\mathcal{U}(\mathcal{A}_{\Pi})$ as defined by \eqref{eq:defma1}, $\mathcal{L}(\mathcal{A})$ and $\mathcal{L}(\mathcal{A}_{\Pi})$ as defined by \eqref{eq:defma2},  and $\mathcal{U}(\Pi)$ and $\mathcal{L}(\Pi)$ as defined by \eqref{eq:UPi} and \eqref{eq:LPi}.

\begin{thm}\label{thm-barriersorigin}
	It holds true that
	\[
		\mathcal{U}(\mathcal{A}_\Pi) \leq \mathcal{U}(\Pi) \leq \mathcal{U}(\mathcal{A})
	\]
	and
	\[
		\mathcal{L}(\mathcal{A}) \leq \mathcal{L}(\Pi) \leq \mathcal{L}(\mathcal{A}_\Pi).
	\]
	Moreover, if $\mathcal{A}_\Pi$ is non-empty, then
	\[
		\mathcal{L}(\mathcal{A}) \leq  \mathcal{L}(\Pi)\leq \mathcal{L}(\mathcal{A}_\Pi) \leq \mathcal{U}(\mathcal{A}_\Pi) \leq \mathcal{U}(\Pi) \leq \mathcal{U}(\mathcal{A}).
	\]
\end{thm}

\subsection{Priors Specified by Marginals}
\label{Sec:finitdimmarginals}

In many settings, probability measures or sets of probability measures are specified through generalized moments or other properties of marginal distributions. To analyse this case, let $\mathcal{Q}$ be a topological space and consider a measurable map $\Psi \colon \mathcal{A} \to \mathcal{Q}$.  Let us abuse notation by also denoting the corresponding pushforward of measures
$\Psi \colon \mathcal{M}(\mathcal{A})\rightarrow \mathcal{M}(\mathcal{Q})$
by the same symbol
$\Psi$. For a probability measure $\mathbb{Q} \in \mathcal{M}(\mathcal{Q})$, let
\[
    \Psi^{-1}\mathbb{Q} := \{ \pi \in \mathcal{M}(\mathcal{A}) \mid \Psi \pi=\mathbb{Q} \}
\]
be the set of probability measures $\pi \in \mathcal{M}(\mathcal{A})$ that push forward to $\mathbb{Q}$.  More generally,
for a non-empty set $\mathfrak{Q} \subseteq \mathcal{M}(\mathcal{Q})$, let
\begin{equation}
	\label{def_pullback}
	\Psi^{-1}\mathfrak{Q}:= \left\{ \pi \in \mathcal{M}(\mathcal{A})\mid \Psi\pi \in \mathfrak{Q} \right\}
\end{equation}
be the set of  probability measures $\pi \in \mathcal{M}(\mathcal{A})$ such that $\Psi\pi \in \mathfrak{Q}$.  Now, let $\mathfrak{Q} \subseteq \mathcal{M}(\mathcal{Q})$ be an admissible set of $\Psi$-marginals.  Then the corresponding admissible set of priors is $\Psi^{-1}\mathfrak{Q} \subseteq \mathcal{M}(\mathcal{A})$ and the corresponding objects to be computed are $\mathcal{U}(\Psi^{-1}\mathfrak{Q})$ and $\mathcal{L}(\Psi^{-1}\mathfrak{Q})$ according to \eqref{eq:UPi} and \eqref{eq:LPi}.

We will now demonstrate how to reduce the computation of $\mathcal{U}(\Psi^{-1}\mathfrak{Q})$ and $\mathcal{L}(\Psi^{-1}\mathfrak{Q})$ when $\mathfrak{Q}$ is specified by linear inequalities.  Later, in Section~\ref{sec-nestedreduction}, we will develop a more powerful \emph{nested} reduction which will provide the foundation for our reduction methods.

Before we begin, we need to introduce some terminology. Following Winkler \cite{Winkler:1988}, let $\mathcal{Y}$ be a topological space and let $\mathcal{M} \subseteq \mathcal{M}(\mathcal{Y})$ be a convex set
of measures. Let $\ext(\mathcal{M})$ denote the set of extreme points of $\mathcal{M}$ and let the evaluation field $\Sigma(\ext(\mathcal{M}))$ be the smallest $\sigma$-algebra of subsets of $\ext(\mathcal{M})$ such that the evaluation map $\nu \mapsto \nu(B)$ is measurable for all $B \in \mathcal{B}(\mathcal{Y})$.  Then a measure $\nu \in \mathcal{M}(\mathcal{Y})$ is said to be a \emph{barycenter} of $\mathcal{M}$ if there exists a probability measure $p$ on $\Sigma(\ext(\mathcal{M}))$ such that the \emph{barycentric formula}
\begin{equation}
	\label{def_barycenter}
	\nu(B)=\int_{\ext(\mathcal{M})}{\nu'(B) \, \mathrm{d} p(\nu')},\quad B \in \mathcal{B}(\mathcal{Y})
\end{equation}
holds.  Furthermore, the following notion of a \emph{measure affine function} is central to
Winkler's \cite{Winkler:1988} reduction theorems, which we use:

\begin{defn}
	\label{def_measureaffine}
	An extended real-valued function $F$ on $\mathcal{M} \subseteq \mathcal{M}(\mathcal{Y})$ is said to be \emph{measure affine} if, for all $\nu \in \mathcal{M}$ and all probability measures $p$ on $\Sigma(\ext(\mathcal{M}))$ for which the barycentric formula \eqref{def_barycenter} holds, $F$ is $p$-integrable and
	\[
		F(\nu) =\int_{\ext(\mathcal{M})}{F(\nu') \, \mathrm{d}p(\nu')} .
	\]
\end{defn}

A major consequence of Assumption \ref{ass_semibounded}, that $\Phi$ is semibounded, is that $\E_{\nu}[\Phi]$ exists, with possible values $\infty$ and $-\infty$, for all finite measures $\nu$.  As a consequence, by \cite[Prop.~3.1]{Winkler:1988}, the extended-real-valued function $\nu \mapsto \E_{\nu}[\Phi]$ is measure affine.

\subsubsection{Primary Reduction for Prior Values}
\label{sec-primaryreduction}

Let us consider the computation of
\begin{equation}
	\label{eq_lada}
	\mathcal{U}(\Psi^{-1}\mathfrak{Q}) =\sup_{\pi \in \Psi^{-1}\mathfrak{Q}} \E_{\pi}[\Phi]
\end{equation}
when $\mathfrak{Q}$ is specified by $n$ generalized moment inequalities determined by measurable functions $g_{1}, \dotsc, g_{n}$.
The situation for the lower bound $\mathcal{L}(\Psi^{-1}\mathfrak{Q})$ is the same.  That is, let $I_{1}, \dotsc, I_{n}$
be $n$ closed intervals, allowing semi-infinite intervals $(-\infty,q_{i}]$ and $[q_{i}, \infty)$,  and define
\[
	\mathfrak{Q} = \left\{ \mathbb{Q} \in \mathcal{M}(\mathcal{Q}) \smid \E_{\mathbb{Q}}[g_{i}] \in I_{i} \text{ for } i=1, \dotsc, n \right\} ,
\]
where implicit in the definition is that all $n$ integrals exist.  Then, by a change of variables,
$\E_{\Psi\pi}[g_{i}]=\E_{\pi}[g_{i}\circ \Psi]$ holds if either integral exists
(see e.g.~\cite[Cor.~19.2]{Bauer}), so we conclude that
\begin{align*}
	\Psi^{-1}\mathfrak{Q}
	&:=\left\{ \pi \in \mathcal{M}(\mathcal{A})\mid \Psi\pi \in \mathfrak{Q}\right\}\\
	&\phantom{:}=\left\{ \pi \in \mathcal{M}(\mathcal{A})\smid \E_{\Psi\pi}[g_{i}]\in I_{i} \text{ for } i=1,,n\right\}\\
 	&\phantom{:}=\left\{ \pi \in \mathcal{M}(\mathcal{A})\smid \E_{\pi}[g_{i} \circ \Psi]\in  I_{i} \text{ for } i=1, \dotsc, n\right\} .
\end{align*}
Hence,
$\Psi^{-1}\mathfrak{Q}$ is defined by the $n$ generalized moment inequalities corresponding to $g_{i} \circ \Psi \colon \mathcal{A} \to \R$ for  $i =1, \dotsc, n$.
Consequently, since the function
$\pi \mapsto \E_{\pi}[\Phi]$ is measure affine, it follows from the reduction theorems of \cite{OSSMO:2011} that we can reduce the supremum on the right-hand side of \eqref{eq_lada} to the convex combination of $n+1$ Dirac masses.  To state the theorem we have just proven, let
\begin{equation}
	\label{Dirac}
    \Delta(n) := \left\{\sum_{i=0}^n \alpha_i \delta_{\mu_i} \smid \mu_i\in \mathcal{A}, \alpha_i \geq 0, \text{for } i=0, \dots, n  \right\}.
\end{equation}
be  the set of non-negative combinations of $n+1$ Dirac masses.  Let  the vector $I$ of intervals have components $I_{i}$ for $i=1, \dotsc, n$, let
\[
	\Pi(I):=\Psi^{-1}\mathfrak{Q}
\]
be defined as above, and consider the subset
\begin{equation}
	\label{eq:uihue32Dirac}
	\Pi(I,n) := \Pi(I)\cap  \Delta(n)  \subseteq \Pi(I)
\end{equation}
of those measures which are the $(n+1)$-fold
convex combinations of Dirac masses.

\begin{thm}
	\label{thm:priorreduce}
    Let $\mathcal{A}$ be Suslin, let $\mathcal{Q}$ be separable and metrizable, and let $\Psi \colon \mathcal{A} \to \mathcal{Q}$ be measurable.  Moreover, for $n$ measurable functions $g_{1}, \dotsc, g_{n} \colon \mathcal{Q} \to \R$ and $n$ closed intervals $I_{1}, \dotsc, I_{n}$, let
	\[
		\mathfrak{Q}:=\left\{\mathbb{Q}\in \mathcal{M}(\mathcal{Q}) \smid \E_{\mathbb{Q}}[g_{i}]\in I_{i} \text{ for } i=1, \dotsc, n\right\}
	\]
	define the admissible set of $\Psi$-marginals.  Then,
	\[
		\mathcal{U}\big(\Pi(I)\big)=\mathcal{U}\big(\Pi(I,n)\big)
	\]
	where
    \begin{equation}
    	\label{eq:2qaltsu}
	 	\mathcal{U}\big(\Pi(I,n)\big)=
	 	\begin{cases}
	 		\sup \sum_{i=0}^n \alpha_i \Phi(\mu_i)\\
			\text{among } \mu_i\in \mathcal{A},\, \alpha_i\geq 0,\, \sum_{i=0}^n \alpha_i=1\\
			\text{such that } \sum_{i=0}^n \alpha_i g_{j}\bigl(\Psi(\mu_i)\bigr)\in I_{j} \text{ for } j=1, \dotsc, n.
		\end{cases}
    \end{equation}
\end{thm}

\begin{rmk}
	\label{inequalities}
	The freedom to determine intervals $I_{i}$, $i=1, \dotsc, n$, is one way to incorporate uncertainty and maintain a reduction to $n+1$ Dirac masses. In particular, by choosing semi-infinite intervals $I_{i}:=(-\infty,q_{i}]$ we obtain a reduction to $n+1$ Dirac masses for inequality constraints of the form $\E_{\mathbb{Q}}[g_{i}] \leq q_{i}$, and by choosing point intervals $I_{i}:=[q_{i},q_{i}]$ we obtain a reduction to $n+1$ Dirac masses for equality constraints of the form $\E_{\mathbb{Q}}[g_{i}]= q_{i}$.  Moreover, by choosing the interval to be semi-infinite or point interval depending on the index $i$ we obtain a reduction to $n+1$ Dirac masses for mixed equality and inequality constraints.
\end{rmk}

Theorem \ref{thm:priorreduce} can be put into a canonical form in the following way:  by considering the modified feature map $\Psi' \colon \mathcal{A} \to \R^{n}$ with components
\[
	\Psi'_{i}:=  g_{i}\circ \Psi, \quad \text{for } i=1, \dotsc, n ,
\]
it follows from the above that
\[
	\Psi^{-1}\mathfrak{Q}= \left\{ \pi \in \mathcal{M}(\mathcal{A}) \smid \E_{\pi}[\Psi']\in I\right\} .
\]
That is, by changing from the feature map $\Psi$ to $\Psi'$ we end up with a constraint set defined by the first moment of the  vector function $\Psi'$.  Therefore, let us remove the $'$ from $\Psi'$, and require $\Psi \colon \mathcal{A} \to \R^{n}$ to be measurable. The following theorem is the canonical form of Theorem~\ref{thm:priorreduce}. It is a corollary of Theorem~\ref{thm:priorreduce} for the constraint $\E_{\pi}[\Psi]\in Z$ when $Z=I$ is a closed rectangle. However, it is true for arbitrary $Z \subseteq \R^{n}$.

\begin{thm}\label{thm_primred}
	Let $\mathcal{A}$ be Suslin, let $\Psi \colon \mathcal{A} \to \R^{n}$ be measurable, let $Z\subset\R^n$, and let
	\begin{equation}
		\label{def-firstmomentbbisw}
		\mathfrak{Q}:=\left\{\mathbb{Q} \in \mathcal{M}(\R^{n}) \smid \E_{Q\sim \mathbb{Q}} [Q] \in Z \right\}
	\end{equation}
	be the set of those measures whose first moment belongs to $Z$. Then, for
	\begin{equation}
		\label{eq:uihue32}
        \Pi(Z) := \Psi^{-1}\mathfrak{Q} = \left\{ \pi \in \mathcal{M}(\mathcal{A}) \smid \E_{\pi}[\Psi]\in Z \right\}
	\end{equation}
	and $\Pi(Z,n) := \Pi(Z)\cap \Delta(n)$,	we have
	\[
		\mathcal{U}\big(\Pi(Z)\big)=\mathcal{U}\big(\Pi(Z,n)\big)
	\]
	where
	\begin{equation}
		\label{eq:2qaltsihuihu}
		\mathcal{U}\big(\Pi(Z,n)\big)=
		\begin{cases}
			\sup \sum_{i=0}^n \alpha_i \Phi(\mu_i)\\
			\text{among } \mu_i\in \mathcal{A}, \alpha_i\geq 0, \sum_{i=0}^n \alpha_i=1 \\
			\text{such that } \sum_{i=0}^n \alpha_i \Psi(\mu_i)\in Z .
		\end{cases}
	\end{equation}
\end{thm}

\begin{eg}
    \label{eg:shiva}
	Let $\mathcal{X}:=[0,1]$, $\mathcal{Q}=\R$ and consider the admissible set $\mathcal{A} := \mathcal{M}([0,1])$, the quantity of interest $\Phi(\mu):=\mu[X\geq a]$ for some $a\in (0,1)$, and the map $\Psi \colon \mathcal{A} \to \R$ defined by $\Psi(\mu):=\E_{\mu}[X]$.  Take as the set of admissible priors $\pi$ on $\mathcal{A}$ the collection
    \[
		\Pi := \left\{ \pi \in \mathcal{M}(\mathcal{A}) \smid \E_{\mu \sim \pi} \bigl[ \E_{\mu}[X] \bigr] = q \right\}
    \]
	for some fixed $q\in (0,a)$.  Then we will show that
	\begin{equation}
		\label{eq:sip1d2a}
		\mathcal{U}(\Pi)=q/a .
	\end{equation}
	To that end, observe that since $\E_{\mu \sim \pi} \bigl[ \E_{\mu}[X] \bigr] = \E_{\pi}[\Psi]$, it follows that
	\[
		\Pi = \left\{ \pi \in \mathcal{M}(\mathcal{A}) \smid \E_{\pi}[\Psi]  = q \right\} ,
	\]
	so that Theorem~\ref{thm_primred} implies that we can reduce the optimization in  $\mathcal{U}(\Pi)$ to the supremum over $\mu_1,\mu_2\in \mathcal{A}$, $\alpha \in [0,1]$ of
    \[
    	\alpha \mu_1[X\geq a]+(1-\alpha) \mu_2[X\geq a]
    \]
	subject to the constraint
    \[
        \alpha \E_{\mu_1}[X]+(1-\alpha) \E_{\mu_2}[X]=q .
    \]
	Introducing the slack variables $q_1:=\E_{\mu_1}[X]$, $q_2:=\E_{\mu_2}[X]$ and using \cite[Thm.~4.1]{OSSMO:2011} to reduce this problem further in $\mu_1,\mu_2$, we obtain that $\mathcal{U}(\Pi)$ is equal to the supremum over $\alpha \in [0,1]$ and $q_1, q_2 \in [0,1]$ of
    \[
    	\alpha \min \{ 1,\tfrac{q_1}{a} \} + (1-\alpha) \min \{ 1,\tfrac{q_2}{a} \}
    \]
    subject to the constraint $\alpha q_1 +(1-\alpha) q_2 =q$.  Observing that the supremum is achieved at $q_1,q_2 \leq a$, we conclude that $\mathcal{U}(\Pi)=q/a$, establishing \eqref{eq:sip1d2a}. 	Moreover, note that $\mathcal{U}(\Pi)=\mathcal{U}(\mathcal{A}_\Pi)$ for $\mathcal{A}_\Pi$ defined in \eqref{def-apiorigin} instead of the general inequality $\mathcal{U}(\mathcal{A}_\Pi) \leq \mathcal{U}(\Pi)$ of Theorem \ref{thm-barriersorigin}.
\end{eg}

\subsubsection{Nested Reduction for Prior Values}
\label{sec-nestedreduction}

The result of Example~\ref{eg:shiva} can also be deduced through a \emph{nested} reduction that we will find generally more useful for two reasons. The first is that, in practice, not only is it highly non-trivial to specify a prior on the space $\mathcal{A}$, since it requires quantifying information on an infinite-dimensional space, but it may also be undesirable to do so.  Indeed, if an expert does not have  a prior on the full space $\mathcal{A}$ but only on some projection $\Psi(\mathcal{A})=\mathcal{Q}$, then, rather than arbitrarily picking one particular prior on the space $\mathcal{A}$ compatible with the specified prior on $\Psi(\mathcal{A})$, it might be preferable to work with the set of priors on $\mathcal{A}$ specified through such  marginals.  Our second and main motivation is that, even when we can do the reduction on the primary space $\mathcal{M}(\mathcal{A})$, the reduced space remains so large that it may not be amenable to computation.  However with the nested reduction theorems given below, the reduced space becomes computationally manageable for finite-dimensional $\mathcal{Q}$.

\begin{eg}
	Consider $\Phi(\mu):=\mu[X\geq a]$, where $a$ is thought of as a safety margin, $\Psi(\mu)=(\E_\mu[X],\operatorname{Var}_\mu[X])$,  $\mathcal{Q}=\R^2$, and $\mathfrak{Q}=\{\mathbb{Q}\}$, where $\mathbb{Q}$ corresponds to the uniform distribution on $[-1,1]\times [3,4]$.
	In that example, the expert has only ``the prior'' that the mean of $X$ with respect to $\mu$ is uniformly distributed on $[-1,1]$ and that the variance of $X$ with respect to $\mu$ is independent of its mean and uniformly distributed on $[3,4]$. Observe that in this situation  $\mathfrak{Q}$ does not uniquely  specify a prior $\pi \in \mathcal{M}(\mathcal{A})$ but an infinite-dimensional set of priors $\Psi^{-1}(\mathfrak{Q}) \subseteq \mathcal{M}(\mathcal{A})$ and a robust approach would require assessing the safety of the system under the whole set $\Psi^{-1}(\mathfrak{Q})$ rather than under a particular element $\pi$ of that set.
\end{eg}

\paragraph{Idea of the Nested Reduction.}
Roughly, the idea of the nested reduction is as follows.  To compute \eqref{eq_lada}, consider the induced
function
\[
	\mathcal{U} \circ \Psi^{-1} \colon \mathcal{Q} \to \R
\]
defined by
\[
	\bigl(\mathcal{U} \circ \Psi^{-1}\bigr)(q) := \mathcal{U} \big(\Psi^{-1}(q) \big)=\sup_{\mu\in \Psi^{-1}(q)}\Phi(\mu), \quad \text{for } q \in \mathcal{Q},
\]
where we use the notation of \eqref{eq:defma1}.  From this it is natural to consider
\[
	\E_{\mathbb{Q}}[\mathcal{U}\circ \Psi^{-1}], \quad \text{for } \mathbb{Q} \in  \mathfrak{Q}.
\]
Let $\mathbb{Q} \in \mathfrak{Q}$.  Then, for any $\pi$ such that $\Psi\pi=\mathbb{Q}$, it follows that
\begin{align*}
	\E_{\mathbb{Q}}[\mathcal{U}\circ \Psi^{-1}]
	&=\E_{\Psi\pi}[\mathcal{U}\circ \Psi^{-1}]\\
	&=\E_{\pi}[\mathcal{U}\circ \Psi^{-1} \circ \Psi]
\end{align*}
Unfortunately, it is not true that $\mathcal{U}\circ \Psi^{-1} \circ \Psi=\Phi$;  instead it is $\bigl(\mathcal{U}\circ \Psi^{-1} \circ \Psi\bigr)(\mu) =\sup_{\mu': \Psi(\mu')=\Psi(\mu)}{\Phi(\mu')}$.
However, if  it were true, then we would obtain
\begin{align*}
	\E_{\mathbb{Q}}[\mathcal{U}\circ \Psi^{-1}]
	&= \E_{\Psi\pi}[\mathcal{U}\circ \Psi^{-1}]\\
	&= \E_{\pi}[\mathcal{U}\circ \Psi^{-1} \circ \Psi]\\
	&= \E_{\pi}[\Phi]
\end{align*}
and conclude that
\[
	\sup_{\mathbb{Q} \in \mathfrak{Q}}\E_{\mathbb{Q}}[\mathcal{U}\circ \Psi^{-1}] =
	\sup_{\pi \in \Psi^{-1}\mathfrak{Q}}\E_{\pi}[\Phi]
	= \mathcal{U}(\Psi^{-1}\mathfrak{Q}) .
\]
We will show that, despite the fact that $\mathcal{U}\circ \Psi^{-1} \circ \Psi\neq \Phi$, the conclusion
\begin{equation}
	\label{eq-nested}
	\mathcal{U}(\Psi^{-1}\mathfrak{Q})=  \sup_{\mathbb{Q} \in \mathfrak{Q}} \E_{\mathbb{Q}}[\mathcal{U}\circ \Psi^{-1}]\,
\end{equation}
is still valid, provided that it is interpreted correctly.  Heuristically, the reason for this is that the supremum $\sup_{\pi \in \Psi^{-1}\mathfrak{Q}}$ in $\mathcal{U}(\Psi^{-1}\mathfrak{Q})$ is exploring the maximum value of $\Phi$ on level sets of $\Psi$ very much like the supremum in $\bigl(\mathcal{U}\circ \Psi^{-1}\bigr)(q)= \sup_{\Psi^{-1}(q)}{\Phi}$.

If $\mathcal{A}$ is such that a reduction theorem, e.g.~from \cite{OSSMO:2011}, applies to reduce the computation of the inner supremum in $\mathcal{U}\circ \Psi^{-1}$ to the supremum over convex combinations of Dirac masses, and the admissible set $\mathfrak{Q}$ is such that a reduction theorem applies to the computation of the  outer supremum of $\sup_{\mathbb{Q} \in \mathfrak{Q}} \E_{\mathbb{Q}}[\mathcal{U}\circ \Psi^{-1}]$, then the identity \eqref{eq-nested} represents a nesting of reductions.

Let us now establish a result like \eqref{eq-nested}.  To do so will require addressing three questions:  (1) What kind of function is $\mathcal{U}\circ \Psi^{-1}$?  (2) What kind of measures $\mathbb{Q} \in \mathcal{M}(\mathcal{Q})$ can define an integral of a function with properties discovered from the answer to (1)?  (3) Can we obtain a measurable
solution operator to the optimization problem $\bigl(\mathcal{U}\circ \Psi^{-1}\bigr)(q)$, where $q \in \mathcal{Q}$?
To that end, let us first recall a definition of universally measurable functions.

\begin{defn}
	Let $(T, \mathcal{T})$ be a measurable space, and for a positive measure $\nu$ on $(T, \mathcal{T})$,  let $\mathcal{T}_\nu$ denote the $\nu$-completion of $\mathcal{T}$. Let  $\widehat{\mathcal{T}}:=\bigcap_{\nu} \mathcal{T}_\nu$, where the intersection is over all positive bounded measures $\nu$, denote the universally measurable sets. A $\widehat{\mathcal{T}}$-measurable function is said to be \emph{universally measurable}.
\end{defn}

At the heart of the commutative representation used for the nested reduction is the following optimal measurable selection lemma answering questions (1) and (3) above:

\begin{lem}
	\label{lem:Umeasurable}
	Let $\mathcal{A}$ be a Suslin space, let $\mathcal{Q}$ be a separable and metrizable space, and let $\Psi \colon \mathcal{A} \to \mathcal{Q}$ be measurable.  Then, for any subset $T \subseteq \Psi(\mathcal{A})$,
	\begin{enumerate}
		\item $\mathcal{U} \circ \Psi^{-1}$ is $\widehat{\mathcal{B}}(T)$-measurable
		\item for all $\delta > 0$, there exists a $\delta$-suboptimal
		$\widehat{\mathcal{B}}(T)$-measurable section of $\Psi$; that is, a $\widehat{\mathcal{B}}(T)$-measurable function $\psi \colon T \to \mathcal{A}$ such that $\Psi\big(\psi(q)\big)=q$ for all $q \in T$ and
		\[
			\Phi\big(\psi(q)\big) > \mathcal{U} \big(\Psi^{-1}(q) \big)-\delta, \quad \text{for all } q \in T.
		\]
\end{enumerate}
\end{lem}

To answer question (2) above, define a \emph{support} $\supp (\mathbb{Q})$ of a measure $\mathbb{Q} \in \mathcal{M}(\mathcal{Q})$, as in \cite[Ch.~12.3]{AliprantisBorder:2006}, to be a closed set such that
\begin{itemize}
	\item $\mathbb{Q}\bigl( \mathcal{Q} \setminus \supp(\mathbb{Q}) \bigr) = 0$, and
	\item if $G \subseteq \mathcal{Q}$ is open and $G \cap \supp(\mathbb{Q}) \neq \varnothing$, then $\mathbb{Q}(G \cap \supp(\mathbb{Q})) > 0$.
\end{itemize}
When $\mathcal{Q}$ is a separable and metrizable space, it follows that it is second countable and
therefore, by \cite[Thm.~12.14]{AliprantisBorder:2006}, all $\mathcal{Q}\in \mathcal{M}(\mathcal{Q})$ have a uniquely defined support.  Now consider a measure $\mathbb{Q} \in \mathcal{M}(\mathcal{Q})$ such that $\supp(\mathbb{Q}) \subseteq \Psi(\mathcal{A})$.  Then, by Lemma~\ref{lem:Umeasurable}, $\mathcal{U}\circ \Psi^{-1}$ is $\widehat{\mathcal{B}}(\supp \, \mathbb{Q})$-measurable.  Therefore, the expected value $\E_{\mathbb{Q}}[\mathcal{U}\circ \Psi^{-1}]$ can be defined
by integration with respect to the completion $\widehat{\mathbb{Q}}$:
\begin{equation}
	\label{def-expectation}
	\E_{\mathbb{Q}}[\mathcal{U}\circ \Psi^{-1}]:= \E_{\widehat{\mathbb{Q}}}[\mathcal{U}\circ \Psi^{-1}] .
\end{equation}
More generally, for any universally measurable function $f$ and any finite measure $\mathbb{Q}$, we define the expected value $\E_{\mathbb{Q}}[f]$ of $f$ by
\begin{equation}
	\label{def-expectation2}
	\E_{\mathbb{Q}}[f]:= \E_{\widehat{\mathbb{Q}}}[f]  .
\end{equation}
Such a method of defining integrals of, possibly non-Borel measurable, but universally measurable, functions brings up many questions such as:  when is it uniquely defined?; for a fixed integrand, when is the expectation operation affine in the measure?;  does it have a change a variables formula?  All such questions have nice answers and, although we are sure that this is classical, we cannot find a reference for these facts so we have included statements and proofs of the facts needed in this paper in  Section~\ref{sec-universal} of the Appendix.

We now state our nested reduction theorem of the form \eqref{eq-nested}:

\begin{thm}
	\label{thm:reducpriormarg}
	Let $\mathcal{A}$ be a Suslin space, let $\mathcal{Q}$ be a separable and metrizable space, and let	$\Psi \colon \mathcal{A} \to \mathcal{Q}$  measurable.  Moreover, let	$\mathfrak{Q} \subseteq \mathcal{M}(\mathcal{Q})$ be such that $\supp (\mathbb{Q}) \subseteq \Psi(\mathcal{A})$ for all $\mathbb{Q} \in \mathfrak{Q}$.  Then, for each $\mathbb{Q} \in \mathfrak{Q}$, $\Psi^{-1}\mathbb{Q}$ is non-empty.  Moreover, the upper bound $\mathcal{U}(\Psi^{-1}\mathfrak{Q})$, defined in \eqref{eq:UPi}, satisfies
	\begin{equation}
		\label{eq:2q}
		\mathcal{U}(\Psi^{-1}\mathfrak{Q}) = \sup_{\mathbb{Q}\in \mathfrak{Q}} \mathbb{E}_{\mathbb{Q}} \big[ \mathcal{U} \circ \Psi^{-1}  \big].
	\end{equation}
	where the expectations on the right-hand side are defined as in \eqref{def-expectation}.  Finally, the expectation operator on the right-hand side is measure affine in $\mathbb{Q}$.
\end{thm}
\begin{rmk}
Note that \eqref{eq:2q} can be written
	\begin{equation}
		\label{eq:2qrewrite}
		\sup_{\pi \in \Psi^{-1}\mathfrak{Q}} \E_{\mu \sim \pi}\big[\Phi(\mu)\big] = \sup_{\mathbb{Q}\in \mathfrak{Q}}\Big[\E_{q\sim \mathbb{Q}}
\big[\sup_{\mu\in \Psi^{-1}(q)}\Phi(\mu) \big]\Big].
	\end{equation}

\end{rmk}

\begin{rmk}
	Since the right-hand side is measure affine in $\mathbb{Q}$, if $\mathbb{Q}$ is specified through (multi-)linear generalized moment inequalities, then the reduction theorems of \cite{OSSMO:2011} can be applied to obtain the supremum over $\mathbb{Q}$ by reducing $\mathbb{Q}$ to a convex combination of a finite number of Dirac masses on $\mathcal{Q}$.  Moreover, if $\mathfrak{Q}$ consists of a single element, i.e.~$\mathfrak{Q}=\{\mathbb{Q}\}$,  then
	\begin{equation}
		\label{eq:2qbiis}
		\mathcal{U}(\Psi^{-1}\mathfrak{Q}) =
		\mathcal{U}(\Psi^{-1}\mathbb{Q}) =
  		\mathbb{E}_{\mathbb{Q}} \bigl[ \mathcal{U} \circ \Psi^{-1}  \bigr],
	\end{equation}
	and the right hand-side of \eqref{eq:2qbiis} can be approximately
	evaluated via Monte Carlo sampling of $q\in \mathcal{Q}$ according to the measure $\mathbb{Q}$.
\end{rmk}

\begin{rmk}
	A similar theorem can obtained for the optimal lower bound $\mathcal{L}(\Psi^{-1}\mathfrak{Q})$. Throughout this paper, results given for optimal upper bounds $\mathcal{U}$ can be translated into results for optimal lower bounds $\mathcal{L}$ by considering the negative quantity of interest $-\Phi$ and for the sake of concision we will not write those results unless necessary.
\end{rmk}

\begin{eg}
	\label{eg:shiva2}
	Consider again Example~\ref{eg:shiva}, where $\mathcal{X}:=[0,1]$, $\mathcal{Q}=\R$, the admissible set $\mathcal{A} := \mathcal{M}([0,1])$, the quantity of interest $\Phi(\mu):=\mu[X\geq a]$ for some $a\in (0,1)$, the map $\Psi \colon \mathcal{A} \to \R$ is defined by $\Psi(\mu):=\E_{\mu}[X]$, and the set of admissible priors $\pi$ on $\mathcal{A}$ is the collection
    \[
    	\Pi := \left\{ \pi \in \mathcal{M}(\mathcal{A}) \smid \E_{\mu \sim \pi} \left[ \E_{\mu}[X] \right] = q \right\}.
    \]
    for some fixed $q\in (0,a)$.  We will now demonstrate how the result $\mathcal{U}(\Pi)=q/a$ of \eqref{eq:sip1d2a} obtained by the primary reduction follows from the nested reduction theorem.  To that end, observe that since $\Psi(\mathcal{A})=[0,1] \subseteq \R$, by restricting to measures $\mathbb{Q} \in \mathcal{M}(\R)$	with support $\supp(\mathbb{Q}) \subseteq [0,1]$, Theorem~\ref{thm:reducpriormarg} implies that
	\begin{equation}
		\label{eg-la}
		\mathcal{U}(\Pi) = \sup_{\mathbb{Q}\in \mathfrak{Q}} \mathbb{E}_{q' \sim \mathbb{Q}} \left[ \sup_{\mu \in \mathcal{M}([0,1])\,:\, \E_\mu[X]=q' } \mu[X\geq a] \right],
	\end{equation}
	where $\mathfrak{Q}$ is the set of probability measures $\mathbb{Q}$ on $\R$ with support contained in $[0,1]$ such that $\E_{\mathbb{Q}}[Q]=q$.  Theorem~4.1 of \cite{OSSMO:2011} shows that the inner supremum of $\mu[X\geq a]$ can be achieved by assuming that $\mu$ is the weighted sum of two Dirac masses, i.e.
	\begin{equation}
		\label{eq:supex1d}
		\sup_{\substack{ \mu \in \mathcal{M}([0,1]) \\ \E_\mu[X]=q' }} \mu[X\geq a] =
		\sup_{\substack{ \alpha, x_1, x_2\in [0,1] \\ \alpha x_1+(1-\alpha) x_2=q' }} (\alpha \delta_{x_1}+(1-\alpha)\delta_{x_2})[X\geq a].
	\end{equation}
	For $q'> a$, the supremum in the right-hand side of \eqref{eq:supex1d} is $1$, and for $q'\leq a$, the supremum in the right-hand side of \eqref{eq:supex1d} is achieved by $x_2=0$, $x_1=a$ and $\alpha=q'/a$, and so we conclude that
	\[
		\sup_{\substack{ \mu \in \mathcal{M}([0,1]) \\ \E_\mu[X]=q' }} \mu[X\geq a] = \min \{ 1, \tfrac{q'}{a} \}.
	\]
	Hence, by identifying the measures $\mathbb{Q} \in \mathcal{M}(\R)$ with support $\supp(\mathbb{Q}) \subseteq [0,1]$ with $\mathcal{M}([0,1])$ in the obvious way, \eqref{eg-la} becomes
	\begin{equation}
		\label{eq:sip1d2}
		\mathcal{U}(\Pi) = \sup_{\substack{ \mathbb{Q}\in \mathcal{M}([0,1]) \\ \E_{\mathbb{Q}}[Q]=q }} \mathbb{E}_{q'\sim \mathbb{Q}} \left[ \min \{ 1, \tfrac{q'}{a} \} \right].
	\end{equation}
	Using \cite[Thm.~4.1]{OSSMO:2011} again, we obtain that the supremum in $\mathbb{Q}$ in the right-hand side of \eqref{eq:sip1d2} is equal to the supremum over $\alpha, q_1, q_2 \in[0,1]$, of
	\begin{equation}
		\label{eq:sip1d3}
		\alpha \min \{ 1, \tfrac{q_1}{a} \} + (1 - \alpha) \min \{ 1, \tfrac{q_2}{a} \}
	\end{equation}
	subject to the constraint that $\alpha q_1 + (1 - \alpha) q_2 = q$.  This supremum is achieved by $q_1 = a$, $q_2 = 0$ and $\alpha = \tfrac{q}{a}$, and so we obtain that $\mathcal{U}(\Pi)=q/a$, in agreement with \eqref{eq:sip1d2a}.
\end{eg}

\section{Optimal Bounds on the Posterior Value}
\label{Sec:OpBoundsPosteriorEstimates}

What happens to the optimal bounds \eqref{eq:UPi} and \eqref{eq:LPi} on the prior value $\E_{\pi}[\Phi]$, investigated in Section \ref{Sec:OpBoundsPriorEstimates}, after conditioning on the data?  Does the interval corresponding to these optimal bounds shrink down to a single point as more and more data comes in?  Does this interval shrink as the measurement noise on the data is reduced?  What happens to posterior estimates associated with two distinct but close priors, possibly sharing the same marginal distribution on a high dimensional space?  These are the questions that will be investigated in this section.  Our answers will show that: (1) optimal bounds on posterior estimates \emph{grow} as data comes in;  (2) optimal bounds on posterior estimates \emph{grow} as measurement noise is reduced (3) two priors sharing the same high-dimensional marginals can lead to \emph{diametrically opposed} posterior estimates. In some sense these results can be seen as extreme occurrences of the  dilation property observed in  robust Bayesian inference  \cite{WassermanSeidenfeld:1994}.

As discussed in Section~\ref{subsec:dataspacesanmaps}, let us now consider the case where the probability distribution of the data is a known function $\Dmap(\mu)$ of the admissible candidates $\mu \in \mathcal{A}$.
As shown in Section \ref{sec:Plan}, directly conditioning measures $\pi\odot \Dmap$ with respect to the random variable  $\Drv$ representing the observed sample data  would require manipulating regular conditional probabilities on  $\mathcal{A} \times \Dspace$.

Furthermore, in Bayesian statistics a prior $\pi$ may represent a ``subjective belief'' about reality and, in such situations, the data may be sampled from $\pi^{\dagger}\cdot \Dmap^{\dagger}$ which may be distinct from $\pi \cdot \Dmap$.  In frequentist analyses of
Bayesian statistics $\pi^{\dagger}$ is called the ``true''  prior, or ``data-generating distribution'',
and $\pi$ a ``subjective'' prior (see \cite{Berger:2006} and references therein).
Although it is known that the subjective prior $\pi$ might be distinct from the true prior $\pi^{\dagger}$, one may still try to evaluate the conditional expectation of the quantity of interest $\Phi$ using $\pi$ as the distribution on $\mathcal{A}$.  We will show here that although the observation of the sample data $\Ddata$ does not uniquely determine the true prior $\pi^{\dagger}$, it does determine a random subset of $\mathcal{M}(\mathcal{A})$ (i.e.~a random subset of priors) denoted $\mathcal{R}(\Ddata)$ such that, $\pi^{\dagger}$-a.s.,
$\pi^{\dagger}\in \mathcal{R}(\Ddata)$.
 This observation is based on the following fundamental lemma:

\begin{lem}
	\label{lem_positivemass}
	For a strongly Lindel\"{o}f space $\mathcal{Y}$ and a Borel measure $\nu$ on $\mathcal{B}(\mathcal{Y})$, define
	\[
		E := \left\{ y\in \mathcal{Y} \smid \begin{array}{c} \text{there is an open neighborhood $\mathcal{O}_{y}$} \\ \text{of $y$ such that $\nu(\mathcal{O}_{y})=0$} \end{array} \right\} .
	\]
	Then $\nu(E)=0$.
\end{lem}

\begin{rmk}
	Recall that a Lindel\"{o}f space is a topological space such that any open cover has a countable subcover and a strongly Lindel\"{o}f space is such that any open subset is Lindel\"{o}f.  Since $\Dspace$ is assumed to be Suslin from Section~\ref{subsec:dataspacesanmaps}, and Suslin implies strongly Lindel\"{o}f, Lemma~\ref{lem_positivemass} shows that any open neighborhood $B_{d}$ of any observed value $d \in \Dspace$ has nonzero measure with probability $1$.
\end{rmk}

\begin{rmk}
	Any separable Hilbert space, in particular the Euclidian space $\R^{k}$, is strongly Lindel\"{o}f.  In this situation, Lemma~\ref{lem_positivemass} implies that if for any observation $y$ generated by a law $\nu\in\mathcal{M}(\mathcal{Y})$ we place an open ball $B_{r(y)}(y)$ of non-zero radius $r(y)>0$ about $y$, then with $\nu$-probability $1$ we have $\nu\bigl(B_{r(y)}(y)\bigr)>0$. That is,
	\[
		\nu \bigl( \{ y\in \mathcal{Y} \mid \nu (B_{r(y)}(y)) > 0 \} \bigr)=1  .
	\]
\end{rmk}

Now suppose the data $d$ are generated according to a probability measure $\pi^{\dagger} \cdot \Dmap$ (where $\pi^{\dagger}$ is the ``true'' prior).  We conclude from Lemma~\ref{lem_positivemass} that when we observe a sample $d$, if we assume that $\pi^{\dagger} \in \mathcal{R}(d)$ where
\[
	\mathcal{R}(d):=\left\{\pi \in \mathcal{M}(\mathcal{A}) \smid \vphantom{\big|} \pi\cdot \Dmap[B]>0 \text{ for all $B$ open containing $d$}\right\},
\]
then we will be correct in this assumption with $\pi^{\dagger}\cdot \Dmap$-probability $1$.
Therefore, when the data $d$ are generated and we observe that $d\in B_d$ where $B_d$ is an open subset containing the data $d$ (to keep our notation simple, we will, later on, drop $d$ in the notation $B_d$), then we restrict our attention to priors $\pi \in \Pi$  such that $\pi \cdot \Dmap[B_d]>0$.  That is to say, we restrict our attention to the intersection of $\Pi $ with the set of priors $\pi$ such that
$\pi \in \mathcal{M}(\mathcal{A})$ and $\pi \cdot \Dmap[B_d]>0$.  We write $\Pi_{B_d}$ for this intersection, i.e.
\[
	\Pi_{B_d}:= \left\{ \pi \in \Pi \smid \pi \cdot \Dmap[B_d]>0 \right\}.
\]
If $\Pi_{B_d}$ is void,
then we assert that ``$\pi^{\dagger}$ is not contained in  $\Pi$'' and we know that this assertion is true  with $\pi^{\dagger} \cdot \Dmap$-probability $1$ on the realization of the data $d$.  Conversely, if  $\pi^{\dagger}$ is contained in  $\Pi$, then $\Pi_{B_d} $ must, with $\pi^{\dagger} \cdot \Dmap$-probability $1$ on the realization of the data $d$, still contain $\pi^{\dagger}$ (in particular it must be non-empty).

Happily, this approach also facilitates the efficient computation of the conditional expectations because now they have
a simple representation.  Indeed, consider the conditional expectation of an object of interest $\Phi$ given a prior $\pi$ and data map $\Dmap$, conditioned on a subset $B \in \mathcal{B}(\Dspace)$ such that $\pi\cdot \Dmap[B]>0$.  It follows from \eqref{eq:palm} and \eqref{eq:cdotexp} that the conditional expectation of $\Phi$ given $B$  is
\[
	\E_{\pi \odot \Dmap}\big[ \Phi\big|B\big]:=\frac{\E_{(\mu,d) \sim \pi \odot \Dmap}\big[  \Phi(\mu) \one_{B}(d) \big]}{\pi\cdot \Dmap[B]},
\]
which, using \eqref{eq:palm} and \eqref{eq:cdotexp}, leads to
\begin{equation}
	\label{eq:defcondexp}
	\E_{\pi \odot \Dmap}\big[ \Phi\big| B\big]=\frac{\E_{\mu\sim \pi}\big[  \Phi(\mu)
 \Dmap(\mu)[B] \big]}{\E_{\mu\sim\pi}\big[\Dmap(\mu)[B]\big]} .
\end{equation}
Moreover, recall that this conditional expectation is the best mean squared approximation of $\Phi$ under the measure $\pi \odot \Dmap$, given the information that $\Drv\in B$, i.e.
\begin{equation}
	\label{eq:orthproj}
	\E_{\pi \odot \Dmap}\big[ \Phi\big| B\big]=
	\argmin_{m \in \R} \E_{\pi \odot \Dmap}\Big[ \big(\Phi-m\big)^2\Big|B\Big].
\end{equation}

Consequently, for any open subset $B \subseteq \Dspace$, we define
\begin{equation}
	\label{def_B}
	\Pi_B := \left\{ \pi \in \Pi \smid (\pi \cdot \Dmap)[B]> 0 \right\}.
\end{equation}
where, by \eqref{eq:cdotexp},
\begin{equation}
        \label{eq:cdotexp2}
        \pi\cdot\Dmap[B]:=\E_{\mu\sim \pi}\big[\Dmap(\mu)[B]\big] .
\end{equation}
Then, since $(\pi \cdot \Dmap)[B]> 0$, the formula \eqref{eq:defcondexp} for conditional expectation implies that
\begin{equation}
	\label{eq:Upostes}
	\mathcal{U}(\Pi|B):=\sup_{\pi \in \Pi_B}\E_{\pi \odot \Dmap}\big[ \Phi\big| B\big]
\end{equation}
\begin{equation}
	\label{eq:Lpostes}
	\mathcal{L}(\Pi|B):=\inf_{\pi \in \Pi_B}\E_{\pi \odot \Dmap}\big[ \Phi\big| B\big]
\end{equation}
where
\begin{equation}
	\label{eq:defcondexp2}
    \E_{\pi \odot \Dmap}\big[ \Phi\big| B\big]=\frac{\E_{\mu \sim \pi}\big[  \Phi(\mu) \Dmap(\mu)[B] \big]}{\E_{\mu\sim\pi}\big[\Dmap(\mu)[B]\big]} .
\end{equation}
Finally, if $B$ is an open neighborhood containing the sample data $d$, then it follows that $\mathcal{U}(\Pi|B)$ and $\mathcal{L}(\Pi|B)$ are optimal upper and lower bounds on the posterior values $\E_{\pi\odot \Dmap} [ \Phi | B ]$, given the observation $\Drv\in B$, over all $\pi\in \Pi$ such that $\pi \cdot \Dmap[B]>0$.

\begin{eg}
	When $\Phi$ is the indicator function of the set $\{\mu \mid \mu[X\geq a ]>\epsilon\}$ (i.e.\ the set of unsafe $\mu$),  $\mathcal{U}(\Pi|B)$ and $\mathcal{L}(\Pi|B)$ are optimal upper and lower bounds on the ``posterior probability'' that the system is unsafe given the observation $\Drv\in B$ (and the set $\Pi$  of priors and observation maps respectively).
\end{eg}

\subsection{General Information Bounds on Posterior Values}

Now let $B \subseteq \Dspace$ be open and let
\begin{equation}
	\label{def-api}
	\mathcal{A}_{\Pi_B} := \left\{  \mu \in \mathcal{A} \smid \delta_{ \mu} \in \Pi \text{ and } \Dmap( \mu)[B] > 0 \right\} ,
\end{equation}
\[
	\mathcal{U}(\mathcal{A}_{\Pi_B}):=\sup_{\mu\in \mathcal{A}_{\Pi_B}}\Phi(\mu)\, ,
\]
and use $\mathcal{L}$ for the corresponding infimum.  The following theorem is a straightforward consequence of \eqref{eq:defcondexp}:

\begin{thm}
\label{thm-barriers}
	It holds true that
	\[
		\mathcal{U}(\mathcal{A}_{\Pi_B}) \leq \mathcal{U}(\Pi_B) \leq \mathcal{U}(\mathcal{A}),
	\]
	and
	\[
		\mathcal{L}(\mathcal{A}) \leq \mathcal{L}(\Pi_B) \leq  \mathcal{L}(\mathcal{A}_{\Pi_B}).
	\]
	Moreover, if $\mathcal{A}_{\Pi_B}$ is non empty, then
	\[
		\mathcal{L}(\mathcal{A}) \leq \mathcal{L}(\Pi_B) \leq  \mathcal{L}(\mathcal{A}_{\Pi_B})\leq \mathcal{U}(\mathcal{A}_{\Pi_B}) \leq \mathcal{U}(\Pi_B) \leq \mathcal{U}(\mathcal{A}).
	\]
\end{thm}

\begin{rmk}
	The dependence of $\mathcal{U}(\mathcal{A}_{\Pi_B})$ and $\mathcal{L}(\mathcal{A}_{\Pi_B})$ on the sample data is very weak.  In particular, if  $\Dmap$ corresponds to observing i.i.d.\ realizations of $(X+\xi,f^\dagger(X)+\xi')$ where $\xi$ and $\xi'$ are centered Gaussian random variables of arbitrarily small (non zero) variance, then it can be shown that $\mathcal{U}(\mathcal{A}_{\Pi_B}) = \mathcal{U}(\mathcal{A}_{\Pi})$ and $\mathcal{L}(\mathcal{A}_{\Pi_B}) = \mathcal{L}(\mathcal{A}_{\Pi})$.  In that situation, if $\mathcal{L}(\mathcal{A}_{\Pi}) < \mathcal{U}(\mathcal{A}_{\Pi})$, then  $\mathcal{U}(\mathcal{A}_{\Pi_B}) - \mathcal{L}(\mathcal{A}_{\Pi_B})$ remains bounded away from $0$ by a strictly positive constant that is independent of $\Dmaps$ and $B$, which, in particular, implies that the range of achievable posterior values cannot shrink towards $\Phi(\mu^\dagger)$ regardless of the number of observed i.i.d.\ samples.  The presence of such information bounds suggests
	that the consistency of Bayesian estimators cannot be established independently of (uniformly in)
	the choice of priors (this point will also be substantiated by Theorem \ref{thm:shiva}).
\end{rmk}

\subsection{Primary Reduction for Posterior Values}
\label{Sec:finitdimmargdeinalsalternate}

As in Section \ref{sec-primaryreduction}, when priors are specified through finite-dimensional inequalities, it is possible to provide a reduction of the computation of $\mathcal{U}(\Pi |B)$ on the primary space.  To that end, let $\mathcal{M}_+(\mathcal{A})$ denote the set of positive bounded measures on $\mathcal{A}$ and let us extend the ``expectation notation'' to mean integration with respect to a positive measure in the natural way: for a measurable function $\psi$ and a $ \pi_+ \in \mathcal{M}_+(\mathcal{A})$  define
\[
	\E_{\pi_+}[\psi] := \int_{\mathcal{A}} \psi \, \mathrm{d} \pi_+
\]
if the integral exists.

Let $\psi_0,\dotsc,\psi_n$ be real-valued measurable functions on $\mathcal{A}$ and define
\[
	\Pi_+:= \left\{\pi_+\in \mathcal{M}_+(\mathcal{A}) \smid \E_{\pi_+}[\psi_0]=1,\text{ and } \E_{\pi_+}[\psi_i]= 0\text{ for } i=1,\ldots, n \right\},
\]
where implicit in the definition is that all $n+1$ integrals exist, and let
\[
	\Pi_{+,n}:= \Pi_+ \cap \Delta(n)
\]
be the set of those measures in $\Pi_{+}$ that are non-negative sums of $n+1$ Dirac masses.  The following theorem is a  generalization of \cite[Thm.~4.1]{OSSMO:2011} to positive measures (see also \cite[Thm.~3.2]{Winkler:1988} from which the proof of \cite[Thm.~4.1]{OSSMO:2011} was derived).

\begin{thm}
	\label{thm:redpip}
	If $\mathcal{A}$ is a Suslin space, then
	\begin{equation}
		\label{eq:hsbiu3}
		\sup_{\pi_+\in \Pi_+} \E_{\pi_+}[\Phi]=\sup_{\pi_+\in \Pi_{+,n+1}} \E_{\pi_+}[\Phi] .
	\end{equation}
	Furthermore, if $\psi_0$ is non-negative on $\mathcal{A}$ and there exists a measurable function $\varphi$ such that $\Phi=\psi_0 \varphi$, then
	\begin{equation}
		\label{eq:hskjbiu3}
		\sup_{\pi_+\in \Pi_+} \E_{\pi_+}[\Phi]=\sup_{\pi_+\in \Pi_{+,n}} \E_{\pi_+}[\Phi] .
	\end{equation}
\end{thm}

Theorem \ref{thm:redpip} can be used to produce a primary reduction of $\mathcal{U}\big(\Pi\odot_B \Dspace\big)$ when  $\Pi$ is defined by a finite number of equalities.  To state the theorem, recall that, for
arbitrary $\Pi$ and $B$, the definition
\[
	\Pi_B := \left\{ \pi  \in \Pi \smid \pi \cdot \Dmap[B]> 0 \right\}
\]
of \eqref{def_B}, where by \eqref{eq:cdotexp2}
\[
	\pi\cdot\Dmap[B]:=\E_{\mu\sim \pi}\big[\Dmap(\mu)[B]\big] ;
\]
recall also the notation of \eqref{eq:Upostes}
\[
	\mathcal{U}(\Pi|B) := \sup_{\pi \in \Pi_B}\E_{\pi \odot \Dmap}\big[ \Phi\big| B\big];
\]
and recall the result \eqref{eq:defcondexp} that, for any $\pi   \in \Pi_B$,
\[
	\E_{\pi \odot \Dmap}\big[ \Phi\big| B\big]=\frac{\E_{\mu\sim \pi}\big[  \Phi(\mu) \Dmap(\mu)[B] \big]}{\E_{\mu\sim\pi}\big[\Dmap(\mu)[B]\big]} .
\]
The proof of the following theorem  is obtained by first proving the theorem for equality constraints $Z=\{q\}$,
by observing that $\mathcal{U}\big(\Pi(q)\big| B\big)$ is a linear fractional optimization problem in $\pi$ and utilizing the fact that such problems are equivalent to linear problems \cite{BoydVandenberghe:2004}, and then applying Theorem~\ref{thm:redpip}.  To extend the result to the subset $Z \subseteq \R^{n}$, one uses a layercake approach as in the proof of Theorem~\ref{thm_primred}.  As in Section~\ref{Sec:OpBoundsPriorEstimates},
the following primary reduction theorem, Theorem~\ref{thm:alternredded}, will be formulated in canonical form and the nested reduction theorem, Theorem~\ref{thm:sivahidden}, will be in the general form.

\begin{thm}
	\label{thm:alternredded}
	Let $\mathcal{A}$ be Suslin and let $\Psi \colon \mathcal{A} \to \R^{n}$ be measurable.  For $Z \subseteq \R^{n}$, let $\Pi(Z) := \{ \pi \in \mathcal{M}(\mathcal{A}) \mid \E_{\pi}[\Psi]\in Z \}$.  Then  $\mathcal{U}\big(\Pi(Z)\big|B\big)$ is equal to the supremum over  $\alpha_{i} \geq 0$, $q\in Z$ and $ \mu_{i} \in \mathcal{A}$ of
	\[
		\sum_{i = 0}^{n } \alpha_{i} \Phi(\mu_{i}) \Dmap( \mu_{i})[B]
	\]
	subject to the constraints
	\[
		\sum_{i = 0}^{n} \alpha_{i} \big(\Psi( \mu_{i})-q\big) = 0
	\]
	and
	\begin{equation}
		\label{eq:polyconslkde}
		\sum_{i = 0}^{n} \alpha_{i} \Dmap( \mu_{i})[B] = 1.
	\end{equation}
\end{thm}

\begin{eg}
	\label{eg:shivabis}
	Consider again Example \ref{eg:shiva} with the admissible set $\mathcal{A} := \mathcal{M}([0,1])$, the quantity of interest $\Phi(\mu):=\mu[X\geq a]$, the map $\Psi(\mu):=\E_\mu[X]$ and the set of admissible priors
	\[
		\Pi := \left\{ \pi \in \mathcal{M}(\mathcal{A}) \smid \E_{\mu \sim \pi} \big[\E_{\mu}[X]\big] = q \right\}.
	\]
	for some $q\in (0,a)$.  We saw in Example~\ref{eg:shiva} that $\mathcal{U}(\Pi) = \frac{q}{a}$.  Now suppose that we observe the random variable $\Drv := (X_1, \dots, X_n)$ corresponding to $n$ i.i.d.\ samples of $\mu^\dagger \in \mathcal{A}$.  More precisely, we observe $\Drv \in B$ where $B = B_1 \times \cdots B_n$ and $B_i$ is the ball in $(0, 1)$ of center $x_i$ and radius $\rho$, $x_i \in (0,1)$ and $0 < \rho \ll 1/n$.  Let
$\Dmap^{n}$ denote the data map corresponding to taking $n$ i.i.d.~samples,  that is, $\Dmap^{n}(\mu):=
\mu \otimes  \cdots \otimes \mu$, and observe that $\Dmap^{n}(\mu)[B]=\prod_{i=1}^n \mu[B_i]$.

	Theorem~\ref{thm:alternredded} implies that $\mathcal{U}(\Pi\odot_B\Dmap^{n})$ is equal to the supremum over $\alpha_1,\alpha_2 \geq 0$, $\mu_1,\mu_2\in \mathcal{A}$ of
	\[
		\alpha_1 \mu_1[X \geq a] \Dmap^{n}(\mu_1)[B] + \alpha_2 \mu_2[X \geq a] \Dmap^{n}(\mu_2)[B]
	\]
	subject to the constraints
	\[
		\alpha_1 (\E_{\mu_1}[X] - q) + \alpha_2 (\E_{\mu_2}[X] - q) = 0,
	\]
	\[
		\alpha_1 \Dmap^{n}(\mu_1)[B] + \alpha_2 \Dmap^{n}(\mu_2)[B] = 1,
	\]
	with $\Dmap^{n}(\mu)[B]=\prod_{i=1}^n \mu(B_i)$.  Introducing slack variables $\beta_{1,i} := \mu_1[B_i]$ and $\beta_{2,i} := \mu_2[B_i]$ as $n$ linear constraints on $\mu_1$ and $n$ linear constraints on $\mu_2$ we obtain (from \cite[Thm.~4.1]{OSSMO:2011}) that the supremum can be achieved by assuming that each $\mu_i$ is the weighted sum of at most $n+2$ Dirac masses. Assuming that the $B_i$ are non intersecting balls of radius $\rho\ll 1/n$ centered on $x_1,\dots,x_n$, $n$ of these Dirac masses will have to be put at $x_1,\dots,x_n$; for optimality, the two others will have to be put at $0$ and $a$ (with weights $p_1$ and $p_2$).  Introducing $\gamma_1=\alpha_1  \Dmap^{n}(\mu_1)[B]$  and $\gamma_2=\alpha_2  \Dmap^{n}(\mu_2)[B]$, it follows that $\mathcal{U}(\Pi\odot_B\Dmap^{n})$ is equal (as $\rho \downarrow 0$) to the supremum over $\gamma_1,\gamma_2\geq 0$, $p_1,p_2 \in [0,1]$ of
	\[
		\gamma_1 p_1+\gamma_2 p_2
	\]
	subject to the constraints
	\[
		\gamma_1 + \gamma_2 = 1,
	\]
	and
	\[
		\gamma_1 \frac{(ap_1+\sum_{i=1}^n x_i \beta_{1,i})-q}{\prod_{i=1}^n \beta_{1,i}}+ \gamma_2
		\frac{(a p_2+\sum_{i=1}^n x_i \beta_{2,i})-q}{\prod_{i=1}^n \beta_{2,i}} = 0.
	\]
	By considering $0<\beta_{i,j} \ll 1$ it is easy to obtain that $\mathcal{U}(\Pi\odot_B\Dmap^{n})=1$.
\end{eg}

\begin{eg}\label{eg:learnvsstab}
We will now use Theorem \ref{thm:alternredded} to prove equation \eqref{eqhieuhdee} of Subsection \ref{subsec:learningvsrobustness}. Let $\Phi$ be defined as in Subsection \ref{subsec:learningvsrobustness}.
Let $\mathcal{A}(\alpha)$ and $\Pi(\alpha)$ be defined as in \eqref{eqadconstdata} and  \eqref{eqpiddkje}. Then, Theorem \ref{thm:alternredded} implies that $\mathcal{U}\big(\Pi(\alpha)|B^n_\delta\big)$, the least upper bound on posterior values, is
equal to the supremum over $\alpha_1,\alpha_2 \geq 0$, $\mu_1,\mu_2\in \mathcal{A}(\alpha)$ of
	\[
		\alpha_1 \mu_1[X \geq a] \mu_1^n[B^n_\delta] + \alpha_2 \mu_2[X \geq a] \mu_2^n[B^n_\delta]
	\]
	subject to the constraints
	\[
\begin{cases}
		\alpha_1 (\E_{\mu_1}[X] - m) + \alpha_2 (\E_{\mu_2}[X] - m) = 0,\\
		\alpha_1 \mu_1^n[B^n_\delta] + \alpha_2 \mu_2^n[B^n_\delta] = 1,
\end{cases}
	\]
	where we have used the notation $\mu^n[B^n_\delta]:=\prod_{i=1}^n \mu(B_\delta(x_i))$.

Introducing $\gamma_1=\alpha_1  \mu_1^n[B^n_\delta]$  and $\gamma_2=\alpha_2  \mu_2^n[B^n_\delta]$, it follows that $\mathcal{U}\big(\Pi(\alpha)|B^n_\delta\big)$ is equal  to the supremum over $\gamma_1,\gamma_2\geq 0$, $\mu_1,\mu_2\in \mathcal{A}(\alpha)$ of

	\[
		\gamma_1  \mu_1[X \geq a]+\gamma_2  \mu_2[X \geq a]
	\]
	subject to the constraints
	\[
\begin{cases}
		\gamma_1 + \gamma_2 = 1,\\
		\gamma_1 \frac{\E_{\mu_1}[X] - m}{\mu_1^n[B^n_\delta]}+ \gamma_2
		\frac{\E_{\mu_2}[X] - m}{\mu_2^n[B^n_\delta]} = 0.
\end{cases}
	\]
which can be simplified to the supremum over $\mu_1,\mu_2\in \mathcal{A}(\alpha)$ of
	\begin{equation}\label{eq:bdbebu3dd}
		\frac{1}{1+\frac{\E_{\mu_1}[X]-m}{m-\E_{\mu_2}[X]}\frac{\mu_2^n[B^n_\delta]}{\mu_1^n[B^n_\delta]}}    \mu_1[X \geq a]+(1-\frac{1}{1+\frac{\E_{\mu_1}[X]-m}{m-\E_{\mu_2}[X]}\frac{\mu_2^n[B^n_\delta]}{\mu_1^n[B^n_\delta]}}  ) \mu_2[X \geq a]
	\end{equation}
By introducing slack variables for $m_1=\E_{\mu_1}[X]$ and $m_2=\E_{\mu_2}[X]$, maximizing \eqref{eq:bdbebu3dd} with $m_1$ and $m_2$, then taking a supremum over $m_1, m_2$, one obtains that the
supremum of \eqref{eq:bdbebu3dd} is achieved, in the limit $\delta \downarrow 0$, in the configuration where $\mu_1$ puts most of its mass on $a$, $\mu_2$ puts most of its mass on $0$, and $\frac{\mu_2^n[B^n_\delta]}{\mu_1^n[B^n_\delta]}\approx \frac{1}{\alpha^2}$ which yields
\begin{equation}\label{eqhieuhdeebis}
\lim_{\delta \rightarrow 0}\mathcal{U}\big(\Pi(\alpha)|B^n_\delta\big)=\frac{1}{1+\frac{1}{\alpha^2} \frac{a-m}{m}} .
\end{equation}
\end{eg}

\subsection{Nested Reduction for Posterior Values}
\label{Sec:finitdimmarginalspost}

Here, as in Section~\ref{Sec:finitdimmarginals}, we show how the optimization problems \eqref{eq:Upostes} and \eqref{eq:Lpostes} can be reduced to nested OUQ optimization problems (i.e.~nested problems analogous to \eqref{eq:defma1} and \eqref{eq:defma2}) when the collection $\Pi$ of admissible priors is defined by how they push forward by a measurable mapping $\Psi \colon \mathcal{A} \to \mathcal{Q}$.  That is, we specify a feature space $\mathcal{Q}$, a measurable map $\Psi \colon \mathcal{A} \to \mathcal{Q}$, a subset $\mathfrak{Q}\subseteq \mathcal{M}(\mathcal{Q})$ and define the admissible set of priors by
\[
	\Pi:=\Psi^{-1}\mathfrak{Q}=\{\pi \in \mathcal{M}(\mathcal{A}) \mid \Psi\pi \in \mathfrak{Q}\} .
\]
As before, we focus on reducing the upper bound
\begin{equation}
	\label{eq:eqUpsipost}
	\mathcal{U}\big(\Psi^{-1}\mathfrak{Q}\big| B\big):=\sup_{\pi  \in (\Psi^{-1}\mathfrak{Q})_B} \E_{\pi \odot \Dmap}[\Phi|B] .
\end{equation}

\begin{thm}
	\label{thm:sivahidden}
	Let $\mathcal{A}$ be a Suslin space, let $\mathcal{Q}$ be a separable and metrizable space, and let $\Psi \colon \mathcal{A} \to \mathcal{Q}$ be measurable.  Moreover, let $\mathfrak{Q} \subseteq \mathcal{M}(\mathcal{Q})$ be such that $\supp(\mathbb{Q}) \subseteq \Psi(\mathcal{A})$ for all $\mathbb{Q} \in \mathfrak{Q}$.  Then, for each $\mathbb{Q} \in \mathfrak{Q}$, $\Psi^{-1}\mathbb{Q}$ is non-empty.  Moreover, the upper bound  $\mathcal{U} \big(\Psi^{-1}\mathfrak{Q} \big| B\big)$, defined in \eqref{eq:eqUpsipost}, satisfies
	\begin{equation}
		\label{eq:2qbiscondihjh}
		\begin{split}
			\mathcal{U}\big(\Psi^{-1}\mathfrak{Q}\big| B\big) = \sup \left\{ \lambda \in \R \smid  \sup_{ \mathbb{Q}\in \mathfrak{Q} } \mathbb{E}_{q\sim \mathbb{Q}} \left[ \sup_{\mu\in  \Psi^{-1}(q)} \bigl(\Phi(\mu)-\lambda\bigr)\Dmap(\mu)[B] \right]>0 \right\}\, ,
		\end{split}
	\end{equation}
	where the expectations on the right-hand side are defined as in  \eqref{def-expectation2}.  Finally, the expectation operator on the right-hand side is measure affine in $\mathbb{Q}$, as defined in \eqref{def_measureaffine}.
\end{thm}

\begin{rmk}
	\label{rmk:sivahiddenrmk}
	Note that Theorem \ref{thm:sivahidden} is more general than Theorem \ref{thm:alternredded} because its application does not require the assumption that $\Psi^{-1}\mathfrak{Q}$ is defined via generalized moments constraints.
\end{rmk}

The following theorem  is our main result.
It shows not only that the right-hand side of the assertion \eqref{eq:2qbiscondihjh} of Theorem \ref{thm:sivahidden} depends on the sample data in a very weak way, but also that, under very mild assumptions, the observation of this sample data leads to an increase (rather than a decrease) of the least upper bound on the quantity of interest:

\begin{thm}[Main Brittleness Theorem]
	\label{thm:shiva}
	Let $\mathcal{A}$ be a Suslin space, let $\mathcal{Q}$ be a separable and metrizable space, and let $\Psi \colon \mathcal{A} \to \mathcal{Q}$ be measurable.  Moreover, let $\mathfrak{Q} \subseteq \mathcal{M}(\mathcal{Q})$ be such that $\supp(\mathbb{Q}) \subseteq \Psi(\mathcal{A})$ for all $\mathbb{Q} \in \mathfrak{Q}$.  Suppose that, for all $\delta >0$, there exists some $\mathbb{Q}\in \mathfrak{Q}$ such that
	\begin{equation}
		\label{eq:dto0}
		\E_{q\sim \mathbb{Q}} \left[ \inf_{\mu\in  \Psi^{-1}(q)} \Dmap(\mu)[B] \right]=0
	\end{equation}
	and
	\begin{equation}
		\label{eq:djkdjehjehj33}
		\mathbb{P}_{q\sim \mathbb{Q}} \left[ \sup_{\mu\in \Psi^{-1}(q),\, \Dmap(\mu)[B]>0}\Phi(\mu) > \sup_{\mu\in \mathcal{A}}\Phi(\mu) - \delta \right]>0 .
	\end{equation}
	Then
	\begin{equation}
		\label{eq:2qbisjhjycondddihjh}
		\mathcal{U}\big(\Psi^{-1}\mathfrak{Q}\big|B\big) =\mathcal{U}(\mathcal{A}) .
	\end{equation}
\end{thm}

\begin{rmk}
	Note that the convention that $\sup \varnothing =-\infty$
	implies that, if the assumption \eqref{eq:djkdjehjehj33} is satisfied, then there is a measure $\mathbb{Q} \in \mathfrak{Q}$ such that the set of $q$ such that $\Dmap(\mu)[B]>0$ for some $\mu\in \Psi^{-1}(q)$ has strictly positive $\mathbb{Q}$-measure.
\end{rmk}

\begin{rmk}
 	Theorem \ref{thm:shiva} states that if there exists $\mathbb{Q} \in \mathfrak{Q}$ putting some mass on a neighborhood of the values $q$ of $\Psi$ where $\sup_{\mu\in \Psi^{-1}(q)} \Phi(\mu)$ achieves its supremum, then
	\[
		\mathcal{U}\big(\Psi^{-1}(\mathfrak{Q})\big| B\big) = \mathcal{U}(\mathcal{A}).
	\]
	On the other hand, Theorem~\ref{thm-barriersorigin} asserts that
	\begin{equation}
		\label{eq:eyvyt3}
\mathcal{U}(\Psi^{-1}\mathfrak{Q}) \leq  \mathcal{U}(\mathcal{A}) ,
	\end{equation}
	so we conclude that
	\begin{equation}
		\label{eq_shv1}
		\mathcal{U}(\Psi^{-1}\mathfrak{Q}) \leq \mathcal{U}\big(\Psi^{-1}\mathfrak{Q}\big|B\big) .
	\end{equation}
	That is, \emph{observing the sample data  does not improve the optimal bound!}  Moreover, when the inequality \eqref{eq:eyvyt3} is strict, if we define
	\[
		\delta := \mathcal{U}(\mathcal{A})-\mathcal{U}(\Psi^{-1}\mathfrak{Q}) >0
	\]
	then it follows that
	\begin{equation}
		\label{eq_shv2}
		\mathcal{U}(\Psi^{-1}\mathfrak{Q}) + \delta \leq \mathcal{U}\big(\Psi^{-1}(\mathfrak{Q})\big|B\big) ,
	\end{equation}
	from which we conclude that when the inequality  \eqref{eq:eyvyt3} is strict, \emph{observing the sample data makes the optimal bound worse!}  In other words, after the observation of the sample data (which may be limited to a single realization of $X$ under the measure $\mu^\dagger$, or an arbitrary large number of independent samples of $X_i$) the optimal upper bound on the quantity of interest,
	\[
		\mathcal{U}(\Psi^{-1}\mathfrak{Q}) = \sup_{\pi\in \Psi^{-1}\mathfrak{Q}}\E_{\mu\sim \pi}\big[\Phi(\mu)\big],
	\]
	increases to
	\[
		\mathcal{U}(\mathcal{A})=\sup_{\mu\in \mathcal{A}}\Phi(\mu).
	\]
\end{rmk}

\begin{eg}
	\label{eg:shiva3}
	Consider $\mathcal{A}:=\mathcal{M}([0,1])$, $\Phi(\mu)=\E_{\mu}[X]$, $\Dmap^{n}(\mu):=\mu\otimes \cdots \otimes \mu$.  In this example are interested in estimating the mean of $X$ under some unknown measure $\mu^\dagger\in \mathcal{A}$ and we observe $d=(d_1,\ldots,d_n)$, $n$ i.i.d.\ samples from $X$;  note that $n$ can be very large.  The sample data contain information on $\mu^\dagger$ through the fact that their distribution is
 $\Dmap^{n}(\mu^\dagger)=\mu^\dagger\otimes \cdots \otimes \mu^\dagger$ (i.e.\ although the distribution of the sample data is unknown, its dependency structure, as a functional of $\mu^\dagger$, is known).

	Let $k$ be a (possibly large) number.  Define $\Pi$ to be the set of priors $\pi$ under which the distribution of $(\E_{\mu}[X],\ldots,\E_{\mu}[X^{k}])$ is $\mathbb{Q}$, where $\mathbb{Q}$ is a distribution on $\R^{k}$ such that $\E_{\mu}[X]$ (its first marginal) is uniformly distributed on $[0,1]$ and such that the (conditional) distribution of $\E_{\mu}[X^2]$ conditioned on $\E_{\mu}[X]=q_1$ is the uniform distribution on the interval
	\[
		\left[ \inf_{\mu \in \mathcal{A},\, \E[X]=q_1} \E_\mu[X^2], \sup_{\mu \in \mathcal{A},\, \E[X]=q_1} \E_\mu[X^2] \right]
	\]
	and such that the conditional distributions of the other marginals $\E_\mu[X^k]$ are defined iteratively in the same manner. For this example, note that $\Psi(\mu)=(\E_{\mu}[X],\ldots,\E_{\mu}[X^{k}])$.  Note that, for $q:=(q_1,\ldots,q^{k})$ in the range of $\Psi$ (i.e.\ $\Psi(\mathcal{A})$),
	$\Psi^{-1}(q)$ is the subset of measures $\mu \in \mathcal{M}([0,1])$ such that $\E_{\mu}[X^i]=q_i$ for $1\leq i \leq k$.  Let $B$ be defined as $B_1\times \cdots B_n$ where each $B_i$ is a ball of radius $\rho$ containing $d_i$.

	We will now use Theorem \ref{thm:shiva} to compute optimal bounds on the posterior values of $\Phi(\mu)=\E_{\mu}[X]$. We will focus our attention on the upper bound.  First observe that in this example $\mathfrak{Q}$ is reduced to the single measure $\mathbb{Q}$ constructed above and $\Dmaps$ is reduced to the single data map $\Dmap^{n}$.

	Let us first check that condition \eqref{eq:djkdjehjehj33} is always satisfied (irrespective of the value of the data $d$).  Note that condition \eqref{eq:djkdjehjehj33} is satisfied if for all $\delta>0$ there exists a subset of values of $q$ of strictly positive $\mathbb{Q}$-measure such that $\big\{\mu \in \Psi^{-1}(q)\mid \Dmap^{n}(\mu)[B]>0 \text{ and } \E_{\mu}[X]\geq 1-\delta  \big\}$ is non empty. So, let $\delta > 0$ be arbitrary and define $\mu_d$ to be the empirical distribution of $d$, i.e.
	\[
 		\mu_d:=\frac{\sum_{i=1}^{n} \delta_{d_i}}{n} .
	\]
	Define
	\[
		\mathcal{A}_{\delta}:= \{\mu \in \mathcal{A}\mid \E_{\mu}[X]\geq 1-\delta/2 \}.
	\]
 	One can show by induction that $\Psi(\mathcal{A}_{\delta})$ has a non-empty interior and that any open subset of $\Psi(\mathcal{A})$ has strictly positive $\mathbb{Q}$-measure. Let $q^{\ast}$ be a point in the interior of $\Psi(\mathcal{A}_{\delta})$, and let $B_\tau(q^{\ast})$ be a ball of center $q^{\ast}$ and radius $\tau$ such that $B_{2\tau}(q^{\ast})$ is contained in the interior of $\Psi(\mathcal{A}_{\delta})$.  Note that $B_\tau(q^{\ast})$ has strictly positive $\mathbb{Q}$-measure.  Furthermore, for $\epsilon$ sufficiently small, for each $q\in B_\tau(q^{\ast})$ there exists $q'\in B_{2\tau}(q^{\ast})$ and $\mu \in \Psi^{-1}(q')$ such that $\mu_\epsilon:=(1-\epsilon)\mu +\epsilon \mu_d \in \Psi^{-1}(q)$.  Since $\Dmap^{n}(\mu_\epsilon)[B]>0$ and $\E_{\mu}[X]\geq 1-\delta/2$, it follows that \eqref{eq:djkdjehjehj33} is satisfied (irrespective of the value of the data $d$).

	Let us now consider condition \eqref{eq:dto0}.  Observe that condition \eqref{eq:dto0} is satisfied if for $\mathbb{Q}$-almost all $q\in \Psi(\mathcal{A})$ and all $\epsilon>0$, there exists $\mu \in \Psi^{-1}(q)$ such that
 $\Dmap^{n}(\mu)[B]<\epsilon$.  Assume that $d$ contains at least $k+2$ distinct points and that $\rho$ is strictly smaller than half of the minimal distance between two of such points, so that the associated $B_i$ do not overlap;  note that this assumption is satisfied with probability converging to one (as $n\to \infty$) if the data are sampled from a measure $\mu^\dagger$ that is absolutely continuous with respect to the Lebesgue measure on $[0,1]$.  Let $q\in \Psi(\mathcal{A})$; by the reduction theorems of \cite{OSSMO:2011} there exists $\mu_q\in \Psi^{-1}(q)$ such that $\mu_q$ is the weighted sum of at most $k+1$ Dirac masses in
 $[0,1]$.  Since there exist at least $k+2$ non-overlapping $B_i$ we have $\Dmap^{n}(\mu_q)[B]=0$ which implies condition \eqref{eq:dto0}.  Hence, Theorem~\ref{thm:shiva} implies that, for this (possibly) highly constrained problem characterized by a (possibly) large number of sampled data points, the optimal bounds on the posterior values of $\E_{\mu}[X]$ are zero and one whereas the set of prior values of $\E_{\mu}[X]$ is the single point $\{\frac{1}{2}\}$.
\end{eg}

\begin{rmk}
For a thorough analysis of Example \ref{eg:shiva3} we refer to
 \cite{OwhadiScovel:2013} where, in particular, a {\em quantitative} version of Theorem \ref{thm:shiva} is developed
and then applied to Example \ref{eg:shiva3}.
Curiously, a refined analysis of the integral geometry of the truncated Hausdorff moment space, used to demonstrate the approximate satisfaction of the  conditions of Theorem \ref{thm:shiva}, is shown in \cite{OwhadiScovel:2013} to lead to a new family of Selberg integral formulas. See \cite{ForresterWarnaar} for a discussion of their importance.
\end{rmk}

\begin{rmk}
	\label{rmk:jhbdbd3}
	Note that the assumptions of Theorem \ref{thm:shiva} are extremely weak. In plain words, Theorem~\ref{thm:shiva} implies that if the probability of observing the data can be arbitrary small under priors contained in $\mathcal{A}$ that are putting mass near the extreme values of $\Phi$, then the optimal bounds on posterior values are the extreme values of $\Phi$ in $\mathcal{A}$ (even if the data comes in the form of a large number of samples and the set of priors is highly constrained).  Example \ref{eg:shiva3} illustrates that one consequence of Theorem \ref{thm:shiva} is that Bayesian posteriors are not robust,  and in fact are fragile with respect to the choices of priors constrained by marginals, even with a highly constrained subset of priors of $\mathcal{M}(\mathcal{A})$.

	Moreover, if $\Pi$ is convex, then by considering priors of the form $\pi_0 \lambda +(1-\lambda) \pi_1$ with $\pi_0, \pi_1 \in \Pi$, $\pi_0\cdot \Dmap[B]>0$ and $\pi_1\cdot \Dmap[B]>0$, it is easy to see that the Bayesian posterior can take any value
	in the interval $\big(\mathcal{L}(\mathcal{A}),\mathcal{U}(\mathcal{A})\big)$, irrespective of the data.
	In addition, it is easy to observe that even
	including the quantity of interest $\Phi$ in the marginal $\Psi$ does not prevent this fragility.  Theorem~\ref{thm:shiva} also leads to the following apparent paradoxes when the Bayesian framework is applied to the space $\mathcal{A}$: (1) Posteriors with different priors may diverge as more and more data comes in;  (2) When the sample data is observed with some (say Gaussian) measurement noise of variance $\sigma^2$, then,  the optimal bound $\mathcal{U}\big(\Psi^{-1}(\mathfrak{Q})\big|B\big)$ on the quantity of interest $\Phi$ converges towards $\mathcal{U}\big(\Psi^{-1}(\mathfrak{Q})\big)$  as  $\sigma^2 \to \infty$.  That is, if one interprets optimal bounds on posterior values as uncertainty bounds, then one would reach the paradoxical conclusion that adding measurement uncertainty decreases the uncertainty of the quantity of interest.  The idea of the proof of this assertion is based on the following observation:
	
	Let $y$ be the (noisy) measurement whose distribution given the value of the data $\Ddata$ is assumed to be independent of $\mu$. Write $p_\sigma(d)[B]$ for the probability that the value of $y$ belongs to a set $B$ and observe that the conditional value of the quantity of interest $\Phi$ given the $y\in B$ is equal to
 	\begin{equation}
 		\label{eq:condnoisy}
		\frac{\E_{\pi}\Big[\Phi(\mu)\E_{d\sim \Dmap(\mu)}\big[p_\sigma(d)[B]\big]\Big]}{\E_{\pi}\Big[\E_{d\sim \Dmap(\mu)}\big[p_\sigma(d)[B]\big]\Big]} \, .
 	\end{equation}
 	We deduce that if $p_\sigma(d)[B]/p_\sigma(d')[B]$  converges towards one as the level of noise $\sigma\to \infty$ uniformly in $(d,d')\in [0,1]^2$ (which is the case if the data in Example \ref{eg:shivabis} is observed with Gaussian noise of increasing variance, see also Example \ref{eg:hiuhdiueh} below), then  \eqref{eq:condnoisy} converges towards the prior value of $\Phi$ as $\sigma\to \infty$ uniformly in $\pi$.

 The fact that optimal bounds on prior values may become  less precise after conditioning is known as the \emph{dilation phenomenon}
 in robust Bayesian inference \cite{WassermanSeidenfeld:1994}, and, in some sense, the brittleness results presented in this paper could be seen as an extreme occurrence of this phenomenon.
 \end{rmk}

\begin{eg}
	\label{eg:hiuhdiueh}
	Consider again Example~\ref{eg:shivabis} with the set of admissible priors $\pi$ on $\mathcal{A}$ defined as the collection
	\[
		\Pi := \left\{ \pi \in \mathcal{M}(\mathcal{A}) \smid \E_{\mu \sim \pi} \big[\E_{\mu}[X]\big] = q \right\}.
	\]
	and the map $\Dmap^{n}$ corresponding to the observation of $n$ i.i.d.\ samples of $\mu$.  For $q \in (0,a)$, let $\mathfrak{Q}$ be the set of probability measures $\mathbb{Q}$ on $[0,1]$ such that $\E_{q' \sim \mathbb{Q}}[q']=q$.  Let $\mathbb{Q}$ be the probability measure on $[0,1]$ with probability density function $p(x)=(1-q)/q$ on $[0,q]$ and $p(x)=q/(1-q)$ on $(q,1]$.  It is easy to check that $\mathbb{Q}\in \mathfrak{Q}$, that
	\begin{equation}
		\label{eq:dtosq0}
		\E_{q'\sim \mathbb{Q}} \left[ \inf_{\mu\in \mathcal{A}\,:\, \E_{\mu}[X]=q'} \prod_{i=1}^n \mu[B_i] \right]=0,
	\end{equation}
	and that, for all $\delta>0$,
	\begin{equation}
		\label{eq:djkdjdeehjehj33}
		\mathbb{P}_{q'\sim \mathbb{Q}} \left[ \sup_{\mu\in \mathcal{A}\,:\, \E_{\mu}[X]=q',\, \prod_{i=1}^n \mu[B_i]>0}\E_\mu[X] >1   -\delta \right]>0 .
	\end{equation}
	It follows from Theorem~\ref{thm:shiva} that
	\begin{equation}
		\label{eq:2qbisewcondddihjh}
		\mathcal{U}\big(\Psi^{-1}\mathfrak{Q}\big|B\big) = 1.
	\end{equation}
\end{eg}

\section{Bayesian Robustness and Consistency}
\label{sec:misspecori}

It is appropriate at this point to place the results of Sections \ref{Sec:OpBoundsPriorEstimates} and \ref{Sec:OpBoundsPosteriorEstimates} in the more well-established context of two key questions about Bayesian inference, namely its \emph{robustness} with respect to perturbations of the prior (and likelihood and observed data), and its frequentist \emph{consistency}.  This discussion will also motivate Section \ref{Sec:local}, where we show that Bayesian inference can be profoundly non-robust even under arbitrarily small local perturbations in total variation and Prokhorov metrics.

\subsection{Bayesian Robustness}

The robust Bayesian viewpoint appears to have been introduced independently by Box \cite{Box:1953} and Huber \cite{Huber:1964};  see e.g.\ \cite{Berger:1984, Berger:1994} and Chapter 15 of \cite{HuberRonchetti:2009} for surveys of the field.  In the robust Bayesian approach, a class $\Pi$ of priors and a class $\Lambda$ of likelihoods together produce a class of posteriors by pairwise combination through Bayes' rule.  Robust Bayesian methods are a subclass of the methods of \emph{imprecise probability};  the idea that the probability of an event need not be a single real number has a history stretching back to Boole \cite{Boole:1854} and Keynes \cite{Keynes:1921}, with more recent and comprehensive foundations laid out in e.g.\ \cite{Kuznetsov:1991, Walley:1991, Weichselberger:2000}.

One way of generating such a class $\Pi$ of priors is via a belief function, as in \cite{Wasserman:1990} and Dempster--Shafer theory more generally.  The belief function framework encompasses prior probabilities whose values are known only on some finite partition of the probability space, and not the whole $\sigma$-algebra;  classes of $\varepsilon$-contaminated priors can also be represented in this way, as well as classes of locally perturbed priors.  The belief function approach has the useful feature that explicit formulae can be given for the lower and upper posterior probabilities of events \cite[Theorem 4.1]{Wasserman:1990}.

Another typical approach to generating a class $\Pi$ might be to consider a finite-dimensional parametrized class of models.  For example, one could consider, instead of a single Gaussian prior on $\R$ of specified mean and variance, a two-parameter class of Gaussian priors with a range of means and variances, or a three-parameter class of skew-Gaussian priors.  Similarly, one might consider a two-parameter class of beta distributions instead of a uniform prior on a bounded interval.

However, a danger in specifying a finite-dimensional class $\Pi$ of priors is that one is making very strong statements about the form of the priors, particularly with regard to the tails, that cannot be justified based on often-limited amounts of prior information.  For example, if all the priors $\pi \in \Pi$ have thin tails, then the class $\Pi$ will have a very difficult time modeling events that lie in those tails, even when exposed to data from those regions.  This problem is particularly important in applied fields such as catastrophe modeling, insurance, and re-insurance, in which the catastrophic events of interest are by definition high-impact low-probability ``Black Swan'' events:  the difference between an exponentially small and an inverse-polynomially small tail can be vitally important.   Also, because members of a finite-dimensional parametric family $\Pi$ of priors often have similar qualitative properties (such as being mutually absolutely continuous), the apparently broader perspective does not not add much to the asymptotic posterior picture in terms of robust consistency, although it does provide a broader understanding given finitely many samples.

Rather than specifying a finite-dimensional $\Pi$, it is epistemologically more reasonable to specify a finite-\emph{codimensional} $\Pi$, for example by specifying interval bounds on the expected values of finitely many observed test functions (i.e.\ generalized moment inequalities);  this setting encompasses the finite-partition belief function framework mentioned above.  Calculation of optimal prior and posterior bounds on quantities of interest is often an exercise in numerical optimization \cite{BertsimasPopescu:2005, OSSMO:2011, Smith:1995} rather than closed-form formulae.

One consequence of Theorems
\ref{thm:alternredded} and \ref{thm:shiva} is that the very same Bayesian sensitivity analysis framework that produces the robustness results of classical robust Bayesian inference under finite-dimensional classes of priors also leads to brittleness results under finite-codimensional classes of priors, when the set of all priors is infinite dimensional.
As illustrated by \eqref{eqhieuhdee} and Example \ref{eg:learnvsstab}, Theorems
\ref{thm:alternredded} and \ref{thm:shiva} can also be used to obtain robustness/stability results by adding sufficiently strong constraints (at the expense of learning) on the probability of the data in the model class.  As discussed in Subsection \ref{subsec:learningvsrobustness}, Example \ref{eg:learnvsstab} suggests that posterior stability and learning are antagonistic properties in Bayesian inference under finite information.

\subsection{Motivation for Bayesian Inconsistency and Model Misspecification}
\label{sec:misspec}

To motivate Section \ref{Sec:local} and interpret the results of this paper in relation to the issue of convergence of posterior values in Bayesian inference we will now analyse and review questions of Bayesian consistency, inconsistency and model misspecification. There is, of course, a large literature on these topics, and we will not attempt to be exhaustive in providing references;  rather, our aims are:  first, to give a short reminder on how Bayesian inference is currently employed in Uncertainty Quantification (UQ);  second, to identify issues and popular beliefs about what one actually learns from Bayesian inference, and thereby motivate the results of this paper;  and, last, to present sufficient references that the interested reader can find technical justification for the formal manipulations of this subsection.

In this subsection, we are interested in estimating $\Phi(\mu^\dagger)$
where $\Phi$ is a known \emph{quantity of interest}
function and $\mu^\dagger$ is an unknown (or partially known) probability measure on $\mathcal{X}$.  For the purposes of exposition, in this subsection, we assume that $\mathcal{X} = \R^k$.  One example of a quantity of interest, when $\mathcal{X} = \R$, is $\Phi(\mu^{\dagger}) := \mu^\dagger[X\geq a]$ (the probability that the random variable $X$ distributed according to $\mu^\dagger$ exceeds the threshold value $a$).  We also assume that we are given $n$ independent samples $\Ddata_1, \dotsc, \Ddata_n$, each distributed according to $\mu^\dagger$.

We will now present the parametric Bayesian answer to this problem.  For the purposes of exposition, in this section, we restrict our attention to parametric Bayesian inference.  We first introduce $\{ \mu(\quark, \theta) \}_{\theta \in \Theta}$ a family of probability distributions on $\mathcal{X}$ parametrized by $\theta \in \Theta$ (and commonly referred to as the \emph{model class}). For the sake of simplicity here we also assume that $\Theta = \R^\ell$.
Let
\[
	\mathcal{A}_0 := \bigl\{ \mu(\quark, \theta) \,\big|\, \theta \in \Theta \bigr\}.
\]
Note that $\mathcal{A}_0$ is a subset of $\mathcal{M}(\mathcal{X})$ that may or may not contain $\mu^\dagger$.  If $\mu^{\dagger} \notin \mathcal{A}_0$, then the model is said to be \emph{misspecified};  otherwise, the model is said to be \emph{well specified}.

We next introduce $p_0 \in \mathcal{M}(\Theta)$, a probability distribution on $\Theta$ (the \emph{prior distribution} on $\theta$).  Let $\pi_0$ be the push-forward (measure) of $p_0$ under the map $\theta \mapsto \mu(\quark,\theta)$ (see \cite{Bogachev1,Bogachev2}, Sections 3.6, 3.7) and observe that $\pi_0$ is a probability distribution on $\mathcal{A}_0$, i.e.\ $\pi_0 \in \mathcal{M}(\mathcal{A}_0)$, and that $\pi_{0}$ is the distribution of the random measure $\mu(\quark,\theta)$ when $\theta$ is distributed according to $p_0$.

The next step is then to estimate $\Phi(\mu^\dagger)$ via conditioning.  Let $p_n \in \mathcal{M}(\Theta)$ be the posterior distribution of $\theta$ given the observation of the i.i.d.~samples $\Ddata_1,\ldots,\Ddata_n$, as obtained using Bayes' formula, and let $\pi_n$ be the push-forward of $p_n$.   The Bayesian estimate of $\Phi(\mu^\dagger)$ is therefore
\begin{equation}
	\label{eq:Bayesianestimation}
	\E_{\mu \sim \pi_n} [\Phi(\mu)].
\end{equation}

For the purposes of exposition, we assume that the measures $\mu(\quark,\theta)$ and $\mu^\dagger$ are all absolutely continuous with respect to the Lebesgue measure and write $f(\quark,\theta)$ and $f^\dagger$ for their densities, which we assume to be continuous.   Similarly, we assume that the measure $p_0$ is absolutely continuous with respect to the Lebesgue measure and, abusing notation, write $p_0$ for both the measure $p_0$ and its (continuous) density, and similarly for $p_n(\quark)$, the posterior density of $\theta$ on $\Theta$ given the observation the samples $\Ddata_1, \dotsc, \Ddata_n$.  We will now examine the convergence properties of the sequence of posterior densities $p_n(\theta)$ as $n\to \infty$.  This analysis being classical (see for instance \cite{Nickl:2012} and references therein), our purpose is not to provide rigorous justifications but rather to familiarize the reader with the mechanisms regarding
 the convergence of posteriors.

We have
\[
	p_n(\theta)
	= \frac{p_0(\theta)\prod_{j=1}^n f(\Ddata_j,\theta)}{\int_{\Theta}p_0(\theta')\prod_{j=1}^n f(\Ddata_j,\theta')\, \mathrm{d} \theta' }
	\equiv \frac{p_0(\theta)\prod_{j=1}^n f(\Ddata_j,\theta)}{\E_{p_0} [ \prod_{j=1}^n f(\Ddata_j, \quark) ] }
\]
which we write as
\[
	p_n(\theta)
	= \frac{p_0(\theta)e^{n L_n(\theta)}}{\int_{\Theta}p_0(\theta')e^{n L_n(\theta')} \, \mathrm{d} \theta' }
	\equiv \frac{p_0(\theta)e^{n L_n(\theta)}}{\E_{p_0} [e^{n L_n(\quark)}]} ,
\]
where
\[
	L_n(\theta):=\frac{1}{n}\sum_{j=1}^n \log f(\Ddata_j,\theta) .
\]
Recall that $\prod_{j=1}^n f(\Ddata_j,\theta)$ is commonly known as the \emph{likelihood} and $L_n(\theta)$ as the (\emph{sample}) \emph{average log-likelihood}.

\paragraph{Consistency and the Large-Sample Limit.}
Now observe that if $\log f(\Ddata_j,\theta)$ is integrable then it follows from the Law of Large Numbers that $L_n(\theta)$ converges almost surely, as $n\to \infty$, to the \emph{expected log-likelihood} $L(\theta)$ defined by
\begin{equation}
	\label{eq:hiwuhiued}
	L(\theta) := \int_{\mathcal{X}} f^\dagger(x) \log \big(f(x,\theta)\big)\, \mathrm{d} x \, .
\end{equation}
Assuming that $L(\theta)$ has a unique maximizer $\theta^{\ast}\in \Theta$ (corresponding to the asymptotic limit of the \emph{maximum likelihood estimator} (MLE), as the number of data points goes to infinity)and that $p_0$ is strictly positive in every neighborhood of $\theta^{\ast}$, it follows under regularity assumptions on $f$
(or local strict convexity in the neighborhood of $\theta^{\ast}$) that $p_n(\theta)$ converges, almost surely, as $n\to \infty$, towards a Dirac mass supported at $\theta^{\ast}$.  Therefore, assuming $\Phi$ to be sufficiently regular, the Bayesian posterior estimate of $\Phi(\mu^\dagger)$, i.e.,
\begin{equation}
	\label{eq:Bayesestim}
	\int_{\Theta} \Phi\big(\mu(\quark,\theta)\big)p_n(\theta) \, \mathrm{d} \theta ,
\end{equation}
converges almost surely as $n \to \infty$ to
\begin{equation}
	\label{eq:Bayesestimlimit}
	\Phi\big(\mu(\quark,\theta^{\ast})\big) .
\end{equation}
Note that
\[
	L(\theta)=\mathrm{Ent}(f^\dagger)-D_{\mathrm{KL}}\bigl( f^{\dagger} \big\| f(\quark, \theta)\bigr),
\]
where  $\mathrm{Ent}(f^\dagger) := -\int_{\mathcal{X}} f^\dagger(x) \log f^\dagger(x) \, \mathrm{d} x$ is the \emph{entropy} of $f^{\dagger}$ and $D_{\mathrm{KL}}$ denotes the \emph{Kullback--Leibler divergence} defined by
\[
	D_{\mathrm{KL}}\bigl( f^{\dagger} \big\| f(\quark, \theta) \bigr) := \E_{x \sim f^{\dagger}} \left[ \log \frac{f^{\dagger}(x)}{f(x, \theta)} \right].
\]
It follows that $\theta^{\ast}$ is also the minimizer of $D_{\mathrm{KL}} \bigl( f^{\dagger} \big\| f(\quark, \theta) \bigr)$ with respect to $\theta$, i.e.\ the MLE $\theta^{\ast}$ is characterized by the property that $\mu(\quark, \theta^{\ast})$ is the distribution having minimal relative entropy to $\mu^{\dagger}$ in the model class $\{ \mu(\quark, \theta) \}_{\theta \in \Theta}$.

An immediate consequence of this observation is the fact if the model is not misspecified, i.e.\ if $\mu^\dagger$ is an element $\mu(\quark,\theta^\dagger)$ of the model class, then $\theta^{\ast} = \theta^\dagger$, $\mu(\quark, \theta^{\ast}) = \mu^\dagger$, and the Bayesian estimate \eqref{eq:Bayesestim} is asymptotically exact in the limit as $n\to \infty$.   In this situation, the Bayesian estimate is said to be \emph{consistent}.

This convergence result is known as the Bernstein--von Mises Theorem (see for instance \cite[Theorem 5]{Nickl:2012}) or as the Bayesian Central Limit Theorem, since the limiting posterior can even be described in a more refined way as being asymptotically normal and not just a point mass.  The condition that every open neighborhood of $\theta^{\dagger}$ has strictly positive $p_0$-probability (or, even more strongly, that the prior be globally supported) has been named \emph{Cromwell's Rule}\footnote{Since the posterior cannot possibly concentrate on a point outside the support of the prior, having a globally-supported prior and hence not ruling out a priori any $\theta \in \Theta$ as a possible $\theta^\dagger$ can be seen as a Bayesian version of Oliver Cromwell's famous injunction to the Synod of the Church of Scotland in 1650:  ``I beseech you, in the bowels of Christ, think it possible that you may be mistaken.''} by Lindley \cite{Lindley:1985}.

Recent results \cite{CastilloNickl:2013, Johnstone:2010, Leahu:2011, Nickl:2012} on the Bernstein--von Mises phenomenon show a notable dependence of the validity of the Bernstein--von Mises property upon subtle geometrical and topological details, and regularity properties of the model and the data-generating distribution.  Therefore, it is to be expected that any general stability condition for Bayesian inference would have to take account of such factors.

\paragraph{What Happens When the Model is Misspecified?}
To provide an illustrative answer to this question, consider the family of Gaussian models $\{ f(\quark , \theta) \mid \theta = (c, \sigma) \in \R \times \R_{+} \}$, where
\[
	f(x , c, \sigma) = \frac{1}{\sigma \sqrt{2 \pi}} \exp \left( - \frac{(x - c)^{2}}{2 \sigma^{2}} \right).
\]
What will happen when this model is exposed to data coming from a potentially non-Gaussian truth $\mu^{\dagger}$, with density $f^{\dagger}$, that has a well-defined mean $c^{\dagger}$ and standard deviation $\sigma^{\dagger}$?  By the above considerations, $\theta^{\ast}$ maximizes the expected log-likelihood \eqref{eq:hiwuhiued} with respect to $\theta$, and the expected log-likelihood is simply
\begin{equation}
	\label{eq:gaussex}
	L(\theta)=	- \int_{\R} f^{\dagger}(x) \frac{(x - c)^{2}}{2 \sigma^{2}} \, \mathrm{d} x - (\log \sigma) \int_{\R} f^{\dagger}(x) \, \mathrm{d} x - \log \sqrt{2 \pi}\, .
\end{equation}
A quick calculation using partial derivatives shows that $\theta^{\ast} = (c^{\ast}, \sigma^{\ast})$ maximizes \eqref{eq:gaussex} if and only if $c^{\ast} = c^{\dagger}$ and $\sigma^{\ast} = \sigma^{\dagger}$.  That is, the Bayesian estimate \eqref{eq:Bayesianestimation} of $\Phi(\mu^\dagger)$, for \emph{any} distribution $\mu^\dagger$ of mean $c^{\dagger}$ and standard deviation $\sigma^{\dagger}$, converges almost surely as the number of sample data goes to infinity, towards $\Phi \bigl( \mu (\quark, (c^{\dagger}, \sigma^{\dagger}) ) \bigr)$, where $\mu (\quark, (c^{\dagger}, \sigma^{\dagger}) )$ is the unique Gaussian distribution on $\R$ with mean $c^{\dagger}$ and standard deviation $\sigma^{\dagger}$.

However, now there is a problem:  there are many different probability distributions $\mu$ on $\R$ that have the same first and second moments as $\mu^{\dagger}$ but have, say, different higher-order moments, or different quantiles.  Predictions of those other moments or quantiles using $\mu (\quark, (c^{\dagger}, \sigma^{\dagger}) )$ can be inaccurate by orders of magnitude.  A trivial, albeit extreme, example is furnished by $\Phi(\mu) := \E_\mu \bigl[ |X-c_\mu| \geq t \sigma_\mu \bigr]$ (where $c_\mu$ and $\sigma_\mu$ denote the mean and standard deviation of $\mu$). Under the Gaussian model, (defining  $\mathrm{erf}(z):=\frac{1}{\sqrt{2\pi}}\int_{-\infty}^z e^{-\frac{t^2}{s}}\,dt$ as the \emph{error function})
\[
	\P \bigl[ | X - c_\mu | \geq t \sigma_\mu \bigr] = 1 + \mathrm{erf} \left( - \frac{t}{\sqrt{2}} \right),
\]
whereas the extreme cases that prove the sharpness of Chebyshev's inequality --- in which the probability measure is a discrete measure with support on at most three points in $\R$ --- have
\[
	\P \bigl[ | X - c_\mu | \geq t \sigma_\mu \bigr] = \min \left\{ 1, \frac{1}{t^{2}} \right\}.
\]
In the case of the archetypically rare ``$6 \sigma$ event'', the ratio between the two is approximately $1.4 \times 10^{7}$.  This is,
of course, an almost perversely extreme comparison:  it would be obvious to any observer with only moderate amounts of sample data that the data were being drawn from a highly non-Gaussian distribution.  However, it is not inconceivable that the true distribution $\mu^{\dagger}$ has a Gaussian-looking bulk but tails that are significantly fatter than those of a Gaussian, and the difference may be difficult to establish using reasonable amounts of sample data;  yet, it is those tails that drive the occurrence of ``Black Swans'', catastrophically high-impact but low-probability outcomes.  The results of this paper suggest that this situation is generic, and cannot be avoided no matter how many moments or integrals of arbitrary test functions of the truth $\mu^{\dagger}$ are matched nor how ``close'' $\mu^{\dagger}$ is to the class $\{ \mu(\quark,\theta) \}_{\theta\in \Theta}$.

\subsection{Bayesian Inconsistency and Model Misspecification}

To quote \cite{Nickl:2012}, ``[w]hile for a Bayesian statistician
the analysis ends in a certain sense with the posterior, one can ask interesting questions about the the properties of posterior-based inference from a frequentist point of view.''  Many of these questions are asymptotic in nature:  for example, in the limit of infinitely many independent $\mu^\dagger$-distributed samples, will the posterior converge in a suitable sense to $\mu^\dagger$ regardless of the initial choice of prior $\pi$?  This property is referred to as \emph{consistency}\footnote{Sometimes the term \emph{frequentist consistency} is used, reflecting the fact that it lies outside the strict Bayesian worldview.};  a general survey of consistency results is found in \cite{WalkerHjort:2001}.  As noted above, the consistency theorem is generically known as the Bernstein--von Mises theorem \cite{Bernstein:1964, vonMises:1964}, although the earliest rigorous proofs are due to Doob \cite{Doob:1949} and Le Cam \cite{LeCam:1953}.

Unfortunately, Cromwell's Rule is only necessary, and not sufficient, to ensure consistency.  In fact, consistency is far from being a generic property, and once the probability space contains infinitely many points (and hence any parameter space $\Theta$ that parametrizes all probability measures on that probability space is infinite-dimensional), inconsistency is not the exception, but the rule \cite{DiaconisFreedman:1998}.  In \cite[Sec.~5]{Freedman:1963}, Freedman considered a countable index set $\mathbb{N} := \{ 1, 2, \dots \}$
 and the parameter space
\[
	\Theta := \left\{ \theta \colon \mathbb{N} \to [0, 1] \smid \sum_{i \in \mathbb{N}} \theta(i) = 1 \right\}.
\]
Each $\theta$ gives rise to a probability distribution $\P_{\theta} = \mu(\quark, \theta)$ under which the observations $X_{1}, X_{2}, \dots$ are IID with $\P_{\theta}[X_{n} = i] = \theta(i)$.  The problem is assumed to be well-specified, so that one particular $\theta^{\dagger} \in \Theta$ is considered to be the ``true'' parameter value, and the frequentist data-generating distribution is $\mu^{\dagger} = \P_{\theta^{\dagger}} = \mu(\quark, \theta^{\dagger})$.  Theorem 5 of \cite{Freedman:1963} shows that, when $\supp (\mu^{\dagger})$ is infinite, given any ``spurious'' probability distribution $\mathbb{Q} = \P_{q}$, there exists a prior probability measure $\pi$ on $\Theta$ that has $\theta^{\dagger}$ in its support, such that the posterior of $\pi$ $\mu^{\dagger}$-a.s.\ concentrates on $q$ in the limit of observing infinitely many i.i.d.\ $\mu^{\dagger}$-distributed samples.  In fact, there is a prior that gives positive mass to every open subset of $\Theta$ but yields consistent posterior estimates for only a first-category set of possible ``true'' (data-generating) parameter values $\theta^{\dagger}$.

There are conditions on priors that do ensure consistency in infinite-dimensional or non-parametric contexts, e.g.\ the tail-free priors introduced by Freedman in \cite{Freedman:1963} and hybrid Bayesian--frequentist tools such as Dirichlet process priors \cite{Ghosal:2010}.  However, while the collection of ``bad'' priors that lead to inconsistent results is measure-theoretically small \cite{Doob:1949, BreimanEtAl:1964}, it is topologically generic \cite{Freedman:1965}.

\begin{rmk}
It is probably fair to say that, despite their popularity and documented successes, Bayesian methods have always attracted some degree of controversy and opposition:  see e.g.~\cite{Gelman:2008a} and rejoinders for a recent academic discussion, and \cite{Malakoff:1999, McGrayne:2012} for less formal treatments.  Often, this opposition is philosophical in nature, particularly with regard to the subjective interpretation of the probabilities involved, which is something that remains counter-intuitive to many commentators:  see \cite[par.~35 \& 37]{EWCA} for a recent example in law.  However, there are also analytical reasons to be careful about the application of Bayesian methods \cite{Senn:2007, Mayo:2012b, Efron:2013}. It is, in fact, now well understood that Bayesian methods may fail to converge or may converge towards the wrong solution if the underlying probability mechanism allows an infinite number of possible outcomes \cite{DiaconisFreedman:1986} and that, in these non-finite-probability-space situations, this lack of convergence (commonly referred to as \emph{Bayesian inconsistency}) is the rule rather than the exception \cite{DiaconisFreedman:1998}.
There is now a wide literature of positive \cite{Bernstein:1964, CastilloNickl:2013, Doob:1949, KleijnVaart:2012, LeCam:1953, vonMises:1964, Stuart:2010} and negative results \cite{Belot:2013, DiaconisFreedman:1986, Freedman:1963, Freedman:1999, Johnstone:2010, Leahu:2011} on the consistency properties of Bayesian inference in parametric and non-parametric settings, and an emerging understanding of the fine topological and geometrical properties that determine (in)consistency.
\end{rmk}

It is important to appreciate that the requirement of positive prior mass in every neighborhood of the true distribution depends upon the topology placed upon $\mathcal{M}(\mathcal{X})$.  For example, Schwartz \cite{Schwartz:1965} shows
that every $\pi$ that puts positive mass on all Kullback--Leibler (relative entropy) neighborhoods of $\mu^{\dagger}$ is weakly consistent.  On the other hand, Freedman \cite{Freedman:1963} and Diaconis \& Freedman \cite{DiaconisFreedman:1986} show that $\pi$ may put positive mass on all weak neighborhoods of $\mu^{\dagger}$ and still fail to be weakly consistent --- e.g.\ by not being tail-free.  Nor are results limited to \emph{weak} convergence of the posterior to $\mu^{\dagger}$.  For example, \cite{BarronEtAl:1999} shows that consistency holds in the Hellinger distance if $\pi$ puts positive mass on all Kullback--Leibler neighborhoods of $\mu^{\dagger}$ and certain smoothness and tail conditions are satisfied;  see \cite{Walker:2004, Wasserman:1998} for further results on Hellinger and Kullback--Leibler consistency.  The amount of prior probability mass that lies Kullback--Leibler-close to the truth, quantified using a notion called \emph{thickness}, can be used to quantify the convergence properties of Bayes estimates \cite{Abrahamcadre:2002, AbrahamCadre:2008, MartinHong:2012}.  However, it is important to note that, in the infinite-dimensional contexts that are increasingly subject to Bayesian analyses, results like the Feldman–-H{\'a}jek dichotomy \cite{Feldman:1958, Hajek:1958} suggest that probability measures are `usually' mutually singular and `rarely' mutually absolutely continuous, and so the Kullback--Leibler neighborhoods of $\mu^{\dagger}$ are `small' sets that are `unlikely' to intersect the model class.

The situation in which there is no $\theta^{\dagger} \in \Theta$ such that $\mu^{\dagger} = \mu(\quark, \theta^{\dagger})$ is referred to as \emph{model misspecification}.  The consistency and other asymptotic properties of misspecified models appear to have first been considered by Berk \cite{Berk:1966, Berk:1970} and Huber \cite{Huber:1967}.  See \cite{KleijnVaart:2006, KleijnVaart:2012} for a recent contribution, and \cite{MartinHong:2012} for convergence rates.

\begin{quotation}
	\noindent ``In practice, Bayesian inference is employed under misspecification \emph{all the time}, particularly so in machine learning applications.  While sometimes it works quite well under misspecification \cite{BleiJordanNg:2003, KleijnVaart:2006}, there are also cases where it does not \cite{Clarke:2004, Fushiki:2005}, so it seems important to determine precise conditions under which misspecification is harmful --- even if such an analysis is based on frequentist assumptions.'' \cite{Grunwald:2006}
\end{quotation}

There is a reasonable popular belief that gross misspecification of the model will be detected by some means before engaging in a serious Bayesian analysis;  indeed there do exist tests \cite{HausmanTaylor:1981, White:1982} for model misspecification, but it is important to note that while one \emph{can} determine that the model is misspecified, one \emph{cannot} be sure that the model is well-specified.  There is also an understandable popular belief that these tests mean that one need only be concerned with the situation of ``mild misspecification'', and that provided $\mu^{\dagger}$ lies ``close enough'' to the model class $\{ \mu(\quark, \theta) \}_{\theta \in \Theta}$,
the posterior estimates will still converge to a usefully informative limit.

\begin{rmk}
This belief echoes G.\ E.\ P.\ Box's statement \cite[p.~424]{Box:1987} that  ``essentially, all models are wrong, but some are useful'' and question \cite[p.~74]{Box:1987} ``Remember that all models are wrong; the practical question is how wrong do they have to be to not be useful?''
\end{rmk}

In terms of the above discussion, one purpose of this paper is to explore the extent to which one can simultaneously have \emph{robust} Bayesian analyses that produce \emph{consistent} answers, given that the models used (both priors and likelihoods) are certain to be \emph{misspecified} to some degree.    Can one be ``just a little bit wrong'' in terms of model misspecification?  Our results suggest that the answer is  negative within the classical framework of Bayesian Sensitivity analysis, when ``closeness'' is measured in terms of total variation and Prokhorov metrics or in terms of a finite (but possibly large) number of marginals of the data generating distribution.

In particular, one aim of Section \ref{Sec:local} is to show that this belief is wrong if ``mild misspecification'' is measured using the Prokhorov or the total variation
metrics, the number of samples is finite (but possibly arbitrarily large), and if convergence is required to hold uniformly in an arbitrarily small neighborhood of the model.

\begin{rmk}
	It is known from the Bernstein--von Mises theorem \cite{Bernstein:1964, vonMises:1964} that, in finite-dimensional situations, posterior values converge towards the quantity of interest if the prior distribution has strictly positive mass in every neighborhood of the truth (see also \cite{LeCam:1953, Nickl:2012}). It is also known that ``even for the simplest infinite-dimensional models, the Bernstein--von Mises theorem does not hold'' \cite{Cox:1993, Freedman:1999}.  This possible lack of convergence, referred to as the consistency problem, has been at the center of a debate between frequentists and Bayesians.  We quote Diaconis and Freedman \cite{DiaconisFreedman:1986} (see also \cite{DiaconisFreedman:1998})
	\begin{quotation}
		\noindent ``If the underlying mechanism allows an infinite number of possible outcomes (e.g., estimation of an unknown probability on the integers), Bayes estimates can be inconsistent:  as more and more data comes in, some Bayesian statisticians will become more and more convinced of the wrong answer.''
	\end{quotation}
	What is the significance of Theorem~\ref{thm:shiva} in that discussion?  To answer this question, consider Example~\ref{eg:shivabis} (and \ref{eg:hiuhdiueh}), in which one is interested in estimating the probability (under the unknown measure $\mu^\dagger$) that $X$ exceeds $a$ after observing $n$ independent samples. We already know from \cite{DiaconisFreedman:1986, Cox:1993} that placing priors on the infinite-dimensional
	space $\mathcal{A}=\mathcal{M}[0,1]$ of probability measures on $[0,1]$ is unlikely to lead to Bayesian posteriors that will converge towards the true value as more and more data comes in.  One strategy to circumvent this lack of convergence would be to consider a finite-dimensional subset of $\mathcal{A}$, i.e.\ a family $(\mu_\lambda)$ of probability measures on $[0,1]$ indexed by a finite-dimensional parameter $\lambda \in \R^k$, put a strictly positive prior $p$ on $\lambda \in \R^k$, and then invoke the Bernstein--von Mises theorem to guarantee the convergence of posterior values.
	
	However,  the Bernstein--von Mises theorem requires that the true distribution under which the data is sampled belongs to $\{\mu_\lambda\mid \lambda \in \R^k\}$, the parametrized finite-dimensional subset of $\mathcal{A}$.  What happens when this is not the case, i.e.\ the situation of \emph{misspecification}?  Write $\pi_p$ for the push-forward of the prior $p$ on $\lambda \in \R^k$ to a prior on $\mathcal{A}$ under the map $\lambda \mapsto \mu_\lambda$.  Assume that the data have been sampled from $\pi^{\dagger} \cdot \Dmap$ where $\pi^{\dagger}$ is the (frequentist) true distribution.  Here Theorem~\ref{thm:shiva}, as illustrated in Example~\ref{eg:shiva3}, can be used to show that the posterior values of the quantity of interest under $\pi_p$ and $\pi^{\dagger}$ may lie near the opposite extreme values of $\Phi$ in $\mathcal{A}$ even if (1) $\pi^{\dagger}$ is a Dirac mass on a measure $\mu^\dagger\in \mathcal{A}$; (2) the number of independent samples is large; and (3) $k$ is large and $k$ moments of $\mu^\dagger$ and $\mu_{\lambda^{\ast}}$ are equal for some $\lambda^{\ast} \in \R^k$.
\end{rmk}

\begin{rmk}
One popular method for detecting failure of convergence under model misspecification is to divide the data into data used for calibrating the parameters of the model and data used for validating the accuracy or predictability of the (calibrated) model.
This approach, oftentimes described as ``frequentist'' \cite{Guyon:2010, bayarri2004},
could be used to validate Bayesian calculations  \cite{Efron:2013}. Although the detection (of the lack of predictability of the model) is asymptotically robust, it requires the availability of sufficient data.
\end{rmk}

\section{Brittleness under Local Misspecification}
\label{Sec:local}

The purpose of this section is to present brittleness results with respect to local perturbations in the total variation and Prokhorov metrics.  Thus, whereas the examples given for Theorem \ref{thm:shiva} highlighted that no finite number of common moments would be sufficient to constrain two priors to give nearby posterior value for the quantity of interest, this section shows that closeness in the TV and Prokhorov metrics is also insufficient to ensure robustness.

We now establish a corollary to the proof of Theorem \ref{thm:shiva} which we will then use to establish
an extreme brittleness theorem for a model with local misspecification.
Recall that, for a map $\Psi \colon \mathcal{A} \to \mathcal{Q}$, a map $\psi \colon \Psi(\mathcal{A})
\to \mathcal{A}$ is called a \emph{section} of $\Psi$ if $\Psi \circ \psi (q)=q$ for all $q \in \Psi(\mathcal{A})$.

\begin{figure}[tp]
	\begin{center}
		\includegraphics[width=0.7\textwidth]{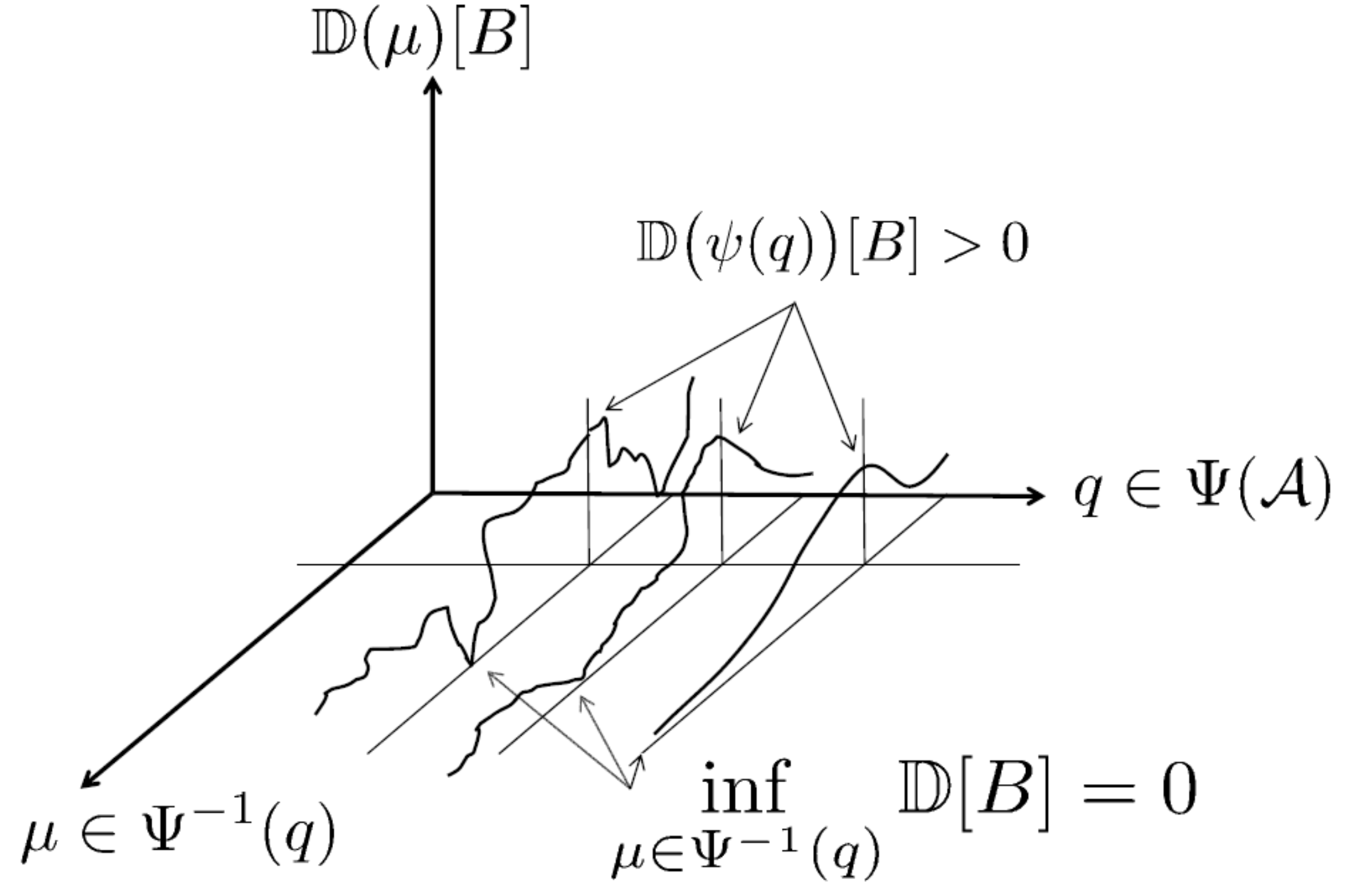}
	\end{center}
	\caption{Illustration of Conditions \eqref{eq:B0} and \eqref{eq:Bp} of Theorem \ref{thm:shiva0cor}.  If, for some data map $\Dmap \in \Dmaps$, all level sets of $\Psi$ go to zero (i.e.\ for
	all $q\in \Psi(\mathcal{A})$, $\inf_{\mu\in  \Psi^{-1}(q)} \Dmap(\mu)[B]=0$), then, for any positive section $\psi$ of $\Psi$ (i.e.\ $\Psi\circ\psi(q)=q$ and
	$ \Dmap(\psi(q))[B]>0$ for $q \in \Psi(\mathcal{A})$), the least upper bound on posterior values is bounded from below by the essential supremum of $\Phi\circ\psi$.}
	\label{fig:localshiva}
\end{figure}

\begin{thm}
	\label{thm:shiva0cor}
	Let $\mathcal{A}$ be a Suslin space, let $\Phi \colon \mathcal{A} \to \R$ be measurable, let  $\mathcal{Q}$ be a separable and metrizable space, and let $\Psi \colon \mathcal{A} \to \mathcal{Q}$  measurable. Let $\mathfrak{Q} \subseteq \mathcal{M}(\mathcal{Q})$ be such that $\supp(\mathbb{Q}) \subseteq \Psi(\mathcal{A})$ for all $\mathbb{Q} \in \mathfrak{Q}$.  Let the data space $\mathcal{D}$ be metrizable and consider $B \in\mathcal{B}(\mathcal{D})$.  Assume that $\Dmap $ is such that all the level sets of $\Psi$ go to zero, in the sense that
	\begin{equation}
		\label{eq:B0}
		\inf_{\mu\in  \Psi^{-1}(q)} \Dmap(\mu)[B]=0, \quad \text{for all }q \in \Psi(\mathcal{A}) .
	\end{equation}
	Then for any positive measurable section $\psi$ of $\Psi$, positive in the sense that
	\begin{equation}
		\label{eq:Bp}
		\Dmap(\psi(q))[B]>0,\quad \text{for all } q \in \Psi(\mathcal{A}) ,
	\end{equation}
	it follows that
    \begin{equation}
    	\label{eq:2qbisjhjycondddihjhlocal}
		\mathcal{U}\big(\Psi^{-1}\mathfrak{Q}\big|B\big) \geq \mathfrak{Q}^{\infty}(\Phi\circ\psi) .
    \end{equation}
	where $\mathfrak{Q}^{\infty}(\Phi\circ\psi)$ is the essential supremum
	\begin{equation}
		\label{eq_esssupmathfrakQ}
		\mathfrak{Q}^{\infty}(\Phi\circ\psi):= \sup_{\mathbb{Q}\in \mathfrak{Q}}  \inf{\bigl\{r \in \R: \mathbb{Q}[\Phi\circ\psi >r]=0\bigr\}} .
	\end{equation}
\end{thm}
\noindent See Figure~\ref{fig:localshiva} for an illustration of Theorem~\ref{thm:shiva0cor}.

We now use Theorem \ref{thm:shiva0cor} to develop a brittleness theorem for a model with local misspecification.  To that end, let $\mathcal{X}$ be a Polish space so that, by \cite[Thm.15.15]{AliprantisBorder:2006}, $\mathcal{M}(\mathcal{X})$ endowed with the weak
topology is Polish. Moreover, by \cite[Thm.~11.3.3]{Dudley:2002}, we know that if we select a complete consistent metric $d$ for $\mathcal{X}$, then the Prokhorov metric $d_{\mathcal{M}}$ defined by
\[
	d_{\mathcal{M}}(\mu_{1},\mu_{2}) := \inf \left\{ \varepsilon > 0 \smid \mu_{1}(A) \leq \mu_{2}(A^{\varepsilon})+\varepsilon \text{ for all } A \in  \mathcal{B}(\mathcal{X}) \right\},
\]
where
\[
	A^{\varepsilon}:= \left\{ x \in \mathcal{X} \smid d(x,x')<\varepsilon  \text{ for some } x' \in A \right\}
\]
is the $\varepsilon$ neighborhood of $A$, metrizes the weak topology on $\mathcal{M}(\mathcal{X})$. Moreover, Prokhorov's theorem
\cite[Cor.~11.5.5]{Dudley:2002} asserts that the Prokhorov metric $d_{\mathcal{M}}$ is a complete metric for the Polish space $\mathcal{M}(\mathcal{X})$.  For $\alpha>0$, $\mu \in \mathcal{M}(\mathcal{X})$, let $B_{\alpha}(\mu):=\{\mu'\in \mathcal{M}(\mathcal{X}) \mid d_{\mathcal{M}}(\mu,\mu') < \alpha\}$ be the open ball of Prokhorov radius $\alpha$ about $\mu$.

Let $\Theta$ be a Polish space and let the model define a map
\[
	\mathcal{P} \colon \Theta \to \mathcal{M}(\mathcal{X}) .
\]
As in Section \ref{sec:misspec}, the image  $\mathcal{P}(\Theta)$ is referred to as the
  (Bayesian) \emph{model class}.

\begin{rmk}
When $\mathcal{P}$ is continuous, it follows from the definition \cite[Sec.~3.2]{Arveson1976} of an analytic set that the the image $\mathcal{P}(\Theta) \subseteq \mathcal{M}(\mathcal{X})$ is analytic, and since the range space $\mathcal{M}(\mathcal{X})$ is Polish it follows that $\mathcal{P}(\Theta)$ is Suslin.  Actually, continuity is not required, since \cite[Thm.~3.3.4]{Arveson1976} implies that if  $\mathcal{P}$ is measurable, then the image $\mathcal{P}(\Theta)$ is Suslin.  If, in addition,  $\mathcal{P}$ is injective, then Suslin's Theorem \cite[Thm.~3.2.3]{Arveson1976} implies that  $\mathcal{P}(\Theta)$ is Borel.
\end{rmk}

Assume that $\mathcal{P}$ is measurable and denote its image by $\mathcal{A}_{0}:=\mathcal{P}(\Theta)$.
Let $\pi_\Theta \in \mathcal{M}(\Theta)$ be a prior distribution on $\Theta$ and let
 $\pi_0:= \mathcal{P}\pi_{\Theta} \in \mathcal{M}(\mathcal{A}_{0})$ be
its pushforward.
 Let $\Phi_{0} \colon \mathcal{M}(\mathcal{X})\to \R$ be a measurable quantity of interest. We are interested in estimating $\Phi_0$ using the prior $\pi_0$ and our purpose is to show the extreme brittleness of this estimation under arbitrarily small perturbations of the model class $\mathcal{A}_0$  in both the Prokhorov and total variation
metrics.

For conditioning on observations, let the data space be $\mathcal{D}:=\mathcal{X}^{n}$, and
consider the $n$-i.i.d.~sample data map $\Dmap^{n}_{0} \colon \mathcal{M}(\mathcal{X}) \to \mathcal{M}(\mathcal{X}^{n})$ defined by
\begin{equation}\label{eq:dmapiid}
	\Dmap^{n}_{0}\mu:=\mu^{n}, \quad \mu \in \mathcal{M}(\mathcal{X})\, .
\end{equation}
For  $x^{n}=(x_{1},\dots,x_{n}) \in \mathcal{X}^{n}$, dropping the notational dependence, denote
the rectangle about $x^{n}$ by
\begin{equation}\label{eq:Bndelta}
	B^{n}_{\delta}:=\prod_{i=1}^{n}{B_{\delta}(x_{i})},
\end{equation}
where $B_{\delta}(x_{i})$ is the open ball of radius $\delta$ about $x_{i}$.
Observe that the prior value of $\Phi_0$ under $\pi_0$ is $\E_{\pi_0}[\Phi_0]$ and its posterior value under the observation $\Ddata \in B^{n}_{\delta}$
is $\E_{\pi_0\odot_{B^{n}_{\delta}} \Dmap^{n}_{0}}[\Phi_0]$.

\begin{figure}[t!]
	\begin{center}
		\includegraphics[width=0.6\textwidth]{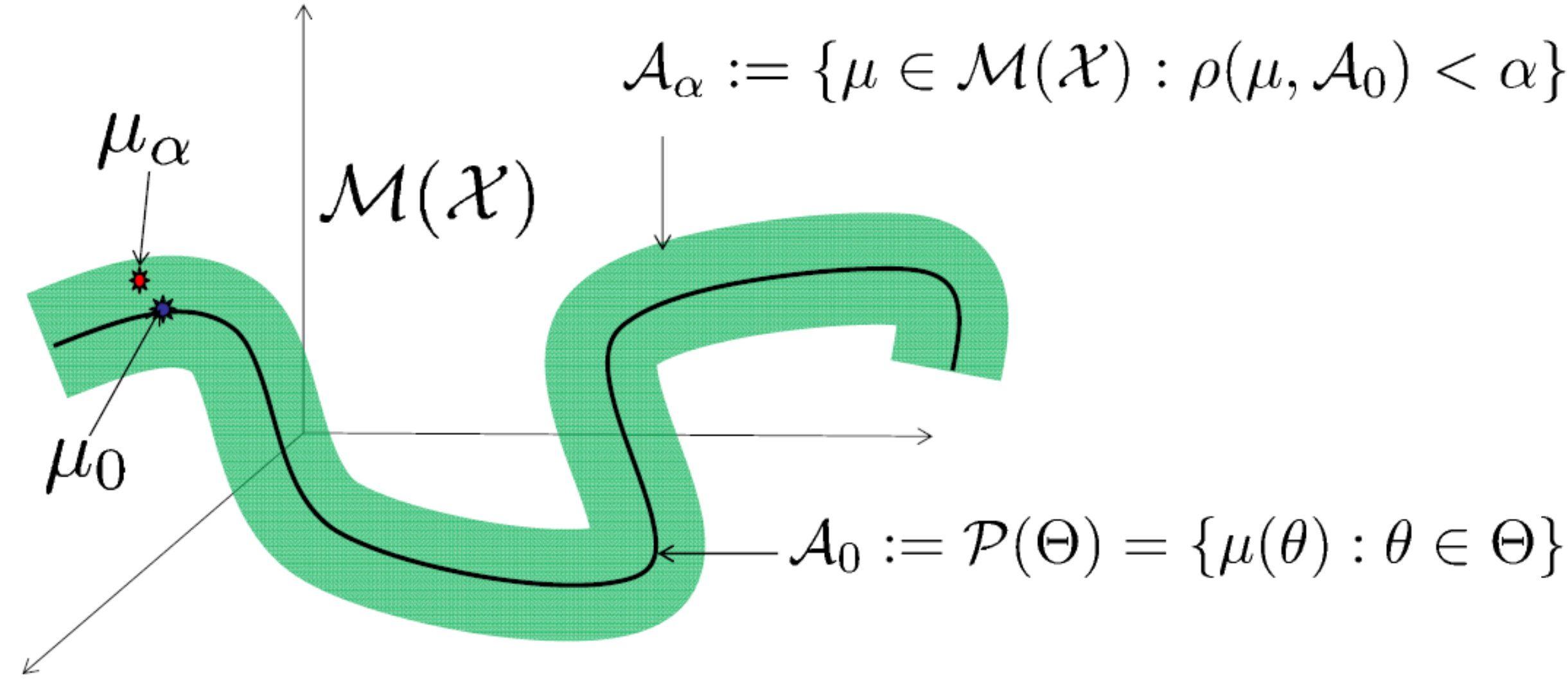}
	\end{center}
	\caption{The original model class $\mathcal{A}_{0}$ (black curve) is enlarged to its metric neighborhood $\mathcal{A}_{\alpha}$ (shaded).  This procedure determines perturbations $\mu_{\alpha} \in \mathcal{A}_{\alpha}$ of the original random measure $\mu_{0} \in \mathcal{A}_{0}$.}
	\label{fig:perturbation}
\end{figure}

To define $\alpha$-perturbations of  the model class $\mathcal{A}_0$ in Prokhorov metric, we introduce, for $\alpha >0$ the
$\alpha$-neighborhood  $\mathcal{A}_{\alpha} \subseteq \mathcal{M}(\mathcal{X})$ of $\mathcal{A}_{0}$ defined by (see Figure \ref{fig:perturbation})
\begin{equation}\label{eq:Aalpha}
	\mathcal{A}_{\alpha}:=\bigcup_{\mu \in \mathcal{A}_{0}} B_{\alpha}(\mu) .
\end{equation}
It is easy to see that the ball fibration (see Remark \ref{rmk:afibration})
\begin{equation}\label{eq:Aup}
	\mathcal{A}:= \left\{ (\mu_{1},\mu_{2}) \in \mathcal{M}(\mathcal{X}) \times \mathcal{M}(\mathcal{X}) \smid \mu_{1}\in \mathcal{A}_{0}, \mu_{2} \in B_{\alpha}(\mu_{1}) \right\}
\end{equation}
of the  set of balls about points of $\mathcal{A}_{0}$ projects to
\begin{align}
	\label{P0}
	P_{0}\mathcal{A}&=\mathcal{A}_{0}\\
	\label{Pa}
	P_{\alpha}\mathcal{A}&=\mathcal{A}_{\alpha}
\end{align}
where $P_{0} \colon \mathcal{M}(\mathcal{X}) \times \mathcal{M}(\mathcal{X})\to \mathcal{M}(\mathcal{X}) $ is the projection onto the first component and $P_{\alpha}$ the projection onto the second.
The  naturally induced set of priors corresponding to   $\pi_0\in \mathcal{M}(\mathcal{A}_0)$ is therefore the set $\Pi_{\alpha} \subset  \mathcal{M}(\mathcal{A}_\alpha)$ defined by
\begin{equation}\label{eq:palphadef}
\Pi_{\alpha}:=\big\{\pi_\alpha \in \mathcal{M}(\mathcal{A}_\alpha) \big| \exists \pi \in \mathcal{M}(\mathcal{A})\text{ with } P_0\pi=\pi_0\text{ and }P_\alpha\pi=\pi_\alpha\big\}\, .
\end{equation}
\begin{rmk}\label{rmk:interpretpialpha}
Observe that each element $\pi_\alpha \in \Pi_{\alpha}$ is the distribution of  a random measure $\mu_2$ on $\mathcal{A}_\alpha$ such that: (i) there exists a random measure $\mu_1\in \mathcal{A}_0$ with distribution $\pi_0$ (that of the model); (ii) $(\mu_1,\mu_2)$ is jointly measurable; and
(iii) with probability one the Prokhorov distance from $\mu_2$ to $\mu_1$ is less than $\alpha$, i.e.\ $d_{\mathcal{M}}(\mu_1,\mu_2)< \alpha$.
Observe in particular that $\pi_0\in \Pi_\alpha$.
\end{rmk}

Our main result is provided in Theorem \ref{thm_localshiva} but for the sake of clarity we will first give this result in the following (simpler) form.

\begin{thm}
	\label{thm_localshivacor}
Using the notations introduced above and the data map \eqref{eq:dmapiid}, let $\Pi_{\alpha}$ be defined as in \eqref{eq:palphadef}.
If
\begin{equation}\label{eq:2jhkjhjkekj3d}
	\lim_{\delta \downarrow 0} \sup_{x \in \mathcal{X}} \sup_{\theta\in \Theta} \mathcal{P}(\theta)[B_{\delta}(x)]=0,
\end{equation}
then, for all $\alpha>0$ there exists $\delta_c(\alpha)>0$ such that for all $0<\delta<\delta_c(\alpha)$, all $n \in \mathbb{N}$, and all $(x_1,\ldots,x_n)\in \mathcal{X}^n$
	\[
		\mathcal{U}(\Pi_{\alpha}|B^{n}_{\delta})\geq \operatorname{ess\,sup}_{\pi_0} (\Phi_{0}) ,
	\]
where
\[
	\operatorname{ess\,sup}_{\pi_0} (\Phi_{0}):=\inf\{r>0\mid \pi_0 [\phi_0>r]=0\} ,
\]
and with similar expressions for the lower bounds $\mathcal{L}$.
\end{thm}

\begin{rmk}
Theorem \ref{thm_localshivacor} implies the extreme brittleness of Bayesian inference under local misspecification.
Indeed, assume that the model class $\mathcal{A}_0$ is well specified (i.e.\ it
contains the truth $\mu^\dagger$) and that, therefore, the Bayesian estimator described by $\pi_0$ is consistent.
One may believe that a model $\mathcal{A}_1$ lying in a `small enough'
neighborhood of $\mathcal{A}_{0}$ should have good convergence properties, Theorem \ref{thm_localshivacor} and
Remark \ref{rmk:interpretpialpha} invalidate this belief, at least as far as the TV and Prokhorov notions of `small enough' are concerned.
Using the notations of Remark \ref{rmk:interpretpialpha},
observe in particular that an unscrupulous practitioner may design a model corresponding to a random measure  $\mu_2$ such that the distance between $\mu_1$ (the well specified model) and $\mu_2$ is a.s.~at most $\alpha$ (where $\alpha$ is arbitrarily small) and the posterior value using the random measure $\mu_2$ is as distant as possible from the posterior value using $\mu_1$ irrespective of the sample size $n$.
\end{rmk}

\begin{rmk}
Observe that the condition \eqref{eq:2jhkjhjkekj3d} is extremely weak and satisfied for most Bayesian models. This condition can in fact be made weaker by
replacing it with the assumption that  for $n$ sufficiently large it holds true that for all $\theta$, $\mathcal{P}(\theta)$ does not contain a Dirac mass
in each ball $B_{\delta}(x_{i})$ (i.e.\ on
the sample data when $\delta \downarrow 0$). We also note that the proof of  Theorem \ref{thm_localshivacor} does not require the samples to be i.i.d., in particular, the same results can be obtained with coupled samples, if, for instance, the data map $\Dmap^{n}_{0}$ is replaced by a data map $\Dmap$ such that
$C_1^n \prod_{i=1}^n \mu(A_i) \leq  \Dmap(\mu)[A_1\times \cdots\times A_n] \leq C_2^n \prod_{i=1}^n \mu(A_i)$ for strictly positive constants $C_1$ and $C_2$.
\end{rmk}

\begin{rmk}\label{rmk:tvrmk}
	Theorem \ref{thm_localshivacor} is a corollary  of Theorem \ref{thm_localshiva} and the proof of Theorem \ref{thm_localshiva} shows that, if  $\Theta$ is compact and $\mathcal{P}$ is continuous and  $\Phi(\mu):=\mu(A)$ for some fixed $A \in \mathcal{B}(\mathcal{X})$, then the result of Theorem \ref{thm_localshivacor} also holds when using the total variation distance $d_{\mathrm{TV}}$ instead of the Prokhorov distance, which produces a much smaller neighborhood.
	
	However, in this metric $\mathcal{M}(\mathcal{X})$ in general is not separable and this introduces measurability difficulties.  These difficulties can be overcome somewhat when $\Theta$ is compact and $\mathcal{P}$ is continuous, since the image of a compact set under a continuous map is compact and therefore measurable. Moreover, validation or certification type quantities of interest defined by $\Phi(\mu):=\mu(A)$ for some fixed $A \in \mathcal{B}(\mathcal{X})$ are easily seen to be continuous and therefore measurable. Moreover, because of continuity,
	\[
		\Pi_{0}^{\infty}(\Phi_{0})\approx\Pi_{\alpha}^{\infty}(\Phi_{0})  .
	\]
	Our motivation in working mainly with the Prokhorov metric lies in the fact that we also seek to lay down measurability foundations for the scientific computation of optimal statistical estimators where  the unknown quantities are products of functions and measures and for such spaces the total variation metric is too strong for the measurability of standard quantities of interest.
\end{rmk}

We will now give a more general version of Theorem \ref{thm_localshivacor} and elaborate on the objects entering in its
formulation.
We start with $\Pi_{\Theta}\subseteq \mathcal{M}(\Theta)$, a set of admissible priors and let
\[
	\Pi_{0}:=\mathcal{P}\Pi_{\Theta} \subseteq \mathcal{M}(\mathcal{A}_{0})
\]
denote the push-forward by the model $\mathcal{P}$.
We consider the pull-back $\Phi_{\Theta}:=\Phi_{0} \circ \mathcal{P}$, of the measurable quantity of interest  $\Phi_{0} \colon \mathcal{M}(\mathcal{X})\to \R$, to a measurable quantity of interest $\Phi_{\Theta} \colon \Theta\to \R$.  Then the change of variables formula \cite[Thm.~4.1.11]{Dudley:2002} implies that, for $\pi_{\Theta}\in \mathcal{M}(\Theta)$,
\[
	\E_{\pi_{\Theta}}[\Phi_{\Theta}]=\E_{\pi_{\Theta}}[\Phi_{0} \circ \mathcal{P}] = \E_{\mathcal{P}\pi_{\Theta}}[\Phi_{0}]
\]
whenever either side is well defined.  Therefore, taking  suprema and infima,
we obtain
\begin{align*}
	\mathcal{U}(\Pi_{\Theta})&= \mathcal{U}(\Pi_{0}),\\
	\mathcal{L}(\Pi_{\Theta})&= \mathcal{L}(\Pi_{0}),
\end{align*}
where we note that the quantity of interest implicit in these definitions is determined by the
argument. For $\alpha>0$, define $\mathcal{A}_{\alpha}$, $\mathcal{A}$, $P_0$   and $P_\alpha$ as in \eqref{eq:Aalpha}, \eqref{eq:Aup}, \eqref{P0} and \eqref{Pa}.
\begin{rmk}\label{rmk:afibration}
Using the affine convexity of $\mathcal{M}(\mathcal{X})$, one can show that $\mathcal{A}$ is indeed a Hurewicz fibration, in that it has the homotopy lifting property, see e.g.~\cite[p.~66]{Spanier}.  Let
\[
	d_{\mathcal{M}}^{-1}(< \alpha) := \{(\mu_{1},\mu_{2}) \mid d_{\mathcal{M}}(\mu_{1},\mu_{2})< \alpha\}
\]
denote the set of all pairs of measures at Prokhorov distance at most $\alpha$ from one another.
Since $d_{\mathcal{M}} \colon \mathcal{M}(\mathcal{X})\times\mathcal{M}(\mathcal{X}) \to \R$ is continuous, it follows that $d_{\mathcal{M}}^{-1}(< \alpha)$ is open and therefore Borel. In addition, since $\mathcal{A}_{0} \subseteq \mathcal{M}(\mathcal{X})$ is Suslin it follows that $\mathcal{A}_{0} \times  \mathcal{M}(\mathcal{X}) \subseteq \mathcal{M}(\mathcal{X}) \times \mathcal{M}(\mathcal{X})$ is Suslin. Therefore, since $\mathcal{A}=d_{\mathcal{M}}^{-1}(<\alpha) \cap \bigl(\mathcal{A}_{0} \times  \mathcal{M}(\mathcal{X})\bigr)$,
it follows that $\mathcal{A}$ is Suslin.
\end{rmk}

Observe that the measurable quantity of interest $\Phi_{0} \colon \mathcal{M}(\mathcal{X}) \to \R$ acting on the second component of $\mathcal{M}(\mathcal{X})\times\mathcal{M}(\mathcal{X})$, naturally pulls back
to the quantity of interest $\Phi \colon \mathcal{M}(\mathcal{X}) \times \mathcal{M}(\mathcal{X}) \to \R$ by $\Phi:=\Phi_{0}\circ P_{\alpha}$, and we have $\sup_{\mathcal{A}_{\alpha}}{\Phi_{0}}= \sup_{\mathcal{A}}{\Phi}$ and $\inf_{\mathcal{A}_{\alpha}}{\Phi_{0}}= \inf_{\mathcal{A}}{\Phi}$,
i.e.
\begin{align*}
	\mathcal{U}(\mathcal{A}_{\alpha})&= \mathcal{U}(\mathcal{A}) ,\\
	\mathcal{L}(\mathcal{A}_{\alpha})&= \mathcal{L}(\mathcal{A}) .
\end{align*}

For a subset $\Pi_{0} \subseteq \mathcal{M}(\mathcal{A}_{0})$, the projection identity \eqref{P0}
implies that the set $\Pi:=P_{0}^{-1}\Pi_{0}$ defined by $P_{0}^{-1}\Pi_{0} :=\{\pi \in \mathcal{M}(\mathcal{A}) \mid P_{0}\pi \in \Pi_{0}\}$ is the induced set of probability measures on $\mathcal{A}$.  Moreover, for  $\pi \in \Pi$, the change of variables formula
\[
	\E_{\pi}[\Phi]=\E_{\pi}[\Phi_{0} \circ P_{\alpha}]=\E_{P_{\alpha}\pi}[\Phi_{0}]
\]
implies that
\begin{align*}
 	\sup_{\pi \in \Pi}{\E_{\pi}[\Phi]}
 	&= \sup_{\pi_{\alpha} \in P_{\alpha}\Pi}{\E_{\pi_{\alpha}}[\Phi_{0}]},\\
	\inf_{\pi \in \Pi}{\E_{\pi}[\Phi]}
	&= \inf_{\pi_{\alpha} \in P_{\alpha}\Pi}{\E_{\pi}[\Phi_{0}]},
\end{align*}
so that
\[
	P_{\alpha}\Pi=P_{\alpha}P_{0}^{-1}\Pi_{0} \subseteq \mathcal{M}(\mathcal{A}_{\alpha})
\]
is the induced set of probability measures on $\mathcal{A}_{\alpha}$.  Let us denote this induced set by
\begin{equation}\label{eq:palphadef2}
	\Pi_{\alpha}:=P_{\alpha}P_{0}^{-1}\Pi_{0} ,
\end{equation}
so that these equalities become
\begin{align*}
	\mathcal{U}(\Pi)&= \mathcal{U}(\Pi_{\alpha}) ,\\
	\mathcal{L}(\Pi)&= \mathcal{L}(\Pi_{\alpha}) .
\end{align*}

For conditioning on observations, define $\Dmap^{n}_{0}$ as in \eqref{eq:dmapiid}
and pull it back to the data map $\Dmap^{n} \colon \mathcal{M}(\mathcal{X}) \times \mathcal{M}(\mathcal{X}) \to \mathcal{M}(\mathcal{X}^{n})$ defined by $\Dmap^{n}:=\Dmap^{n}_{0}\circ P_{\alpha}$. Define $B^{n}_{\delta}$ as in \eqref{eq:Bndelta} and recall
the definition \eqref{eq:defcondexp}
\[
	\E_{\pi \odot \Dmap^{n}}\big[ \Phi\big| B^{n}_{\delta}\big]=\frac{\E_{(\mu_{1},\mu_{2})\sim \pi}\big[  \Phi(\mu_{1},\mu_{2}) \Dmap^{n}(\mu_{1},\mu_{2})[B^{n}_{\delta}] \big]}{\E_{(\mu_{1},\mu_{2})\sim\pi}\big[\Dmap^{n}(\mu_{1},\mu_{2})[B^{n}_{\delta}]\big]} .
\]
of the conditional expectation and the corresponding \eqref{eq:Upostes} upper value
\[
	\mathcal{U}(\Pi|B^{n}_{\delta}):=\sup_{\pi\odot \Dmap^{n} \in \Pi\odot_{B^{n}_{
\d}}\Dmap^{n}}{ \E_{\pi \odot \Dmap^{n}} \big[ \Phi\big| B^{n}_{\delta}\big]}
\]
in terms of the admissible set \eqref{def_B}
\[
	\Pi_{B^{n}_{\delta}} := \Bigl\{  \pi\in  \Pi : (\pi \cdot \Dmap^{n})[B^{n}_{\delta}]> 0 \Bigr\}\,
\]
of product measures, where the marginal is defined by
\[
	(\pi \cdot \Dmap^{n})[B^{n}_{\delta}] := \E_{(\mu_{1},\mu_{2})\sim\pi}\big[\Dmap^{n}(\mu_{1},\mu_{2})[B^{n}_{\delta}]\bigr] .
\]

Let us indicate the dependence on some measure $\underline{\pi}$ of the essential supremum of some quantity of interest $\underline{\Phi}$
by
\[
	\underline{\pi}^{\infty}(\underline{\Phi}) := \inf \left\{r \in \R \smid \underline{\pi}\{\underline{\Phi} >r\}=0 \right\}
\]
and, for a set $\underline{\Pi}$ of measures, let
\begin{equation}
	\label{eq_esssup}
	\underline{\Pi}^{\infty}(\underline{\Phi}):= \sup_{\underline{\pi} \in \underline{\Pi}} \underline{\pi}^{\infty}(\underline{\Phi}) .
\end{equation}
For $\pi_{\alpha}=P_{\alpha}\pi$ with $\pi \in \Pi$, we have
\begin{align*}
	\pi_{\alpha}[\Phi_{0} >r]
	&= \bigl(P_{\alpha}\pi\bigr)[\Phi_{0} >r]\\
	&= \pi[\Phi_{0}\circ P_{\alpha} >r]\\
	&= \pi[\Phi >r]
\end{align*}
so that we conclude that
\[
	\Pi^{\infty}(\Phi) = \Pi_{\alpha}^{\infty}(\Phi_{0}) .
\]

Let us now quantify a type of regularity for the model $\mathcal{P}$.  For $x \in \mathcal{X}$, let $B_{0}(x):=\{x\}$ and define
\[
	\mathcal{P}_{\infty}(\delta) := \sup_{x \in \mathcal{X}} \sup_{\theta\in \Theta} \mathcal{P}(\theta)[B_{\delta}(x)], \quad \text{for } \delta \geq 0 .
\]
It is clear that $\mathcal{P}_{\infty} \colon \R^{+} \to [0,1]$ is an increasing function. Moreover, for most parametric families, it is easy to show that $\mathcal{P}_{\infty}$ is continuous and $\mathcal{P}_{\infty}(0)=0$, and for many of them not difficult to find useful upper bounds.

Finally, let us assume that the model $\mathcal{P}$ is positive, in that $\mu(B_{\delta}(x)) > 0$ for all $\mu \in \mathcal{A}_{0}$, $x \in\mathcal{X}$,  and $\delta >0$.
Theorem \ref{thm_localshivacor} is a  direct consequence of the following theorem.

\begin{thm}[Brittleness under Local Misspecification]
	\label{thm_localshiva}
	With the notation and assumptions above, let $\Pi_{\alpha}$ be defined as in \eqref{eq:palphadef2}, and let $\delta>0 $ and $0 < \alpha < 1$  satisfy
	\[
		\mathcal{P}_{\infty}(\d) < \alpha .
	\]
	Then, using $\Dmap^{n}_{0}$ for the distribution of the data, for all integers $n \geq 1$,
	\[
		\mathcal{U}(\Pi_{\alpha}|B^{n}_{\delta})\geq \Pi_{0}^{\infty}(\Phi_{0})
	\]
	with similar expressions for the lower bounds $\mathcal{L}$.
\end{thm}

\begin{rmk}
When Cromwell's rule (see  Section \ref{sec:misspec}) is implemented (i.e.\ if
the prior measure of every non-empty neighborhood is strictly positive), it follows that $\Pi_{0}^{\infty}(\Phi_{0})=\mathcal{U}(\mathcal{A}_{0})$ so that the
conclusion of Theorem \ref{thm_localshiva} becomes
 \[
                \mathcal{U}(\Pi_{\alpha}|B^{n}_{\delta})\geq \mathcal{U}(\mathcal{A}_{0})\, .
        \]

\end{rmk}
\begin{rmk}
	Theorem \ref{thm_localshiva} provides conditions sufficient to guarantee how bad things can get regardless of how many samples are taken. One might hope that when these conditions are not satisfied, that more samples may prove beneficial.  However, when the condition
	\[
		\inf_{(\mu,\mu') \in  \Psi^{-1}\mu} \Dmap^{n}(\mu,\mu')[B^{n}_{\delta}]
=  0, \quad \mu \in \mathcal{A}_{0}
	\]
	of Theorem \ref{thm:shiva0cor}  is only approximately satisfied, the inequality
	\[
		\Dmap^{n}(\mu,\mu')[B^{n}_{\delta}] =\bigl(\mu'\bigr)^{n}[B^{n}_{\delta}]=\prod_{i=1}^{n}{\mu'[B_{\delta}(x_{i})]}
	\]
	and the quantitative version of Theorem \ref{thm:shiva} (given in \cite[Thm.~3.1]{OwhadiScovel:2013}, see also \cite[Rmk.~3.2]{OwhadiScovel:2013})
	imply that things actually get `worse'
	with more samples.

\end{rmk}

\section{Conclusions and Further Developments}\label{sec:conclusions}

In this paper, we have looked at the robustness of Bayesian Inference in the classical framework of Bayesian Sensitivity Analysis.
In that (classical) framework, the data is fixed, and one computes optimal bounds on (i.e.~the sensitivity of) posterior values with respect to variations of the prior in a given class of priors.
Although robustness is already well established  when the class of priors is finite dimensional,
  we  observe that, under general conditions,  when the class of priors is finite codimensional, the optimal bounds on
 posterior values are as large as possible, no matter the number of data points.
Our motivation for specifying a finite codimensional class of priors is to look at what classical Bayesian sensitivity analysis would conclude under finite information, and the best way to understand this notion of ``brittleness under finite information''  is through the simple example provided in Subsection \ref{subsecex1}.

The mechanism causing this ``brittleness'' has its origin in the fact that, in classical Bayesian sensitivity analysis, optimal bounds on posterior values are computed  after  the observation of the specific value of the data, and that the probability of observing the data under some feasible prior may be arbitrarily small (the example given in Subsection \ref{subsecex2} provides an illustration of this phenomenon). This data dependence of {\em worst priors} is inherent to this classical framework and the resulting brittleness under finite information can be seen as an extreme occurrence of the dilation phenomenon (the fact that optimal bounds on prior values may become less precise after conditioning) observed in classical robust Bayesian inference \cite{WassermanSeidenfeld:1994}.
Although these worst priors do depend on the data, ``look nasty'', and make the probability of observing the data very small,
they are not ``isolated pathologies'' but directions of instability (of Bayesian conditioning) and their number increase with the number of data points.
The example given in Subsection \ref{subsec:learningvsrobustness} provides an illustration of this point and also suggests that learning and robustness are, to some degree, antagonistic properties: a strong constraint on the probability of the data makes the method robust but learning impossible and, as the constraint is relaxed, learning becomes possible but posterior values become brittle.

Since ``brittleness under finite information'' appears to be inherent to classical Bayesian Sensitivity Analysis (in which worst priors are computed given the specific value of the data), one may ask whether robustness could be established under finite information by exiting the strict framework of Robust Bayesian Inference and computing the sensitivity of posterior conclusions independently of the specific value of the data.
To investigate this question, Hampel and Cuevas'  notion of {\em qualitative robustness} has been generalized in \cite{OwhadiScovel:2014}
   to Bayesian inference based on the quantification of the {\em sensitivity of the distribution of the posterior distribution} with respect to perturbations of the prior and the data generating distribution, in the limit when the number of data points grows towards infinity. Note that, contrary to classical Bayesian Sensitivity Analysis considered here, in the qualitative formulation the data is not fixed
 and posterior values are therefore analyzed as dynamical systems randomized through the distribution of the data.
To express finite information,  the total variation, Prokhorov, and Ky Fan metrics have been used  to quantify perturbations and sensitivities.

Since this notion of {\em qualitative robustness} is established in the limit when the number of data points grows towards infinity, it is natural to expect that the notion of {\em consistency} (i.e.~the property that posterior distributions convergence towards the data generating distribution) will play an important role.
Although consistency is primarily a frequentist notion, it is also equivalent to {\em intersubjective agreement}
which means that two Bayesians will ultimately have very close predictive distributions. Therefore, it also has importance for Bayesians.
Fortunately,
 not only are there  mild conditions which
 guarantee consistency, but the  Bernstein--von Mises theorem goes further in providing mild conditions under which the posterior
is asymptotically normal. The most famous of these are Doob \cite{Doob:1949}, Le Cam and Schwartz \cite{Lecam_necessary},
 and Schwartz \cite[Thm.~6.1]{Schwartz:1965}.
 Moreover, the assumptions needed for this consistency are so mild that one can be lead to the conclusion that the prior does not really matter once there is enough data. For example, we quote
 Edwards, Lindeman and Savage \cite{Edwards_bayesian}:
\begin{quotation}
        \noindent
``Frequently,
the data so completely
control your posterior opinion that
there is no practical need to attend to
the details of your prior opinion.''
\end{quotation}

To some, the consistency results appeared to generate more confidence than possibly they should.
 We quote A.~W.~F.~Edwards \cite[p.~60]{Edwards}:
\begin{quotation}
        \noindent
 ``It is sometimes said, in defence of the Bayesian concept, that the choice of prior distribution is unimportant in practice,
 because it hardly influences the posterior distribution at all when there are moderate amounts of data.
 The less said about this 'defence' the better.''
\end{quotation}

\cite{OwhadiScovel:2014} shows that the {\em Edwards defence}  is essentially what
 produces  non \emph{qualitative robustness} in Bayesian inference. In particular, the assumptions required for consistency (e.g.~the assumption that the prior has Kullback--Leibler support at the parameter value generating the data) are such that arbitrarily small local perturbations of the prior distribution (near the data generating distribution) results in consistency or non-consistency, and therefore have large impacts on the asymptotic behavior of posterior distributions.
 These mechanisms  are different and complementary to those discovered by Hampel   and developed by Cuevas, and they suggest that consistency and robustness are, to some degree, antagonistic requirements (a careful selection of the prior is important if both properties, or their approximations, are to be achieved) and also indicate that misspecification generates non {\em qualitative robustness}.\\

In conclusion, the exploration of Bayesian inference in a continuous world has revealed both positive and negative results. However, positive results regarding the classical or qualitative robustness
of Bayesian inference under finite information have yet to be obtained. To that end,
 observe that the example provided in Subsection  \ref{subsec:learningvsrobustness} suggests that there may be a missing stability condition for Bayesian inference in a continuous world under finite information akin to the CFL condition for the stability of a discrete numerical scheme used to approximate a continuous PDE. Although numerical schemes that do not satisfy the CFL condition may look grossly inadequate, the existence of such perverse examples certainly does not imply the dismissal of the necessity of a stability condition. Similarly, although one may, as in the example provided in Subsection \ref{subsecex2}, exhibit grossly perverse worst priors, the existence of such priors does not invalidate the need for a study of stability conditions for using Bayesian Inference under finite information. The example of Subsection \ref{subsec:learningvsrobustness} suggests that, in the framework of Bayesian Sensitivity Analysis, under finite information, such a stability condition would strongly depend on how well the probability of the data is  known or  constrained in the  model class in addition to the class of priors and the resolution of the measurements.
 It is natural to expect that such robustness and stability questions will increase in importance
as Bayesian methods increase in popularity due to the availability of computational methodologies and environments to compute the posteriors.
Indeed, when posterior distributions are approximated using such methods, the robustness analysis   naturally
 includes not only quantifying sensitivities with respect to the data generating distribution and the choice of prior, but also the analysis of convergence and stability of the computational method. This is particularly true in Bayesian updating where Bayes' rule is applied iteratively and computed  posterior distributions become prior distributions for the next iteration.
 Oftentimes these  posterior distributions (which are then treated as prior distributions) are only approximated (e.g. via MCMC methods) and the Brittleness results discussed here and in \cite{OwhadiScovel:2014} suggest that having strong convergence (of these MCMC methods) in TV would not be enough to ensure stability. At a higher level, these results appear to suggest that robust inference (in a continuous world under finite information) should be done with reduced/coarse models rather than highly sophisticated/complex models (and the level of ``coarseness/reduction'' would depend on the available ``finite information'').

\section{Proofs}
\label{Sec:proofs}

\subsection{Proof of Theorem \ref{thm_primred}}

For $q \in \R^{n}$, define
\[
	\Pi(q) := \Psi^{-1}\mathfrak{Q}= \left\{ \pi \in \mathcal{M}(\mathcal{A}) \smid \E_{\pi}[\Psi]=q \right\}
\]
and let $\Pi(q,n):= \Pi(q) \cap \Delta(n) \subseteq \Pi(q)$ be the subset consisting of $(n+1)$-fold convex combinations of Dirac masses. Using a `layercake'
approach, we use the fact that
\[
	\Pi(Z) = \bigcup_{q \in Z}  \Pi(q) \quad \text{and}\quad  \Pi(Z,n) =\bigcup_{q \in Z} \Pi(q,n) ,
\]
while applying Theorem~\ref{thm:priorreduce} with equality constraints $\Pi(q), q \in \R^{n}$,
and the fact that the supremum over a union is a supremum of suprema to obtain a reduction as follows:
\begin{align*}
	\mathcal{U}\bigl( \Pi(Z) \bigr)
	&= \mathcal{U} \left( \bigcup_{q\in Z}\Pi(q) \right)\\
	&= \sup_{q \in Z} \mathcal{U}\big(\Pi(q)\big) \\
	&= \sup_{q \in Z} \mathcal{U}\big(\Pi(q,n)\big) \\
	&= \mathcal{U}\left(\bigcup_{q\in Z} \Pi(q,n)\right)\\
	&= \mathcal{U}\big(\Pi(Z,n)\bigr),
\end{align*}
which completes the proof.  \hfill\qedsymbol

\subsection{Proof of Lemma \ref{lem:Umeasurable}}

Since $T \subset \mathcal{Q}$ is a subset of a separable metrizable space, \cite[Cor.~3.5]{AliprantisBorder:2006} implies that it is itself separable and metrizable. Consider the set-valued map with non-empty values $\Psi^{-1} \colon T \twoheadrightarrow \mathcal{A}$ with graph $G$ defined by
\begin{equation}
	\label{eq:graphG}
	G := \left\{ (q,\mu)\in T \times \mathcal{A} \smid \Psi(\mu)=q \right\} .
\end{equation}
Let $d$ be a metric that generated the topology of $T$ and define $h \colon T \times \mathcal{A} \to \R$ by
$h(q,\mu):=d(\Psi(\mu),q)$.  Then, since $d$ is continuous in each of its arguments,
it follows that $h$ is a Carath\'{e}odory function, as defined in Definition~\ref{def:caratheodory}.  Since
$T$ is separable and metrizable, Lemma~\ref{def:caratheodorylemma} implies that $h$ is $\mathcal{B}(T) \otimes \mathcal{B}(\mathcal{A})$-measurable. Rewriting Equation \eqref{eq:graphG} as
\[
	G:= \left\{(q,\mu)\in T \times \mathcal{A} \smid h(q,\mu)=0 \right\}
\]
yields that $G$  belongs to $\mathcal{B}(T)\otimes \mathcal{B}(\mathcal{A})$.  Lemma~\ref{lemIII39CstaingValafier} (through the identification $S=\mathcal{A}$, $s=\mu$, $\varphi(t, s)=\Phi(\mu)$)  implies that the function $\mathcal{U}\circ \Psi^{-1} \colon T \to \R$ defined for $q \in T$ by $q \mapsto \sup_{\mu\in \Psi^{-1}(q)} \Phi(\mu)$ is $\widehat{\mathcal{B}}(T)$-measurable, thereby establishing the first assertion.  The second assertion then follows from the second part of Lemma~\ref{lemIII39CstaingValafier}.  \hfill\qedsymbol

\subsection{Proof of Theorem \ref{thm:reducpriormarg}}

For the first assertion, consider $\mathbb{Q} \in \mathfrak{Q}$.  Then, by the second assertion of Lemma~\ref{lem:Umeasurable}, there exists a $\widehat{\mathcal{B}}(\supp \mathbb{Q})$-measurable section $\psi$ of $\Psi$, i.e.~a $\widehat{\mathcal{B}}(\supp(\mathbb{Q}))$-measurable function $\psi \colon \supp(\mathbb{Q})
\to \mathcal{A}$ such that $\Psi(\psi(q)) = q$ for all $q \in \supp(\mathbb{Q})$.  Let $\mathbb{Q}$ also denote its restriction to its support and $\widehat{\mathbb{Q}}$ its completion.  Let $\pi:=\psi\widehat{\mathbb{Q}} \in \mathcal{M}(\mathcal{A})$, so that, for all $A \in \mathcal{B}(\supp(\mathbb{Q}))$,
\begin{align*}
	\bigl(\Psi\pi\bigr)(A)
	&= \bigl(\Psi\psi\widehat{\mathbb{Q}}\bigr)(A)\\
	&= \bigl((\Psi\circ \psi)\widehat{\mathbb{Q}}\bigr)(A)\\
	&= \widehat{\mathbb{Q}}(A)\\
	&= \mathbb{Q}(A).
\end{align*}
Hence, $\Psi\pi=\mathbb{Q}$, establishing the first assertion.

For the main assertion, observe that, for all $\mu \in \mathcal{A}$,
\begin{equation}
	\label{eq-pull}
	\bigl(\mathcal{U}\circ \Psi^{-1} \circ \Psi\bigr)(\mu) =\sup_{\mu': \Psi(\mu')=\Psi(\mu)} \Phi(\mu') \geq \Phi(\mu)  .
\end{equation}
Consequently, for $\mathbb{Q} \in \mathfrak{Q}$, the first assertion shows that there is a  $\pi$  such that $\Psi\pi=\mathbb{Q}$, so that a change of variables (Proposition~\ref{prop_change}) and the monotonicity properties
(Proposition~\ref{prop_E}) of these integrals, together with the inequality \eqref{eq-pull}, imply that
\begin{align*}
	\E_{\mathbb{Q}}[\mathcal{U}\circ \Psi^{-1}]
	&= \E_{\Psi\pi}[\mathcal{U}\circ \Psi^{-1}]\\
	&= \E_{\pi}[\mathcal{U}\circ \Psi^{-1} \circ \Psi]\\
	&\geq \E_{\pi}[\Phi] ,
\end{align*}
and therefore
\[
	\E_{\mathbb{Q}}[\mathcal{U}\circ \Psi^{-1}] \geq  \sup_{\pi \in \Psi^{-1}\mathbb{Q}} \E_{\pi}[\Phi] .
\]
Consequently,
\[
	\sup_{\mathbb{Q} \in \mathfrak{Q}} \E_{\mathbb{Q}}[\mathcal{U}\circ \Psi^{-1}] \geq \sup_{\pi \in \Psi^{-1}\mathfrak{Q}} \E_{\pi}[\Phi] = \mathcal{U}(\Psi^{-1}\mathfrak{Q})
\]
and, in particular,
\begin{equation}
	\label{eq-e1}
	\sup_{\mathbb{Q} \in \mathfrak{Q}} \E_{\mathbb{Q}}[\mathcal{U}\circ \Psi^{-1}] \geq \mathcal{U}(\Psi^{-1}\mathfrak{Q}) .
\end{equation}

To obtain the reverse inequality, for $\delta>0$  consider $\mathbb{Q} \in \mathfrak{Q}$ and apply Lemma \ref{lem:Umeasurable} to conclude that there exists a $\delta$-optimal $\widehat{\mathcal{B}}(\supp(\mathbb{Q}))$-measurable section of $\Psi$; that is,  a $\widehat{\mathcal{B}}(\supp(\mathbb{Q}))$-measurable function $\psi \colon \supp(\mathbb{Q}) \to\mathcal{A}$ such that $\Psi\big(\psi(q)\big)=q$ for all $q \in  \supp(\mathbb{Q})$ and $\bigl(\Phi \circ \psi \bigr)(q) > \bigl(\mathcal{U}\circ \Psi^{-1}\bigr)(q) -\delta$ for all $q \in \supp(\mathbb{Q})$.  Now let $\pi:=\psi\widehat{\mathbb{Q}} \in \mathcal{M}(\mathcal{A})$, and observe from the proof  of the first assertion that $\Psi\pi=\mathbb{Q}$, and therefore $\pi \in \Psi^{-1}\mathbb{Q}$.  Therefore, by a change of variables,
\begin{align*}
	\E_{\pi}[\Phi]
	&= \E_{\psi\widehat{\mathbb{Q}}}[\Phi]\\
	&= \E_{\widehat{\mathbb{Q}}}[\Phi\circ \psi]\\
 	&> \E_{\widehat{\mathbb{Q}}}[\mathcal{U}\circ \Psi^{-1}] -\delta.
\end{align*}
Since, by definition, $\E_{\mathbb{Q}}[\mathcal{U}\circ \Psi^{-1}]:=\E_{\widehat{\mathbb{Q}}}[\mathcal{U}\circ \Psi^{-1}]$,
it follows that
\begin{align*}
	\mathcal{U}(\Psi^{-1}\mathfrak{Q})
	&= \sup_{\pi \in \Psi^{-1}\mathfrak{Q}}\E_{\pi}[\Phi]\\
	&\geq \sup_{\mathbb{Q} \in \mathfrak{Q}}\E_{\mathbb{Q}}[\mathcal{U}\circ \Psi^{-1}] -\delta.
\end{align*}
Since $\delta > 0$ was arbitrary, it follows that
\[
	\mathcal{U}(\Psi^{-1}\mathfrak{Q}) \geq \sup_{\mathbb{Q} \in \mathfrak{Q}}\E_{\mathbb{Q}}[\mathcal{U}\circ \Psi^{-1}] .
\]
Recalling the reverse inequality \eqref{eq-e1}, we obtain the main assertion.

The assertion of measure affinity follows from Lemma \ref{lem_affine}.

For the assertion  \eqref{eq:hskjbiu3}, define
\[
	\Pi:=\left\{\pi\in \mathcal{M}(\mathcal{A}) \smid  \E_{\pi}[\psi_i]= 0\text{ for } i=1,\ldots, n \right\}.
\]
Let $\epsilon>0$.  Assume that $\sup_{\pi_+\in \Pi_+} \E_{\pi_+}\big[\Phi\big]>\lambda$ and that $\pi_+ \in \Pi_+$ is such that $\E_{\pi_+}\big[\Phi\big]>\lambda$.  Observe that $\pi:=\pi_+/\pi_+(\mathcal{A})$ is an element of $\Pi$ that satisfies $\E_{\pi}\big[\Phi-\lambda \psi_0\big]>0$.  Define $\Pi_n$ as in \eqref{eq:uihue32Dirac} and apply
\cite[Thm.~4.1]{OSSMO:2011} to $\sup_{\pi \in \Pi}\E_{\pi}\big[\Phi-\lambda \psi_0\big]$ to conclude that there exists $\pi^{\ast} \in \Pi_n$ such that $\E_{\pi^{\ast}}\big[\Phi-\lambda \psi_0\big]>0$.  Since $\Phi-\lambda \psi_0=(\varphi-\lambda)\psi_0$ and $\psi_0$ is positive, it also follows that $\E_{\pi^{\ast}}\big[\psi_0\big]>0$.  Writing $\pi_+^{\ast}:=\pi^{\ast}/\E_{\pi^{\ast}}[\psi_0]$ we obtain that $\pi_+^{\ast} \in \Pi_{+,n}$ and $\E_{\pi_+^{\ast}}\big[\Phi\big]> \lambda$, which concludes the proof of \eqref{eq:hskjbiu3}.  \hfill\qedsymbol

\subsection{Proof of Lemma \ref{lem_positivemass}}

Consider the set
\[
	\mathcal{Y}':=\bigcup \left\{ \mathcal{O}_{y} \smid \mathcal{O}_{y} \subseteq \mathcal{Y} \text{ is open and } \nu(\mathcal{O}_{y}) = 0 \right\}.
\]
First let us show that $E=\mathcal{Y}'$.  To see this, first observe that trivially we have $E \subseteq \mathcal{Y}'$.  Now suppose that $y \in \mathcal{Y}'$.  Then there exists a $y'\in  \mathcal{Y}$ and an open $\mathcal{O}_{y'} \ni y'$ such that $y \in \mathcal{O}_{y'}$ and $\nu(\mathcal{O}_{y'} )=0$.  Therefore, $y\in E$ and hence $E=\mathcal{Y}'$.

Now, since $\mathcal{Y}'$ is a union of open sets, it is open and therefore measurable.
Moreover, since $\mathcal{Y}$ is strongly Lindel\"{o}f, it follows that $\mathcal{Y}'$ is Lindel\"{o}f
and that the open cover of $\mathcal{Y}'$ by $\nu$-null open sets used in the definition of $\mathcal{Y}'$ has a countable subcover, so that
\[
	\mathcal{Y}' = \bigcup_{i \in \mathbb{N}} \mathcal{O}_{y_{i}}
\]
where each $\mathcal{O}_{y_{i}}$ is open and has $\nu(\mathcal{O}_{y_{i}})=0$.  It follows that
\[
	\nu(E)=\nu(\mathcal{Y}') \leq \sum_{i \in  \mathbb{N}} \nu(\mathcal{O}_{y_{i}}) = 0
\]
and the proof is finished.  \hfill\qedsymbol

\subsection{Proof of Theorem \ref{thm:redpip}}

The first assertion, \eqref{eq:hsbiu3}, follows by layering the set of positive measures of finite total mass as
$\bigcup_{r \in \R_{+}} \{r\mathcal{M}(\mathcal{A})\}$, using the fact that the supremum over a union is a supremum of suprema, and applying the reduction theorem \cite[Thm.~4.1]{OSSMO:2011} in $r\mathcal{M}(\mathcal{A})$ separately.

For the second assertion, \eqref{eq:hskjbiu3}, define
\[
	\Pi:= \left\{\pi\in \mathcal{M}(\mathcal{A}) \smid  \E_{\pi}[\psi_i]= 0\text{ for } i=1,\ldots, n\right\}
\]
Let $\epsilon>0$.  Assume that $\sup_{\pi_+\in \Pi_+} \E_{\pi_+}\big[\Phi\big]>\lambda$ and that $\pi_+ \in \Pi_+$ is such that $\E_{\pi_+}\big[\Phi\big]>\lambda$.  Observe that $\pi:=\pi_+/\pi_+(\mathcal{A})$ is an element of $\Pi$ that satisfies $\E_{\pi}\big[\Phi-\lambda \psi_0\big]>0$.  Defining $\Pi_n$ as in \eqref{eq:uihue32Dirac} and applying \cite[Thm.~4.1]{OSSMO:2011} to $\sup_{\pi \in \Pi} \E_{\pi}\big[\Phi-\lambda \psi_0\big]$, we deduce that there exists $\pi^{\ast} \in \Pi_n$ such that $\E_{\pi^{\ast}}\big[\Phi-\lambda \psi_0\big]>0$.  Since $\Phi-\lambda \psi_0=(\varphi-\lambda)\psi_0$ and $\psi_0$ is positive, it also follows that $\E_{\pi^{\ast}}\big[\psi_0\big]>0$. Let $\pi_+^{\ast}:=\pi^{\ast}/\E_{\pi^{\ast}}[\psi_0]$ to obtain that $\pi_+^{\ast} \in \Pi_{+,n}$ and $\E_{\pi_+^{\ast}}\big[\Phi\big]> \lambda$, which concludes the proof of \eqref{eq:hskjbiu3}.  \hfill\qedsymbol

\subsection{Proof of Theorem \ref{thm:alternredded}}

 First, we prove that
\begin{equation}
    \label{eq:intdvhge}
	\mathcal{U}(\Pi(q)|B)= \sup_{\pi_+\in \Pi_+(q)} \E_{\mu \sim \pi_+}\big[  \Phi(\mu) \Dmap(\mu)[B] \big],
\end{equation}
where $\Pi_+(q)$ is the set of positive finite measures $\pi_+$ on $\mathcal{A}$ such that $\E_{\pi_+}\big[\Psi(\mu)-q\big]=0$ and $\E_{\pi_+ }\big[ \Dmap(\mu)[B] \big]=1$.
To that end, first observe that
\[
	\mathcal{U}(\Pi(q)|B)= \sup_{\pi \in \Pi(q): \pi \odot \Dmap[B]>0}\E_{\pi\odot \Dmap}\big[ \Phi\big|  B\big]
\]
and that, for any $\pi \in \Pi(q)$ such that $ \pi \odot \Dmap[B]>0$,
\begin{equation}
	\label{eq:intedegy}
    \E_{\pi \odot \Dmap}\big[ \Phi\big| B\big]=\frac{\E_{\mu\sim \pi} \big[  \Phi(\mu)
 \Dmap(\mu)[B] \big]}{\E_{\mu\sim\pi} \big[\Dmap(\mu)[B]\big]}.
\end{equation}
Now consider  $\pi\in \Pi(q)$ such that $\pi \odot \Dmap[B]>0$.  Then $\pi_+:=\pi/\E_{\pi}\big[ \Dmap(\mu)[B] \big]$ is an element of $\Pi_+(q)$ and \eqref{eq:intedegy} implies that
\[
	\E_{\pi\odot  \Dmap}\big[ \Phi\big|B\big]=\E_{\mu \sim \pi_+}\big[  \Phi(\mu) \Dmap(\mu)[B] \big] .
\]
Conversely, if $\pi_+\in \Pi_+(q)$, then $\pi:=\pi_+/\pi_+[\mathcal{A}]$ is an element of $\Pi(q)$  such that  $\pi \odot \Dmap[B]>0$ and
\[
	\E_{\mu \sim \pi_+}\big[  \Phi(\mu) \Dmap(\mu)[B] \big]= \frac{\E_{\mu \sim \pi}\big[  \Phi(\mu) \Dmap(\mu)[B] \big]}{\E_{\mu \sim \pi}\big[ \Dmap(\mu)[B] \big]}.
\]
Since the above argument also shows that $\Pi(q)\odot_B\Dmap$ is nonempty if and only if
$\Pi_+(q)$ is nonempty, \eqref{eq:intdvhge} follows.  The right hand side of \eqref{eq:intdvhge} is a linear program in $\pi_+$, so Theorem~\ref{thm:redpip} implies that the supremum in $\pi_+$ can be achieved by assuming $\pi_+$ to be the weighted sum of at most $n+1$ Dirac masses, i.e.~by assuming that
\begin{equation}
	\label{eq:polycondeesde}
	\pi_+=\sum_{i = 0}^{n} \alpha_{i} \delta_{\mu_{i}}.
\end{equation}
This finishes the proof of Theorem~\ref{thm:alternredded}.  \hfill\qedsymbol

\subsection{Proof of Theorem \ref{thm:sivahidden}}

First let us show that, for $\lambda \in \R$, the statement that
\begin{equation}
	\label{eq:e1}
	\E_{\pi\odot \Dmap}[\Phi| B] >\lambda, \quad  \pi\in  (\Psi^{-1}(\mathfrak{Q}))_B
\end{equation}
is equivalent to the statement that
\begin{equation}
	\label{eq:e2}
	\E_{\mu \sim \pi}\Big[\big(\Phi(\mu)-\lambda\big)\Dmap(\mu)[B]\Big]>0  .
\end{equation}
To that end, assume \eqref{eq:e1} and observe that the definition \eqref{def_B} of $(\Psi^{-1}(\mathfrak{Q}))_B$ implies that $\pi \cdot \Dmap[B]> 0$, where, by \eqref{eq:cdotexp2},
\begin{equation}
	\label{eq:cdotexp3}
    \pi\cdot\Dmap[B]:=\E_{\mu\sim \pi}\big[\Dmap(\mu)[B]\big] .
\end{equation}
Consequently, by \eqref{eq:defcondexp},
\[
	\E_{\pi \odot \Dmap}\big[ \Phi\big| B\big]=\frac{\E_{\mu\sim \pi}\big[  \Phi(\mu) \Dmap(\mu)[B] \big]}{\E_{\mu\sim\pi}\big[\Dmap(\mu)[B]\big]} >\lambda ,
\]
and the denominator is strictly positive.  Therefore,
\begin{align*}
	& \E_{\mu\sim \pi}\big[ \bigl(\Phi(\mu)-\lambda) \Dmap(\mu)[B] \big] \\
	& \quad = \E_{\mu\sim \pi}\big[  \Phi(\mu) \Dmap(\mu)[B] \big] -\lambda \E_{\mu\sim\pi}\big[\Dmap(\mu)[B]\big] \\
	& \quad >0 ,
\end{align*}
and \eqref{eq:e2} follows.  Conversely, assume \eqref{eq:e2} and observe that $\pi \cdot \Dmap[B]> 0$.  To see this, observe that, if $\pi \cdot \Dmap[B]= 0$, then \eqref{eq:cdotexp3} implies that
$\Dmap(\mu)[B]=0 $ $\pi$-a.s.\ and so
\[
	\E_{\mu\sim \pi}\big[  \bigl(\Phi(\mu)-\lambda) \Dmap(\mu)[B] \big] =0,
\]
which is a contradiction. Consequently, $\pi \cdot \Dmap[B] >0$ and dividing the assumption
\begin{align*}
	& \E_{\mu\sim \pi}\big[ \bigl(\Phi(\mu)-\lambda) \Dmap(\mu)[B] \big] \\
	& \quad = \E_{\mu\sim \pi}\big[  \Phi(\mu) \Dmap(\mu)[B] \big] -\lambda \E_{\mu\sim\pi}\big[\Dmap(\mu)[B]\big]  \\
	& \quad > 0
\end{align*}
by $\pi\cdot\Dmap[B]:=\E_{\mu\sim \pi}\big[\Dmap(\mu)[B]\big]$ throughout yields
\eqref{eq:e1} and the equivalence is  established.

The  main assertion now follows from a direct application of Theorem~\ref{thm:reducpriormarg}.
Finally, since $\Phi$ is semibounded, it follows that $\mu \mapsto \Phi(\mu)\Dmap(\mu)[B]$ is semibounded
and measurable, and the measure-affinity assertion follows from Lemma~\ref{lem_affine}.  \hfill\qedsymbol

\subsection{Proof of Theorem \ref{thm:shiva}}

Let us first establish that the assumptions of the theorem are well defined. To that end, note that Lemma~\ref{lem:Umeasurable} implies that $q \mapsto \inf_{\mu\in \Psi^{-1}(q)} \Dmap(\mu)[B]$ is $\widehat{\mathcal{B}}(\supp(\mathbb{Q}))$-measurable and hence \eqref{eq:dto0} is well defined. Similarly \eqref{eq:djkdjehjehj33} is well defined.

For the proof of the theorem, fix $\delta>0$,  let $\mathbb{Q}\in \mathfrak{Q}$ and $\Dmap\in \Dmaps$ satisfy the assumptions, and define $\lambda:=\mathcal{U}(\mathcal{A})-\delta$. Since $(\Phi(\mu)-\lambda)\Dmap(\mu)[B]$ is bounded and measurable, Lemma~\ref{lem:Umeasurable} implies that the function $q \mapsto \theta(q):= \sup_{\mu\in \Psi^{-1}(q)} (\Phi(\mu)-\lambda)\Dmap(\mu)[B]$ is $\widehat{\mathcal{B}}(\supp(\mathbb{Q}))$-measurable.  Moreover, \eqref{eq:dto0} implies that the function $\theta$ is non-negative with $\mathbb{Q}$-probability
one and \eqref{eq:djkdjehjehj33} implies that $\theta$ is strictly positive on a subset of strictly positive $\mathbb{Q}$-measure.  Hence,
\[
	E_{q\sim \mathbb{Q}} \left[\sup_{\mu\in \Psi^{-1}(q)} (\Phi(\mu)-\lambda)\Dmap(\mu)[B] \right] = \E_{\mathbb{Q}}[\theta] > 0,
\]
and, therefore,
\[
	\sup_{Q \in \mathfrak{Q}} \E_{q\sim \mathbb{Q}} \left[\sup_{\mu\in \Psi^{-1}(q)} (\Phi(\mu)-\lambda)\Dmap(\mu)[B] \right] > 0 .
\]
It then follows from Theorem~\ref{thm:sivahidden} that $\mathcal{U}(\Psi^{-1}\mathfrak{Q}|B) \geq \lambda=\mathcal{U}(\mathcal{A})-\delta$.  Since $\delta >0$ was arbitrary, it follows that $\mathcal{U}(\Psi^{-1}\mathfrak{Q}|B) \geq \mathcal{U}(\mathcal{A})$.  Theorem~\ref{thm-barriers} implies that
\[
	\mathcal{U}(\Psi^{-1}\mathfrak{Q} |B) \leq  \mathcal{U}(\mathcal{A})
\]
and the theorem follows.  \hfill\qedsymbol

\subsection{Proof of Theorem \ref{thm:shiva0cor}}

We will need the following notation: for an admissible set $\Pi \subseteq \mathcal{M}(\mathcal{A})$ of priors,
an observation map $\Dmap$, and an open subset $B \subseteq \Dspace$, let $\Pi \odot_B \Dmap$ be the set of probability distributions $\pi\odot \Dmap$ on $\mathcal{A}\times \Dspace$ generated by $\pi \in \Pi$:
\[
	\Pi\odot_B \Dmap := \{ \pi \odot \Dmap \mid \pi \in \Pi \text{ and }(\pi \cdot \Dmap)[B]> 0\}.
\]
We also define
\[
	\mathcal{U}(\Pi\odot_B\Dmap):=\sup_{\pi\odot \Dmap \in \Pi\odot_B\Dmap}\E_{\pi \odot \Dmap}\big[ \Phi\big| B\big].
\]

The proof follows from the proof of Theorem \ref{thm:shiva} as follows.  Let $\delta > 0$, and let  a measurable section $\psi$ satisfy the assumptions.  Define $\lambda:=\mathfrak{Q}^{\infty}(\Phi\circ\psi)-\delta$, and the universally measurable function $q \mapsto \theta(q):= \sup_{\mu\in \Psi^{-1}(q)} (\Phi(\mu)-\lambda)\Dmap(\mu)[B]$. Then assumption \eqref{eq:B0} implies that the function $\theta$ is non-negative.  It follows from the definition \eqref{eq_esssup} of $\mathfrak{Q}^{\infty}(\Phi\circ\psi)$, and $\lambda <\mathfrak{Q}^{\infty}(\Phi\circ\psi)$, that there is a $\mathbb{Q} \in \mathfrak{Q}$ such that $\Phi\circ\psi >\lambda $ with nonzero $\mathbb{Q}$-measure. Since
\begin{align*}
	\theta(q)
	&= \sup_{\mu\in \Psi^{-1}(q)} (\Phi(\mu)-\lambda)\Dmap(\mu)[B]\\
	&\geq  (\Phi \circ \psi(q)-\lambda)\Dmap(\psi(q))[B] ,
\end{align*}
the positivity assumption \eqref{eq:Bp} implies that $\theta$ is positive on a subset of positive $\mathbb{Q}$-measure. Hence,
\[
	\E_{q\sim \mathbb{Q}} \left[ \sup_{\mu\in \Psi^{-1}(q)} (\Phi(\mu)-\lambda)\Dmap(\mu)[B]\right] = \E_{\mathbb{Q}}[\theta] >0
\]
and, therefore,
\[
	\sup_{\mathbb{Q} \in \mathfrak{Q}} \E_{q\sim \mathbb{Q}} \left[ \sup_{\mu\in \Psi^{-1}(q)} (\Phi(\mu)-\lambda)\Dmap(\mu)[B] \right] > 0 .
\]
It then follows from Theorem \ref{thm:sivahidden} that
\[
	\mathcal{U}(\Psi^{-1}\mathfrak{Q}|B) \geq \lambda=\mathfrak{Q}^{\infty}(\Phi\circ\psi)-\delta.
\]
Since $\delta >0$ was arbitrary, the assertion is proved.  \hfill\qedsymbol

\subsection{Proof of Theorem \ref{thm_localshiva}}

We appeal to the corollary, Theorem \ref{thm:shiva0cor}, to Theorem \ref{thm:shiva}. To that end, let $\mathcal{A}$ be defined as in \eqref{eq:Aup}, and let $\mathcal{Q}:=\mathcal{A}_{0}$, $\Psi:=P_{0}$, $\Dmaps:=\{\Dmap^{n}\}$.

Since $\Dmap^{n}=\Dmap^{n}_{0}\circ P_{\alpha}$ is a pull-back,
\begin{align*}
	(\pi \cdot \Dmap^{n})[B^{n}_{\delta}]
	&= \E_{(\mu_{1},\mu_{2})\sim\pi} \big[\Dmap^{n}(\mu_{1},\mu_{2})[B^{n}_{\delta}]\big]\\
	&= \E_{(\mu_{1},\mu_{2})\sim\pi} \big[ \Dmap^{n}_{0}\circ P_{\alpha}(\mu_{1},\mu_{2})[B^{n}_{\delta}] \big]\\
	&= \E_{\mu_{2}\sim P_{\alpha}\pi}\big[\Dmap^{n}_{0}(\mu_{2})[B^{n}_{\delta}]\big]\\
	&= (P_{\alpha}\pi \cdot \Dmap^{n}_{0})[B^{n}_{\delta}] ,
\end{align*}
from which we conclude that $(P_{\alpha}\pi \cdot \Dmap^{n}_{0})[B^{n}_{\delta}]>0$ if and only if $(\pi \cdot \Dmap^{n})[B^{n}_{\delta}]>0$, and so conclude
\[
	\Pi_{\alpha}\odot_{B^{n}_{\delta}} \Dmap^{n}_{0}=P_{\alpha}\bigl(\Pi \odot_{B^{n}_{\delta}} \Dmap^{n}\bigr) ,
\]
where $P_{\alpha}$ acts on each component in the natural way.  Moreover since $\Phi=\Phi_{0}\circ P_{\alpha}$ is also a pull-back, for $\pi \in \Pi$, we have
\begin{align*}
	\E_{\pi \odot \Dmap^{n}}\big[ \Phi\big| B^{n}_{\delta}\big]
	&= \frac{\E_{(\mu_{1},\mu_{2})\sim \pi}\big[  \Phi(\mu_{1},\mu_{2}) \Dmap^{n}(\mu_{1},\mu_{2})[B^{n}_{\delta}] \big]}{\E_{(\mu_{1},\mu_{2})\sim\pi} \big[\Dmap^{n}(\mu_{1},\mu_{2})[B^{n}_{\delta}]\big]} \\
	&= \frac{\E_{(\mu_{1},\mu_{2})\sim \pi}\big[ \Phi_{0}\circ P_{\alpha}(\mu_{1},\mu_{2}) \cdot \Dmap^{n}_{0}\circ P_{\alpha}(\mu_{1},\mu_{2})[B^{n}_{\delta}] \big]}{\E_{(\mu_{1},\mu_{2})\sim\pi}\big[\Dmap^{n}_{0} \circ P_{\alpha}(\mu_{1},\mu_{2})[B^{n}_{\delta}]\big]} \\
	&= \frac{\E_{\mu_{2}\sim P_{\alpha}\pi}\big[  \Phi_{0}(\mu_{2}) \Dmap^{n}_{0}(\mu_{2})[B^{n}_{\delta}] \big]}{\E_{\mu_{2}\sim P_{\alpha}\pi}\big[\Dmap^{n}_{0}(\mu_{2})[B^{n}_{\delta}]\big]}\\
	&= \E_{P_{\alpha}\pi \odot \Dmap^{n}_{0}}\big[ \Phi_{0}\big| B^{n}_{\delta}\big]
\end{align*}
and so we conclude that
\begin{equation}
	\label{eq_uuuu}
	\mathcal{U}(\Pi\odot_{B^{n}_{\delta}}\Dmap^{n})=\mathcal{U}(\Pi_{\alpha}\odot_{B^{n}_{\delta}}\Dmap^{n}_{0}) .
\end{equation}

We will now need the following proposition
\begin{prop}
	\label{prop_TV}
	Define the total variation metric $d_{\mathrm{TV}}$ on $\mathcal{M}(\mathcal{X})$ by
	\[
		d_{\mathrm{TV}}(\mu_{1},\mu_{2}):=\sup_{A \in \mathcal{B}(\mathcal{X})} |\mu_{1}(A)-\mu_{2}(A)| .
	\]
	Consider $B \in \mathcal{B}(\mathcal{X})$. Then for $\mu \in \mathcal{M}(\mathcal{X})$ such that $\mu(B) < 1$, we have
	\[
		d_{\mathrm{TV}}(\mu, \mu|_{B^{c}}) \leq \mu(B) .
	\]
\end{prop}

\begin{proof}
	For $A \in \mathcal{B}(\mathcal{X})$, we have
	\begin{align*}
		\mu(A)-\mu|_{B^{c}}(A)
		&= \mu(A)-\frac{\mu(A \cap B^{c})}{\mu(B^{c})}\\
		&= \mu(A\cap B)+ \mu(A\cap B^{c})-\frac{\mu(A \cap B^{c})}{\mu(B^{c})}\\
		&= \mu(A\cap B)-\frac{\mu(B)}{1-\mu(B)} \mu(A\cap B^{c})
	\end{align*}
	and therefore
	\[
		\mu(A)-\mu|_{B^{c}}(A) \leq \mu(A\cap B) \leq \mu(B)
	\]
	and
	\begin{align*}
		\mu(A)-\mu|_{B^{c}}(A)
		& \geq -\frac{\mu(B)}{1-\mu(B)} \mu(A\cap B^{c})\\
		& \geq -\frac{\mu(B)}{1-\mu(B)} \mu(B^{c})\\
		& = -\mu(B)\, ,
	\end{align*}
	thus establishing the assertion.
\end{proof}

For $B \in \mathcal{B}(\mathcal{X})$ and $\mu \in \mathcal{M}(\mathcal{X})$ such that $\mu(B) < 1$, the conditional measure $\mu|_{B^{c}} \in \mathcal{M}(\mathcal{X})$ is defined by
\[
	\mu|_{B^{c}}(A):=\frac{\mu(A \cap B^{c})}{\mu(B^{c})}, \quad A \in \mathcal{B}(\mathcal{X}) .
\]

It follows from Proposition \ref{prop_TV} that $d_{\mathrm{TV}}(\mu, \mu|_{B^{c}}) \leq \mu(B)$
and since $d_{\mathcal{M}} \leq d_{\mathrm{TV}}$ (see e.g.~\cite[Eq.~2.24]{HuberRonchetti:2009}), we conclude that
\begin{equation}
	\label{prokhorov}
	d_{\mathcal{M}}(\mu, \mu|_{B^{c}}) \leq \mu(B) .
\end{equation}

Let $B_{\delta}:=B_{\delta}(x_{1})$ denote the ball about the first sample of $x^{n}=(x_{1},\dots,x_{n})$. Then, for $\mu_{0} \in \mathcal{A}_{0}$, it follows from the assumptions that
\begin{align*}
	d_{\mathcal{M}}(\mu_{0}, \mu_{0}|_{B_{\delta}^{c}})
	&\leq \mu_{0}(B_{\delta})\\
	&\leq \mathcal{P}^{\infty}(\delta)\\
	&< \alpha
\end{align*}
and therefore
\[
	\bigl(\mu_{0}, \mu_{0}|_{B_{\delta}^{c}}\bigr) \in \Psi^{-1}\mu_{0} .
\]
Moreover, since
\begin{align*}
	\Dmap^{n}_{(\mu_{0}, \mu_{0}|_{B_{\delta}^{c}})}[B^{n}_{\delta}]
	&= \bigl(\mu_{0}|_{B_{\delta}^{c}}\bigr)^{n}[B^{n}_{\delta}]\\
	&\leq \mu_{0}|_{B_{\delta}^{c}}[B_{\delta}]\\
	&= 0 ,
\end{align*}
we conclude that the condition \eqref{eq:B0}
\[
	\inf_{(\mu_{0},\mu_{0}')\in  \Psi^{-1}\mu_{0}} \Dmap^{n}(\mu_{0},\mu_{0}')[B^{n}_{\delta}] =  0
\]
of Theorem \ref{thm:shiva0cor} is satisfied for all $\mu_{0}\in\mathcal{A}_{0}$.

Now consider the diagonal map $\Delta \colon \mathcal{M}(\mathcal{X}) \to \mathcal{M}(\mathcal{X}) \times \mathcal{M}(\mathcal{X})$ defined by
\[
	\Delta(\mu):=(\mu,\mu),\quad \mu\in \mathcal{M}(\mathcal{X})  .
\]
Since
\[
	\Psi \circ \Delta(\mu)=P_{0}\circ \Delta(\mu)=\mu, \quad \text{for all } \mu\in \mathcal{M}(\mathcal{X}),
\]
it follows, if we define $\Delta$ on the first component of the product $\mathcal{M}(\mathcal{X}) \times \mathcal{M}(\mathcal{X})$ and then restrict to $\mathcal{A}_{0}$, that $\Delta$ is a section of $\Psi=P_{0}$. It is clearly measurable, but also satisfies
\[
	P_{\alpha} \circ \Delta(\mu)=\mu,\quad \text{for all }\mu\in \mathcal{M}(\mathcal{X}),
\]
that is, $P_{\alpha} \circ \Delta $ is the identity map from the first component of $\mathcal{M}(\mathcal{X})\times \mathcal{M}(\mathcal{X})$ to the second.  Then, for $\mu_{0} \in \mathcal{A}_{0}$, the positivity of the model $\mathcal{P}$ implies that
\begin{align*}
	\Dmap^{n}\bigl(\Delta(\mu_{0})\bigr)[B_{\delta}^{n}]
	&= \Dmap^{n}_{0}\circ P_{\alpha}\bigl(\Delta(\mu_{0})\bigr)[B_{\delta}^{n}]\\
	&= \Dmap^{n}_{0}(\mu_{0})[B_{\delta}^{n}]\\
	&= (\mu_{0})^{n}[B_{\delta}^{n}]\\
	&= \prod_{i=1}^{n}{\mu_{0}[B_{\delta}(x_{i})]}\\
	&> 0
\end{align*}
so that the second condition \eqref{eq:Bp} of Theorem \ref{thm:shiva0cor} is satisfied for all $\mu_{0}\in \mathcal{A}_{0}$.  Theorem \ref{thm:shiva0cor} then asserts that
\[
	\mathcal{U}(\Psi^{-1}\Pi_{0}\odot_{B^{n}_{\delta}}\Dmap^{n}) \geq \Pi_{0}^{\infty}(\Phi\circ \Delta).
\]
Moreover, since
\[
	\Phi\circ \Delta=\Phi_{0}\circ P_{\alpha} \circ \Delta=\Phi_{0},
\]
now as a function on the first component of $\mathcal{M}(\mathcal{X}) \times \mathcal{M}(\mathcal{X})$, and
\[
	\Psi^{-1}\Pi_{0}=P_{0}^{-1}\Pi_{0}=\Pi,
\]
we conclude that
\[
	\mathcal{U}(\Pi\odot_{B^{n}_{\delta}}\Dmap^{n}) \geq \Pi_{0}^{\infty}(\Phi_{0}).
\]
The identity $\mathcal{U}(\Pi \odot_{B^{n}_{\delta}} \Dmap^{n}) = \mathcal{U}(\Pi_{\alpha}\odot_{B^{n}_{\delta}}\Dmap^{n}_{0})$ of \eqref{eq_uuuu} then implies the assertion.  \hfill\qedsymbol

\section{Appendix}
\label{sec:appendix}

The following lemma is Lemma III.39  p. 86 of \cite{CastaingValadier:1977}. We also refer to p. 87 of \cite{CastaingValadier:1977} for the existence of the measurable selection $\eta$ (which is also derived from Theorem III.38 p.85 of \cite{CastaingValadier:1977}).
These results are related to Aumann's measurable section principle \cite{Aumann:1967}  (the extension to Suslin space is due to Sainte-Beuve \cite{SainteBeuve1974}).

\begin{lem}
	\label{lemIII39CstaingValafier}
	Let $(T,\mathcal{T})$ be a measurable space, $S$ a Suslin space. $\varphi \colon T \times S \to \bar{R}$ a $\mathcal{T}\otimes \mathcal{B}(S)$ measurable function and $\Gamma$ a multifunction (i.e.~a set-valued map) from $T$ to non-empty subsets of $S$ whose graph $G$ belongs to $\mathcal{T}\times \mathcal{B}(S)$. Then
	\begin{enumerate}
		\item the function
		\[
			m(t):=\sup\{\phi(t,x)\mid x\in \Gamma(t)\}
		\]
		is a $\widehat{\mathcal{T}}$-measurable function of $t$.
		\item for $\delta>0$, there exists $\eta$, a $\widehat{\mathcal{T}}$-measurable
		function of $t$, such that $\eta(t)\in \Gamma(t)$ and $\varphi(t,\eta(t))>m(t)-\delta$.
	\end{enumerate}
\end{lem}

The following definition is Definition 4.50 in \cite{AliprantisBorder:2006}:

\begin{defn}
	\label{def:caratheodory}
	Let $(S,\Sigma)$ be a measurable space, and let $X$ and $Y$ be topological spaces. A function $h\colon S\times X \to Y$ is a \emph{Carath\'{e}odory function} if:
	\begin{enumerate}
		\item for each $x\in X$, the function $h^x=h(.,x) \colon S\to Y$
		is $\big(\Sigma, \mathcal{B}(Y)\big)$-measurable; and
		\item for each $s\in S$, the function $h_s=h(s,.) \colon X\to Y$
		is continuous.
	\end{enumerate}
\end{defn}

The following lemma is Lemma 4.51 in \cite{AliprantisBorder:2006}  (see also \cite[p.~70]{CastaingValadier:1977}):

\begin{lem}
	\label{def:caratheodorylemma}
	Let $(S,\Sigma)$ be a measurable space, $X$ a separable metrizable space, and $Y$ a metrizable space. Then every Carath\'{e}odory function $h\colon S\times X \to Y$ is jointly measurable.
\end{lem}

\subsection{Universally Measurable Functions}
\label{sec-universal}

For a topological space $T$ let $\widehat{\mathcal{B}}(T)$ denote the $\sigma$-algebra of universally measurable sets. For a measure $\mu$, let $\widehat{\mu}$ denote its completion.  Here we state the following proposition that allows us to define the expected value of $\widehat{\mathcal{B}}(T)$ measurable functions with respect to Borel measures. In all statements in the following proposition,  the assertions follow when the integrals involved exist, in particular for
semibounded functions.  The proof is straightforward but tedious and follows from e.g.
\cite[Thm.~pg.~37]{Doob:1994}, the English version of
\cite[Ch.~2, pg.~49]{DellacherieMeyer:1975}, and \cite{CastaingValadier:1977}.

\begin{prop}
	\label{prop_E}
	Let $T$ be a  topological space. Then we have
	\begin{itemize}
		\item For a measurable function $f$ we have $\E_{\widehat{\mu}}f=\E_{\mu}f$
		\item Let $f$ be $\widehat{\mathcal{B}}(T)$-measurable. Then there exist two measurable functions $\underline{f}$ and $\overline{f}$ such that
		\[
			\underline{f} \leq f \leq \overline{f},\quad  \mu(\underline{f}\neq \overline{f})=0
		\]
		and, for any such functions, we have
		\[
			\E_{\mu}[\underline{f}] = \E_{\widehat{\mu}}[f] = \E_{\mu}[\overline{f}]
		\]
		\item For a fixed $\mu$, $f \mapsto \E_{\widehat{\mu}}[f]$   defines an affine function on the cone of non-negative $\widehat{\mathcal{B}}(T)$-measurable functions
		\item For a fixed $\widehat{\mathcal{B}}(T)$-measurable function $f$, the function $\mathcal{M}(T) \ni \mu \mapsto \E_{\widehat{\mu}}[f]$ is affine.
		\item  Suppose that $f_{1},f_{2}$ are $\widehat{\mathcal{B}}(T)$-measurable non-negative functions such that $f_{1}\leq f_{2}$. Then $\E_{\widehat{\mu}}[f_{1}] \leq E_{\widehat{\mu}}[f_{2}] $ for all $\mu \in \mathcal{M}(T)$.
\end{itemize}
\end{prop}

Proposition \ref{prop_E} leads to the following definition for the expectation  of
$\widehat{\mathcal{B}}(T)$-measurable functions with respect to Borel probability measures on $T$:

\begin{defn}
	\label{def_E}
	For a Borel probability measure $\mu \in \mathcal{M}(T)$, we define the integral of a $\widehat{\mathcal{B}}(T)$-measurable function $f$ by
	\[
		\E_{\mu}[f]:= \E_{\widehat{\mu}}[f]
	\]
	when the latter exists, where $\widehat{\mu}$ is the completion of the measure $\mu$ as described in \cite[p.~37]{Doob:1994}.
\end{defn}

Recall that a \emph{carrier} $T$ for a probability measure $\mathbb{Q} \in \mathcal{M}(\mathcal{Q})$ is a set $T \in \mathcal{B}(\mathcal{Q})$ such that $\mathbb{Q}(T)=1$. For a carrier $T$, since $T \in \mathcal{B}(\mathcal{Q})$, it follows that
$\mathcal{B}(T)=\mathcal{B}(\mathcal{Q})\cap T$ and we can define the trace measure
$\mathbb{Q}_{T} \in \mathcal{M}(T)$ by $\mathbb{Q}_{T}(A):=\mathbb{Q}(A), A \in \mathcal{B}(\mathcal{Q})\cap T$.  The following proposition shows that the expectation of a function can be defined with respect to measures which possess carriers upon which the function is universally measurable:

\begin{prop}
	\label{prop_Ecarrier}
Let $S$ be a topological space.
	Suppose that $f$ is $\widehat{\mathcal{B}}(T)$-measurable for all measurable $T \subseteq S$.
  Suppose also that $\mathbb{Q} \in \mathcal{M}(S)$ has a carrier $T \subseteq S$. Then, using  Definition~\ref{def_E}, any such carrier $T$ defines an expectation
	\[
		\E_{\mathbb{Q}_{T}}[f]:=\E_{\widehat{\mathbb{Q}}_{T}}[f],
	\]
	and this definition is independent of the carrier; that is, if $T' \subset S$ is another carrier, then
	\[
		\E_{\widehat{\mathbb{Q}}_{T'}}[f]= \E_{\widehat{\mathbb{Q}}_{T}}[f].
	\]
	Moreover, this expectation satisfies the assertions of affinity and monotonicity of Proposition~\ref{prop_E}.
\end{prop}

We  also need a change of variables formula for expectations of universally measurable functions.

\begin{prop}
	\label{prop_change}
	Let $X$ and $Y$  be topological spaces, $\Psi \colon X\to Y$ a measurable map and suppose that $f \colon Y\to \R$ is $\widehat{\mathcal{B}}(Y)$ measurable.  Then $f \circ \Psi \colon X \to \R$ is $\widehat{\mathcal{B}}(X)$-measurable and, for $\pi\in \mathcal{M}(X)$,
	\[
		\E_{\Psi\pi}[f]=\E_{\pi}[f\circ \Psi].
	\]
\end{prop}

For Suslin space $\mathcal{X}$ and a subset $M \subset \mathcal{M}(\mathcal{X})$ let $\Sigma(M)$ denote the smallest $\sigma$-subalgebra  of subsets of $M$ for which the the evaluation map $\nu \mapsto \nu(B)$ is measurable for all $B \in \mathcal{B}(\mathcal{X})$.  The following version of a result of von Weizsacker \& Winkler
\cite{WeizsackerWinkler:1979} as stated in \cite[Thm.~3.1]{Winkler:1988} will be useful to us:

\begin{thm}
	\label{thm_winkler}
	Consider a Suslin space $\mathcal{X}$, measurable functions $f_{1}, \dots, f_{n} \colon \mathcal{X}\to R$, constants $c_{1}, \dots, c_{n} \in \R$,
	and define
	\[
		H:=\bigl\{\nu \in  \mathcal{M}(\mathcal{X}) \,\big|\, f_{i}\, \text{is $\nu$-integrable and } \E_{\nu}[f_{i}] \leq c_{i}, \text{ for } i=1,\dots,n \bigr\}
	\]
	Then, for each $\nu \in H$, there is a probability measure $p$ on $\Sigma(\ext(H))$ such that
	\begin{equation}
		\label{eq_barycenter}
		\nu(B)=\int_{\ext(H)} \nu'(B) \, \mathrm{d}p(\nu'),\quad  \text{for all } B \in \mathcal{B}(\mathcal{X}) .
	\end{equation}
\end{thm}

\cite[Prop.~3.1]{Winkler:1988} shows that if a measurable function $f \colon \mathcal{X}\to \R$ is integrable with respect to all measures in $H$ (allowing the values $\infty$ and $-\infty$), then integration
\[
	F(\nu):=\int_{\mathcal{X}} f \, \mathrm{d} \nu
\]
is measure affine per Definition \ref{def_measureaffine}.  We need a slightly more general result:

\begin{lem}
	\label{lem_affine}
	Consider the situation of Theorem \ref{thm_winkler}, let $f \colon \mathcal{X}\to \R$ be a semibounded universally measurable function.  Then
	\[
		F(\nu):=\E_{\widehat{\nu}}[f], \quad \text{for }\nu \in H,
	\]
	is measure affine per Definition \ref{def_measureaffine}.
\end{lem}

The next lemma extends \cite[Thm.~2.1]{Winkler:1988} to the case where the constraint functions
$f_{i}$, for $i=1,\dots,n$, are universally measurable:

\begin{lem}
	\label{lem_ext}
	Let $\mathcal{X}$  be Suslin, and fix universally measurable real-valued functions $f_{1}, \dots, f_{n}$ and constants $c_{1}, \dots, c_{n}$. Then
	\begin{equation}
		\label{def_contraints-univ}
		H:=\bigl\{\nu \in \mathcal{M}(\mathcal{X}) \,\big|\, f_{i}\, \text{is $\widehat{\nu}$-integrable and } \E_{\widehat{\nu}}[f_{i}] \leq c_{i} \text{ for } i=1,\dots,n \bigr\}
	\end{equation}
	 is convex and
	\[
		\ext(H)=\Bigl\{\nu \in H \,\Big|\, \nu =\sum_{i=1}^{m}\alpha_{i}\d_{x_{i}}, \alpha_{i}\geq 0, x_{i}\in \mathcal{X},
i=1,\dots,m,   \sum_{i=1}^{m}\alpha_{i}=1, 1\leq m \leq n+1,
	\]
	\[
		\text{the vectors} \bigl(f_{1}(x_{i}),f_{2}(x_{i}),\dots,f_{n}(x_{i}),1\bigr), 1\leq i \leq m \,\, \text{are linearly independent}\, \Bigr\}\, .
	\]
\end{lem}

\subsection{Proofs}

\subsubsection{Proof of Proposition \ref{prop_Ecarrier}}
Let $T$ and $T'$ be two carriers for $\mathbb{Q}\in \mathcal{M}(S)$  and $f$ a function such that
$f_{T}$ and $f_{T'}$ are $\widehat{\mathcal{B}}(t)$- and $\widehat{\mathcal{B}}(T')$-measurable respectively.
 Then Proposition \ref{prop_E} implies that there are functions
$f_{1},f_{2}$ measurable on $T$ and $f'_{1},f'_{2}$ measurable on $T'$
such that
\begin{align*}
f_{1} \leq  f_{T}  \leq f_{2} && \mathbb{Q}_{T}(f_{1} \neq  f_{2})=0\\
f'_{1}  \leq  f_{T'}  \leq f'_{2} && \mathbb{Q}_{T'}(f'_{1} \neq  f'_{2})=0
\end{align*}
so that
\begin{align*}
\E_{\widehat{\mathbb{Q}}_{T}}[f_{T}]&=\E_{\mathbb{Q}_{T}}[f_{1}]\\
\E_{\widehat{\mathbb{Q}}_{T'}}[f_{T'}]&=\E_{\mathbb{Q}_{T'}}[f'_{1}]
\end{align*}

Now, it is easy to see that $T\cap T'$ is also a carrier and that we have
\[f_{1}(x) \leq f(x) \leq f_{2}(x),\quad x \in T\cap T'\]
and
\[\mathbb{Q}_{T \cap T'}(f_{1}\neq f_{2}) \leq \mathbb{Q}_{T}(f_{1}\neq f_{2})=0\]
so that we conclude from Proposition \ref{prop_E} that
\begin{align*}
 \E_{\widehat{\mathbb{Q}}_{T \cap T'}}[f]&=\E_{\mathbb{Q}_{T \cap T'}}[f_{1}]\\
&=\E_{\mathbb{Q}_{T}}[f_{1}]-\E_{\mathbb{Q}_{T \setminus T'}}[f_{1}]\\
&=\E_{\mathbb{Q}_{T}}[f_{1}]\\
&=\E_{\widehat{\mathbb{Q}}_{T}}[f]
\end{align*}
and so conclude that
\[ \E_{\widehat{\mathbb{Q}}_{T \cap T'}}[f]=\E_{\widehat{\mathbb{Q}}_{T}}[f].\]
By the same argument on $T'$ we conclude that
$ \E_{\widehat{\mathbb{Q}}_{T \cap T'}}[f]=\E_{\widehat{\mathbb{Q}}_{T'}}[f]$ and therefore
the first assertion is proved. The assertions of affinity and monotonicity are similarly straightforward.  \hfill\qedsymbol

\subsubsection{Proof of Proposition \ref{prop_change}}

Consider $\pi \in \mathcal{M}(X)$ and its pushforward $\nu:=\Psi\pi$.
 By Proposition \ref{prop_E} and the assumptions,
 there exists two measurable functions
$\underline{f}$ and $\overline{f}$ such that
\[
	\underline{f} \leq f \leq \overline{f},\quad  \nu(\underline{f}\neq \overline{f})=0
\]
from which we conclude that
\[
        \underline{f}\circ \Psi \leq   f\circ \Psi \leq  \overline{f}\circ \Psi
\]
and
\begin{align*}
	0
	&= \nu[\underline{f} \neq \overline{f}]\\
	&= \Psi\pi[\underline{f} \neq \overline{f}]\\
	&= \pi[\Psi^{-1}\{\underline{f} \neq \overline{f}\}]\\
	&= \pi[\underline{f}\circ \Psi \neq \overline{f}\circ\Psi]
\end{align*}
so that we obtain
\[\pi[\underline{f}\circ \Psi \neq \overline{f}\circ\Psi]=0\, .\]
Since $\pi$ was arbitrary, it follows that $f\circ\Psi$ is $\widehat{\mathcal{B}}(X)$-measurable.  To obtain the change of variables formula, compute
\begin{align*}
	\E_{\pi}[f\circ \Psi]
	&:= \E_{\widehat{\pi}}[f\circ \Psi]\\
	&\phantom{:}= \E_{\pi}[\overline{f}\circ \Psi]\\
	&\phantom{:}= \E_{\Psi\pi}[\overline{f}]
\end{align*}
and
\begin{align*}
	\E_{\Psi\pi}[f]
	&:= \E_{\widehat{\Psi\pi}}[f]\\
	&\phantom{:}= \E_{\Psi\pi}[\overline{f}]
\end{align*}
from which we conclude the change of variables formula
\[
	\E_{\pi}[f\circ \Psi] = \E_{\Psi\pi} [f],
\]
which completes the proof.  \hfill\qedsymbol

\subsubsection{Proof of Lemma \ref{lem_affine}}

Fix $\nu \in H$ and a probability measure $p$ such that the barycentric formula  \eqref{eq_barycenter} holds.  Proposition \ref{prop_E}  asserts that there are measurable functions $f_{1}\leq f\leq f_{2}$ such that $\nu(f_{1}\neq f_{2})=0$.  Therefore, $f_{2}-f\geq 0$, $\E_{\widehat{\nu}}(f_{2}-f)=0$, $f-f_{1}\geq 0$, and $\E_{\widehat{\nu}}(f-f_{1})=0$.  Moreover, it is easy to see then we can make both $f_{1}$ and $f_{2}$ semibounded.  Therefore $F$ is a well defined extended real valued function.  Moreover, \cite[Prop.~3.1]{Winkler:1988}  asserts that the function
$\nu \mapsto \E_{\nu}[f_{i}]$ is measure affine for $i=1,2$, and so
\[
	\E_{\nu}[f_{i}]=\int_{\ext(H)} \E_{\nu'}[f_{i}] \, \mathrm{d}p(\nu'), \quad \text{for } i=1,2 .
\]
Consequently, since $\nu[f_{1}\neq f_{2}]=0$, it follows that $\E_{\nu}[f_{2}-f_{1}]=0$ so that
\[
	0=\E_{\nu}[f_{2}-f_{1}]=\int_{\ext(H)} \E_{\nu'}[f_{2}-f_{1}] \, \mathrm{d}p(\nu'), \quad \text{for } i=1,2,
\]
and since $f_{2}-f_{1}\geq 0$ it follows that
\[
	\nu'[f_{2}\neq f_{1}]=0, \quad p\text{-a.e.}
\]
and therefore
\[
	\widehat{\nu'}[f\neq f_{1}]=0, \quad p\text{-a.e.}
\]
Therefore we conclude that
\begin{align*}
	F(\nu)
	&:= \E_{\widehat{\nu}}[f]\\
	&\phantom{:}= \E_{\widehat{\nu}}[f_{1}]+\E_{\widehat{\nu}}[f-f_{1}]\\
	&\phantom{:}= \E_{\widehat{\nu}}[f_{1}]\\
	&\phantom{:}= \E_{\nu}[f_{1}]\\
	&\phantom{:}= \int_{\ext(H)} \E_{\nu'}[f_{1}]\, \mathrm{d}p(\nu') \\
	&\phantom{:}= \int_{\ext(H)} \E_{\widehat{\nu'}}[f_{1}]\, \mathrm{d}p(\nu')\\
	&\phantom{:}= \int_{\ext(H)} \E_{\widehat{\nu'}}[f_{1}]\, \mathrm{d}p(\nu') + \int_{\ext(H)} \E_{\widehat{\nu'}}[f-f_{1}]\, \mathrm{d}p(\nu') \\
	&\phantom{:}= \int_{\ext(H)} \E_{\widehat{\nu'}}[f]\, \mathrm{d}p(\nu') \\
	&\phantom{:}= \int_{\ext(H)} F(\nu')\, \mathrm{d}p(\nu') ,
\end{align*}
and the assertion is proved.  \hfill\qedsymbol

\subsubsection{Proof of Lemma \ref{lem_ext}}

Let us first establish that
\begin{align}
	\label{compl_add}
	\widehat{\nu_{1}+\nu_{2}} &= \widehat{\nu_{1}}+\widehat{\nu_{2}}, \quad \text{for all }\nu_{1}, \nu_{2} \in \mathcal{M}(\mathcal{X}), \\
	\widehat{\alpha\nu} &= \alpha\widehat{\nu}, \quad \text{for all}\, \nu \in \mathcal{M}(\mathcal{X}) .
\end{align}
This follows from the fact that $(\nu_{1}+\nu_{2})(N)=0$ if and only if $\nu_{j}(N)=0$ for $j=1,2$ and the characterization of the completion $\widehat{\nu}$ by
\[
	\widehat{\nu}(B \cup S):=\nu(B), \quad  B \in \mathcal{B}(\mathcal{X}), \, S \subset N,\,  \nu(N)=0
\]
as found, for example, in \cite[p.~18]{Ash:1972}.  For then, for such $B$ and $S$,
\begin{align*}
	\widehat{\nu_{1}+\nu_{2}}(B \cup S)
	&= (\nu_{1}+\nu_{2})(B)\\
	&= \nu_{1}(B)+\nu_{2}(B)\\
	&= \widehat{\nu_{1}}(B\cup S)+\widehat{\nu_{2}}(B \cup S)
\end{align*}

Now for the proof of the main assertion.  Following the proof of \cite[Thm.~2.1]{Winkler:1988}, it is sufficient to  show that for
\[
	K:= \bigl\{\nu \in \mathcal{M}(\mathcal{X}) \,\big|\, f_{i} \text{ is $\widehat{\nu}$-integrable for } i=1,\dots,n \bigr\},
\]
we have
\begin{equation}
	\label{extK}
	\ext(K):=\{\d_{x}, x \in \mathcal{X}\},
\end{equation}
and that $\R_{+}K \subset \R_{+} \mathcal{M}(\mathcal{X}) $ is a lattice cone in its own ordering.  For the first, observe that since  $\ext\bigl(\mathcal{M}(\mathcal{X})\bigr)=\{\delta_{x} \mid x \in \mathcal{X}\}$ and that $f_{i}$ are $\delta_{x}$-integrable for all $i=1,\dots,n$, $x \in \mathcal{X}$, it follows that
\[
	\{\delta_{x} \mid x \in \mathcal{X}\} \subseteq \ext(K).
\]
Now suppose that $\nu \in \ext(K)$ is not a Dirac mass.  Then, as in the proof that the extreme points
of $\mathcal{M}(\mathcal{X}) $ are the Dirac masses, see e.g.~\cite[Thm.~15.9]{AliprantisBorder:2006}, and using the fact that the support of $\nu$ must contain $2$ or more points, we can decompose $\nu$ as a  convex combination
$\nu=\alpha\nu_{1}+(1-\alpha)\nu_{2}$ where $\nu_{1}\neq \nu_{2}$ and $\alpha \in (0,1)$.  Moreover, from
\[
	\widehat{\nu}=\alpha\widehat{\nu_{1}}+(1-\alpha)\widehat{\nu_{2}},
\]
we conclude that $f_{i}$ being $\widehat{\nu}$-integrable implies that $f_{i}$ is $\widehat{\nu_{j}}$-integrable for $j=1,2$ and $i=1,\dots,n$.  Consequently, $\nu_{j}\in K$ for $j=1,2$.  Since $\nu$ was an extreme point we conclude that $\nu_{1}=\nu_{2}$ which is a contradiction, and \eqref{extK} follows.

Now let us demonstrate that $\R_{+}K$ is a lattice cone in its own ordering.  To that end, note that by \cite[Lem.~10.4]{Phelps:2001}, it suffices to show that $\R_{+}K \subset \R_{+} \mathcal{M}(\mathcal{X}) $ is a hereditary subcone, in that $\nu_{1} \in \R_{+}K $, $ \nu_{2} \in \R_{+} \mathcal{M}(\mathcal{X}) $ and $\nu_{1}-\nu_{2} \in \R_{+}K$ together imply that $\nu_{2} \in \R_{+}K$.  To that end, consider such $\nu_{1}$ and $\nu_{2}$. Then \eqref{compl_add} implies that $\widehat{(\nu_{1}-\nu_{2})} = \widehat{\nu_{1}}-\widehat{\nu_{2}}$ and so we conclude that
\[
	0 \leq \E_{\widehat{(\nu_{1}-\nu_{2})}} [|f_{i}|] = \E_{\widehat{\nu_{1}}}[|f_{i}|] - \E_{\widehat{\nu_{2}}}[|f_{i}|]
\]
and therefore
\[
	\E_{\widehat{\nu_{2}}}[|f_{i}|] \leq \E_{\widehat{\nu_{1}}}[|f_{i}|] < \infty,
\]
from which we conclude that $\nu_{2} \in \R_{+}K$.  Hence, $\R_{+}K$ is a hereditary subcone, and the assertion then follows as in the proof of \cite[Thm.~2.1]{Winkler:1988}.  \hfill\qedsymbol

\section*{Acknowledgments}
The authors gratefully acknowledge  support for this work from the Air Force Office of Scientific Research under Award FA9550-12-1-0389 (Scientific Computation of Optimal Statistical Estimators).
We thank P.\ Diaconis, D.\ Mayo, P.\ Stark, and L.\ Wasserman for stimulating discussions and relevant references and pointers.
We thank the anonymous referees for detailed comments and suggestions.

\addcontentsline{toc}{section}{Acknowledgments}

\newpage

\addcontentsline{toc}{section}{References}
\bibliographystyle{plain}
\bibliography{./refs}

\def\polhk#1{\setbox0=\hbox{#1}{\ooalign{\hidewidth
  \lower1.5ex\hbox{`}\hidewidth\crcr\unhbox0}}} \def\cprime{$'$}
\begin{thebibliography}{100}

\bibitem{Abrahamcadre:2002}
C.~Abraham and B.~Cadre.
\newblock Asymptotic properties of posterior distributions derived from
  misspecified models.
\newblock {\em C. R. Math. Acad. Sci. Paris}, 335(5):495--498, 2002.

\bibitem{AbrahamCadre:2008}
C.~Abraham and B.~Cadre.
\newblock Concentration of posterior distributions with misspecified models.
\newblock {\em Ann. I.S.U.P.}, 52(3):3--14, 2008.

\bibitem{Akhiezer:1965}
N.~I. Akhiezer.
\newblock {\em The {C}lassical {M}oment {P}roblem and {S}ome {R}elated
  {Q}uestions in {A}nalysis}.
\newblock Translated by N. Kemmer. Hafner Publishing Co., New York, 1965.

\bibitem{AliprantisBorder:2006}
C.~D. Aliprantis and K.~C. Border.
\newblock {\em Infinite {D}imensional {A}nalysis: {A} {H}itchhiker's {G}uide}.
\newblock Springer, Berlin, third edition, 2006.

\bibitem{Arveson1976}
W.~Arveson.
\newblock {\em An Invitation to C{*}-Algebras}.
\newblock Springer-Verlag, New York, 1976.

\bibitem{Ash:1972}
R.~B. Ash.
\newblock {\em Real {A}nalysis and {P}robability}.
\newblock Academic Press, New York, 1972.
\newblock Probability and Mathematical Statistics, No. 11.

\bibitem{Aumann:1967}
R.~J. Aumann.
\newblock Measurable utility and the measurable choice theorem.
\newblock {\em La d\'{e}cision C.N.R.S.}, pages 15--26, 1967.

\bibitem{BahadurSavage:1956}
R.~R. Bahadur and L.~J. Savage.
\newblock The nonexistence of certain statistical procedures in nonparametric
  problems.
\newblock {\em Ann. Math. Statist.}, 27:1115--1122, 1956.

\bibitem{BarronEtAl:1999}
A.~Barron, M.~J. Schervish, and L.~Wasserman.
\newblock The consistency of posterior distributions in nonparametric problems.
\newblock {\em Ann. Statist.}, 27(2):536--561, 1999.

\bibitem{Bauer}
H.~Bauer.
\newblock {\em Measure and {I}ntegration {T}heory}, volume~26 of {\em de
  Gruyter Studies in Mathematics}.
\newblock Walter de Gruyter \& Co., Berlin, 2001.
\newblock Translated from the German by Robert B. Burckel.

\bibitem{bayarri2004}
M.~J. Bayarri and J.~O. Berger.
\newblock The interplay of bayesian and frequentist analysis.
\newblock {\em Statist. Sci.}, 19(1):58--80, 02 2004.

\bibitem{Belot:2013}
G.~Belot.
\newblock Bayesian orgulity.
\newblock {\em Philos. Sci.}, 80(4):483--503, 2013.

\bibitem{Belot:2013b}
G.~Belot.
\newblock Failure of calibration is typical.
\newblock {\em arXiv:1306.4943}, 2013.

\bibitem{Berger:2006}
J.~Berger.
\newblock The case for objective {B}ayesian analysis.
\newblock {\em Bayesian Anal.}, 1(3):385--402, 2006.

\bibitem{Berger:1984}
J.~O. Berger.
\newblock The robust {B}ayesian viewpoint.
\newblock In {\em Robustness of {B}ayesian {A}nalyses}, volume~4 of {\em Stud.
  Bayesian Econometrics}, pages 63--144. North-Holland, Amsterdam, 1984.
\newblock With comments and with a reply by the author.

\bibitem{Berger:1994}
J.~O. Berger.
\newblock An overview of robust {B}ayesian analysis.
\newblock {\em Test}, 3(1):5--124, 1994.
\newblock With comments and a rejoinder by the author.

\bibitem{Berk:1966}
R.~H. Berk.
\newblock Limiting behavior of posterior distributions when the model is
  incorrect.
\newblock {\em Ann. Math. Statist. 37 (1966), 51--58; correction, ibid},
  37:745--746, 1966.

\bibitem{Berk:1970}
R.~H. Berk.
\newblock Consistency a posteriori.
\newblock {\em Ann. Math. Statist.}, 41:894--906, 1970.

\bibitem{Bernstein:1964}
S.~N. Bern{\v{s}}te{\u\i}n.
\newblock {\em Sobranie sochinenii. {T}om {IV}: {T}eoriya veroyatnostei.
  {M}atematicheskaya statistika. 1911--1946}.
\newblock Izdat. ``Nauka'', Moscow, 1964.

\bibitem{BertsimasPopescu:2005}
D.~Bertsimas and I.~Popescu.
\newblock Optimal inequalities in probability theory: a convex optimization
  approach.
\newblock {\em SIAM J. Optim.}, 15(3):780--804 (electronic), 2005.

\bibitem{BleiJordanNg:2003}
D.~M. Blei, M.~I. Jordan, and A.~Y. Ng.
\newblock Hierarchical {B}ayesian models for applications in information
  retrieval.
\newblock In {\em Bayesian statistics, 7 ({T}enerife, 2002)}, pages 25--43.
  Oxford Univ. Press, New York, 2003.

\bibitem{Bogachev1}
V.~I. Bogachev.
\newblock {\em Measure {T}heory. {V}ol. {I}}.
\newblock Springer-Verlag, Berlin, 2007.

\bibitem{Bogachev2}
V.~I. Bogachev.
\newblock {\em Measure {T}heory. {V}ol. {II}}.
\newblock Springer-Verlag, Berlin, 2007.

\bibitem{Boole:1854}
G.~Boole.
\newblock {\em An {I}nvestigation of the {L}aws of {T}hought on {W}hich are
  {F}ounded the {M}athematical {T}heories of {L}ogic and {P}robabilities}.
\newblock Walton and Maberly, London, 1854.

\bibitem{Box:1953}
G.~E.~P. Box.
\newblock Non-normality and tests on variances.
\newblock {\em Biometrika}, 40:318--335, 1953.

\bibitem{Box:1987}
G.~E.~P. Box and N.~R. Draper.
\newblock {\em Empirical model-building and response surfaces}.
\newblock Wiley Series in Probability and Mathematical Statistics: Applied
  Probability and Statistics. John Wiley \& Sons Inc., New York, 1987.

\bibitem{BoydVandenberghe:2004}
S.~Boyd and L.~Vandenberghe.
\newblock {\em Convex {O}ptimization}.
\newblock Cambridge University Press, Cambridge, 2004.

\bibitem{BreimanEtAl:1964}
L.~Breiman, L.~Le~Cam, and L.~Schwartz.
\newblock Consistent estimates and zero-one sets.
\newblock {\em Ann. Math. Statist.}, 35:157--161, 1964.

\bibitem{CastaingValadier:1977}
C.~Castaing and M.~Valadier.
\newblock {\em Convex {A}nalysis and {M}easurable {M}ultifunctions}.
\newblock Lecture Notes in Mathematics, Vol. 580. Springer-Verlag, Berlin,
  1977.

\bibitem{CastilloNickl:2013}
I.~Castillo and R.~Nickl.
\newblock Nonparametric {B}ernstein--von {M}ises theorems in {G}aussian white
  noise.
\newblock {\em Ann. Statist.}, 41(4):1999--2028, 2013.

\bibitem{Clarke:2004}
B.~Clarke.
\newblock Comparing {B}ayes model averaging and stacking when model
  approximation error cannot be ignored.
\newblock {\em J. Mach. Learn. Res.}, 4(4):683--712, 2004.

\bibitem{Cox:1993}
D.~D. Cox.
\newblock An analysis of {B}ayesian inference for nonparametric regression.
\newblock {\em Ann. Statist.}, 21(2):903--923, 1993.

\bibitem{Daley}
D.~J. Daley and D.~Vere-Jones.
\newblock {\em An {I}ntroduction to the {T}heory of {P}oint {P}rocesses. {V}ol.
  {II}}.
\newblock Probability and its Applications (New York). Springer, New York,
  second edition, 2008.
\newblock General theory and structure.

\bibitem{DellacherieMeyer:1975}
C.~Dellacherie and P.-A. Meyer.
\newblock {\em Probabilit\'es et {P}otentiel}.
\newblock Hermann, Paris, 1975.
\newblock Chapitres I {\`a} IV, {\'E}dition enti{\`e}rement refondue,
  Publications de l'Institut de Math{\'e}matique de l'Universit{\'e} de
  Strasbourg, No. XV, Actualit{\'e}s Scientifiques et Industrielles, No. 1372.

\bibitem{DiaconisFreedman:1986}
P.~Diaconis and D.~Freedman.
\newblock On the consistency of {B}ayes estimates.
\newblock {\em Ann. Statist.}, 14(1):1--67, 1986.
\newblock With a discussion and a rejoinder by the authors.

\bibitem{DiaconisFreedman:1998}
P.~W. Diaconis and D.~Freedman.
\newblock Consistency of {B}ayes estimates for nonparametric regression: normal
  theory.
\newblock {\em Bernoulli}, 4(4):411--444, 1998.

\bibitem{Donoho:1988}
D.~L. Donoho.
\newblock One-sided inference about functionals of a density.
\newblock {\em Ann. Statist.}, 16(4):1390--1420, 1988.

\bibitem{Doob:1949}
J.~L. Doob.
\newblock Application of the theory of martingales.
\newblock In {\em Le {C}alcul des {P}robabilit\'es et ses {A}pplications},
  Colloques Internationaux du Centre National de la Recherche Scientifique, no.
  13, pages 23--27. Centre National de la Recherche Scientifique, Paris, 1949.

\bibitem{Doob:1994}
J.~L. Doob.
\newblock {\em Measure {T}heory}, volume 143 of {\em Graduate Texts in
  Mathematics}.
\newblock Springer-Verlag, New York, 1994.

\bibitem{Dudley:2002}
R.~M. Dudley.
\newblock {\em Real {A}nalysis and {P}robability}, volume~74 of {\em Cambridge
  Studies in Advanced Mathematics}.
\newblock Cambridge University Press, Cambridge, 2002.
\newblock Revised reprint of the 1989 original.

\bibitem{Edwards}
A.~W.~F. Edwards.
\newblock {\em Likelihood}.
\newblock Johns Hopkins University Press, Baltimore, expanded edition, 1992.

\bibitem{Edwards_bayesian}
W.~Edwards, H.~Lindman, and L.~J. Savage.
\newblock Bayesian statistical inference for psychological research.
\newblock {\em Psychological Review}, 70(3):193, 1963.

\bibitem{Efron:2013}
B.~Efron.
\newblock Bayes' theorem in the 21st century.
\newblock {\em Science}, 340(6137):1177--1178, 2013.

\bibitem{EWCA}
England and Wales~Court of~Appeal (Civil~Division).
\newblock Nulty \& {O}rs v.\ {M}ilton {K}eynes {B}orough {C}ouncil, 2013.
\newblock [2013] EWCA Civ 15, Case No. A1/2012/0459.
  \url{http://www.bailii.org/ew/cases/EWCA/Civ/2013/15.html}.

\bibitem{Feldman:1958}
J.~Feldman.
\newblock Equivalence and perpendicularity of {G}aussian processes.
\newblock {\em Pacific J. Math.}, 8:699--708, 1958.

\bibitem{ForresterWarnaar}
P.~J. Forrester and S.~O. Warnaar.
\newblock The importance of the {S}elberg integral.
\newblock {\em Bull Amer. Math. Soc.}, 45(4):489--534, 2008.

\bibitem{Freedman:1999}
D.~Freedman.
\newblock On the {B}ernstein-von {M}ises theorem with infinite-dimensional
  parameters.
\newblock {\em Ann. Statist.}, 27(4):1119--1140, 1999.

\bibitem{Freedman:1963}
D.~A. Freedman.
\newblock On the asymptotic behavior of {B}ayes' estimates in the discrete
  case.
\newblock {\em Ann. Math. Statist.}, 34:1386--1403, 1963.

\bibitem{Freedman:1965}
D.~A. Freedman.
\newblock On the asymptotic behavior of {B}ayes estimates in the discrete case.
  {II}.
\newblock {\em Ann. Math. Statist.}, 36:454--456, 1965.

\bibitem{Fushiki:2005}
T.~Fushiki.
\newblock Bootstrap prediction and {B}ayesian prediction under misspecified
  models.
\newblock {\em Bernoulli}, 11(4):747--758, 2005.

\bibitem{Gelman:2008a}
A.~Gelman.
\newblock Objections to {B}ayesian statistics.
\newblock {\em Bayesian Anal.}, 3(3):445--449, 2008.

\bibitem{Ghosal:2010}
S.~Ghosal.
\newblock The {D}irichlet process, related priors and posterior asymptotics.
\newblock In {\em Bayesian {N}onparametrics}, Camb. Ser. Stat. Probab. Math.,
  pages 35--79. Cambridge Univ. Press, Cambridge, 2010.

\bibitem{Grunwald:2006}
P.~D. Gr{\"u}nwald.
\newblock Bayesian inconsistency under misspecification, 2006.
\newblock \url{http://homepages.cwi.nl/~pdg/ftp/valenciapost.pdf}.

\bibitem{GustafsonWasserman:1995}
P.~Gustafson and L.~Wasserman.
\newblock Local sensitivity diagnostics for {B}ayesian inference.
\newblock {\em Ann. Statist.}, 23(6):2153--2167, 1995.

\bibitem{Guyon:2010}
I.~Guyon, A.~Saffari, G.~Dror, and G.~Cawley.
\newblock Model selection: Beyond the {B}ayesian/{F}requentist divide.
\newblock {\em J. Mach. Learn. Res.}, 11:61--87, March 2010.

\bibitem{Hajek:1958}
J.~H{\'a}jek.
\newblock On a property of normal distribution of any stochastic process.
\newblock {\em Czechoslovak Math. J.}, 8 (83):610--618, 1958.

\bibitem{HausmanTaylor:1981}
J.~A. Hausman and W.~E. Taylor.
\newblock A generalized specification test.
\newblock {\em Econom. Lett.}, 8(3):239--245, 1981.

\bibitem{Huber:1964}
P.~J. Huber.
\newblock Robust estimation of a location parameter.
\newblock {\em Ann. Math. Statist.}, 35:73--101, 1964.

\bibitem{Huber:1967}
P.~J. Huber.
\newblock The behavior of maximum likelihood estimates under nonstandard
  conditions.
\newblock In {\em Proc. {F}ifth {B}erkeley {S}ympos. {M}ath. {S}tatist. and
  {P}robability ({B}erkeley, {C}alif., 1965/66), {V}ol. {I}: {S}tatistics},
  pages 221--233. Univ. California Press, Berkeley, Calif., 1967.

\bibitem{HuberRonchetti:2009}
P.~J. Huber and E.~M. Ronchetti.
\newblock {\em Robust {S}tatistics}.
\newblock Wiley Series in Probability and Statistics. John Wiley \& Sons Inc.,
  Hoboken, NJ, second edition, 2009.

\bibitem{Johnstone:2010}
I.~M. Johnstone.
\newblock High dimensional {B}ernstein--von {M}ises: simple examples.
\newblock In {\em Borrowing strength: theory powering applications---a
  {F}estschrift for {L}awrence {D}. {B}rown}, volume~6 of {\em Inst. Math.
  Stat. Collect.}, pages 87--98. Inst. Math. Statist., Beachwood, OH, 2010.

\bibitem{Kallenberg:1975}
O.~Kallenberg.
\newblock {\em Random {M}easures}.
\newblock Akademie-Verlag, Berlin, 1975.
\newblock Schriftenreihe des Zentralinstituts f{\"u}r Mathematik und Mechanik
  bei der Akademie der Wissenschaften der DDR, Heft 23.

\bibitem{Kechris:1995}
A.~S. Kechris.
\newblock {\em Classical Descriptive Set Theory}.
\newblock Graduate Texts in Mathematics. Springer-Verlag, New York, 1995.

\bibitem{Kendall:1962}
D.~G. Kendall.
\newblock Simplexes and vector lattices.
\newblock {\em J. London Math. Soc.}, 37(1):365--371, 1962.

\bibitem{Keynes:1921}
J.~M. Keynes.
\newblock {\em A {T}reatise on {P}robability}.
\newblock Macmillan and Co., London, 1921.

\bibitem{KleijnVaart:2006}
B.~J.~K. Kleijn and A.~W. van~der Vaart.
\newblock Misspecification in infinite-dimensional {B}ayesian statistics.
\newblock {\em Ann. Statist.}, 34(2):837--877, 2006.

\bibitem{KleijnVaart:2012}
B.~J.~K. Kleijn and A.~W. van~der Vaart.
\newblock The {B}ernstein-{V}on-{Mises} theorem under misspecification.
\newblock {\em Electron. J. Stat.}, 6:354--381, 2012.

\bibitem{Kuznetsov:1991}
V.~P. Kuznetsov.
\newblock {\em Intervalnye {S}tatisticheskie {M}odeli [Interval Statistical
  Models]}.
\newblock ``Radio i Svyaz\cprime'', Moscow, 1991.

\bibitem{LeCam:1953}
L.~Le~Cam.
\newblock On some asymptotic properties of maximum likelihood estimates and
  related {B}ayes' estimates.
\newblock {\em Univ. California Publ. Statist.}, 1:277--329, 1953.

\bibitem{Lecam_necessary}
L.~Le~Cam and L.~Schwartz.
\newblock A necessary and sufficient condition for the existence of consistent
  estimates.
\newblock {\em The Annals of Mathematical Statistics}, pages 140--150, 1960.

\bibitem{Leahu:2011}
H.~Leahu.
\newblock On the {B}ernstein-von {M}ises phenomenon in the {G}aussian white
  noise model.
\newblock {\em Electron. J. Stat.}, 5:373--404, 2011.

\bibitem{Lindley:1985}
D.~V. Lindley.
\newblock {\em Making {D}ecisions}.
\newblock John Wiley \& Sons, Ltd., London, second edition, 1985.

\bibitem{Malakoff:1999}
D.~Malakoff.
\newblock Bayes offers a `new' way to make sense of numbers.
\newblock {\em Science}, 286(5444):1460--1464, 1999.

\bibitem{MartinHong:2012}
R.~Martin and L.~Hong.
\newblock On convergence rates of {B}ayesian predictive densities and posterior
  distributions.
\newblock {\em arXiv}, 1210.0103v1, 2012.

\bibitem{Mayo:2012}
D.~G. Mayo.
\newblock How can we cultivate {S}enn's ability?
\newblock {\em RMM}, 3:14--18, 2012.

\bibitem{Mayo:2012b}
D.~G. Mayo.
\newblock Statistical science and philosophy of science part 2: Shallow versus
  deep explorations.
\newblock {\em RMM}, 3(56), 2012.

\bibitem{MayoSpanos:2004}
D.~G. Mayo and A.~Spanos.
\newblock Methodology in practice: statistical misspecification testing.
\newblock {\em Philos. Sci.}, 71(5):1007--1025 (2005), 2004.

\bibitem{McGrayne:2012}
S.~B. McGrayne.
\newblock {\em The {T}heory {T}hat {W}ould {N}ot {D}ie: {H}ow {B}ayes' {R}ule
  {C}racked the {E}nigma {C}ode, {H}unted {D}own {R}ussian {S}ubmarines, and
  {E}merged {T}riumphant from {T}wo {C}enturies of {C}ontroversy}.
\newblock Yale University Press, 2012.

\bibitem{Nickl:2012}
R.~Nickl.
\newblock Statistical {T}heory, 2013.
\newblock \url{http://www.statslab.cam.ac.uk/~nickl/Site/__files/stat2013.pdf}.

\bibitem{OwhadiScovel:2013}
H.~Owhadi and C.~Scovel.
\newblock Brittleness of {B}ayesian inference and new {S}elberg formulas.
\newblock 2013.
\newblock Preprint at arXiv:1304.7046.

\bibitem{OwhadiScovel:2014}
H.~Owhadi and C.~Scovel.
\newblock Qualitative {R}obustness in {B}ayesian {I}nference.
\newblock 2014.
\newblock Preprint at arXiv:1411.3984.

\bibitem{OSSMO:2011}
H.~Owhadi, C.~Scovel, T.~J. Sullivan, M.~McKerns, and M.~Ortiz.
\newblock Optimal uncertainty quantification.
\newblock {\em SIAM Rev.}, 55(2):271--345, 2013.

\bibitem{Oxtoby:1971}
J.~C. Oxtoby.
\newblock {\em Measure and {C}ategory. {A} {S}urvey of the {A}nalogies
  {B}etween {T}opological and {M}easure {S}paces}.
\newblock Springer-Verlag, New York, 1971.
\newblock Graduate Texts in Mathematics, Vol. 2.

\bibitem{Phelps:2001}
R.~R. Phelps.
\newblock {\em Lectures on {C}hoquet's {T}heorem}, volume 1757 of {\em Lecture
  Notes in Mathematics}.
\newblock Springer-Verlag, Berlin, second edition, 2001.

\bibitem{SainteBeuve1974}
M.-F. Sainte-Beuve.
\newblock On the extension of von {N}eumann-{A}umann's theorem.
\newblock {\em J. Functional Analysis}, 17:112--129, 1974.

\bibitem{Schwartz:1965}
L.~Schwartz.
\newblock On {B}ayes procedures.
\newblock {\em Z. Wahrscheinlichkeitstheorie und Verw. Gebiete}, 4:10--26,
  1965.

\bibitem{Schwartz:1974}
L.~Schwartz.
\newblock {\em Radon {M}easures on {A}rbitrary {T}opological {S}paces and
  {C}ylindrical {M}easures}.
\newblock Oxford Univ. Press, Oxford, 1974.

\bibitem{Senn:2007}
S.~Senn.
\newblock Trying to be precise about vagueness.
\newblock {\em Statistics in Medecine}, 26:1417--–1430, 2007.

\bibitem{Senn:2011}
S.~Senn.
\newblock You may believe you are a {B}ayesian but you are probably wrong.
\newblock {\em RMM}, 2:48--66, 2011.

\bibitem{Smith:1995}
J.~E. Smith.
\newblock Generalized {C}hebychev inequalities: theory and applications in
  decision analysis.
\newblock {\em Oper. Res.}, 43(5):807--825, 1995.

\bibitem{Spanier}
E.~H. Spanier.
\newblock {\em Algebraic Topology}.
\newblock Springer-Verlag, New York, 1966.

\bibitem{Stuart:2010}
A.~M. Stuart.
\newblock Inverse problems: a {B}ayesian perspective.
\newblock {\em Acta Numer.}, 19:451--559, 2010.

\bibitem{Telgarsky:1987}
R.~V. Telg\'{a}rsky.
\newblock Topological games: on the 50th anniversary of the {B}anach--{M}azur
  game.
\newblock {\em Rocky Mountain J. Math.}, 17(2):227--276, 1987.

\bibitem{TibshiraniWasserman:1988}
R.~Tibshirani and L.~A. Wasserman.
\newblock Sensitive parameters.
\newblock {\em The Canadian Journal of Statistics}, 16(2):185--192, 1988.

\bibitem{Topsoe:1970}
F.~Tops{\o}e.
\newblock {\em Topology and {M}easure}.
\newblock Lecture Notes in Mathematics, Vol. 133. Springer-Verlag, Berlin,
  1970.

\bibitem{vonMises:1964}
R.~von Mises.
\newblock {\em Mathematical {T}heory of {P}robability and {S}tatistics}.
\newblock Edited and Complemented by Hilda Geiringer. Academic Press, New York,
  1964.

\bibitem{WeizsackerWinkler:1979}
H.~von Weizs{\"a}cker and G.~Winkler.
\newblock Integral representation in the set of solutions of a generalized
  moment problem.
\newblock {\em Math. Ann.}, 246(1):23--32, 1979/80.

\bibitem{Walker:2004}
S.~Walker.
\newblock New approaches to {B}ayesian consistency.
\newblock {\em Ann. Statist.}, 32(5):2028--2043, 2004.

\bibitem{WalkerHjort:2001}
S.~Walker and N.~L. Hjort.
\newblock On {B}ayesian consistency.
\newblock {\em J. R. Stat. Soc. Ser. B Stat. Methodol.}, 63(4):811--821, 2001.

\bibitem{Walley:1991}
P.~Walley.
\newblock {\em Statistical {R}easoning with {I}mprecise {P}robabilities},
  volume~42 of {\em Monographs on Statistics and Applied Probability}.
\newblock Chapman and Hall Ltd., London, 1991.

\bibitem{Wasserman:1998}
L.~Wasserman.
\newblock Asymptotic properties of nonparametric {B}ayesian procedures.
\newblock In {\em Practical nonparametric and semiparametric {B}ayesian
  statistics}, volume 133 of {\em Lecture Notes in Statist.}, pages 293--304.
  Springer, New York, 1998.

\bibitem{WassermanEtAl:1993}
L.~Wasserman, M.~Lavine, and R.~L. Wolpert.
\newblock Linearization of {B}ayesian robustness problems.
\newblock {\em J. Statist. Plann. Inference}, 37(3):307--316, 1993.

\bibitem{WassermanSeidenfeld:1994}
L.~Wasserman and T.~Seidenfeld.
\newblock The dilation phenomenon in robust {B}ayesian inference.
\newblock {\em J. Statist. Plann. Inference}, 40:345--356, 1994.

\bibitem{Wasserman:1990}
L.~A. Wasserman.
\newblock Prior envelopes based on belief functions.
\newblock {\em Ann. Statist.}, 18(1):454--464, 1990.

\bibitem{Weichselberger:2000}
K.~Weichselberger.
\newblock The theory of interval-probability as a unifying concept for
  uncertainty.
\newblock {\em Internat. J. Approx. Reason.}, 24(2-3):149--170, 2000.
\newblock Reasoning with imprecise probabilities (Ghent, 1999).

\bibitem{White:1982}
H.~White.
\newblock Maximum likelihood estimation of misspecified models.
\newblock {\em Econometrica}, 50(1):1--25, 1982.

\bibitem{Winkler:1988}
G.~Winkler.
\newblock Extreme points of moment sets.
\newblock {\em Math. Oper. Res.}, 13(4):581--587, 1988.

\bibitem{Zsilinszky:1998}
L.~Zsilinszky.
\newblock Topological games and hyperspace topologies.
\newblock {\em Set-Valued Anal.}, 6(2):187--207, 1998.

\end{thebibliography}

\end{document}